\documentclass{amsart}
\usepackage{latexsym,amsxtra,amscd,ifthen}
\usepackage{amsfonts,cleveref}
\usepackage{verbatim}
\usepackage{amsmath}
\usepackage{amsthm}
\usepackage{amssymb}
\usepackage{adjustbox}
\usepackage{url}
\usepackage[arrow,matrix]{xy}
\usepackage{blkarray, bigstrut}
\usepackage{xparse}

\theoremstyle{plain}

\newtheorem{theorem}{Theorem}
\newtheorem{lemma}[theorem]{Lemma}
\newtheorem{convention}[theorem]{Convention}
\newtheorem{proposition}[theorem]{Proposition}
\newtheorem{corollary}[theorem]{Corollary}

\numberwithin{theorem}{section}
\numberwithin{equation}{theorem}

\theoremstyle{definition}
\newtheorem{definition}[theorem]{Definition}

\newtheorem{example}[theorem]{Example}
\newtheorem{remark}[theorem]{Remark}
\newtheorem{question}[theorem]{Question}
\newtheorem{project}[theorem]{Project}
\newtheorem*{question*}{Question}

\newcommand{\Z}{\mathbb{Z}}

\DeclareMathOperator{\proj}{proj}

\DeclareMathOperator{\fpd}{fpd}

\DeclareMathOperator{\fpv}{fpv}

\DeclareMathOperator{\RHom}{RHom}

\DeclareMathOperator{\End}{End}
\DeclareMathOperator{\Ext}{Ext}
\DeclareMathOperator{\Tor}{Tor}
\DeclareMathOperator{\Hom}{Hom}

\DeclareMathOperator{\gr}{gr}
\DeclareMathOperator{\Repr}{rep}

\DeclareMathOperator{\Mod}{Mod}
\DeclareMathOperator{\Modfd}{mod}

\usepackage[normalem]{ulem}

\DeclareMathOperator{\rank}{rank}

\begin{document}
\title[Frobenius-Perron Theory of Quivers]
{Frobenius-Perron Theory of \\
Representations of Quivers}
\author{J.J. Zhang and J.-H. Zhou}

\address{J.J. Zhang: Department of Mathematics, Box 354350,
University of Washington, Seattle, Washington 98195, USA}

\email{zhang@math.washington.edu}

\address{J.-H. Zhou: School of Mathematics, 
Shanghai University of Finance and Economics, Shanghai 200433, 
Shanghai, China, $\quad$
Shanghai Center for Mathematical Sciences,
Fudan University, Shanghai 200433, Shanghai, China}

\email{zhoujingheng@mail.shufe.edu.cn}

\begin{abstract}
The Frobenius-Perron theory of an endofunctor of a category
was introduced in recent years \cite{CGWZZZ2017, CGWZZZ2019}.
We apply this theory to monoidal (or tensor) triangulated
structures of quiver representations.
\end{abstract}





\maketitle

\setcounter{section}{-1}
\section{Introduction}
\label{xxsec0}

Throughout let $\Bbbk$ be a base field that is algebraically
closed. Algebraic objects are defined over $\Bbbk$.

The Frobenius-Perron dimension of an object in a semisimple
finite tensor (or fusion) category was introduced by
Etingof-Nikshych-Ostrik in 2005 \cite{ENO2005}. Since then
it has become an extremely useful invariant in the study of
fusion categories and representations of semisimple (weak
and/or quasi-)Hopf algebras. By examining the Frobenius-Perron
dimension of all objects in a finite tensor category one can
determine whether the category is equivalent to the
representation category of a finite-dimensional quasi-Hopf
algebra \cite[Proposition 2.6]{EO2004}. The Frobenius-Perron
dimension of a fusion category is also a crucial invariant in
the classification of fusion categories as well as that of
semisimple Hopf algebras. An important project is to develop
the Frobenius-Perron theory for not-necessarily semisimple
tensor (or monoidal) categories. A step departing from
semisimple categories, or abelian categories of global
dimension 0, is to study the hereditary ones (of global
dimension one). Ultimately Frobenius-Perron theory should
provide powerful tools and useful invariants for projects like

\begin{project}
\label{xxpro0.1}
Describe and understand weak bialgebras
[Definition \ref{xxdef1.7}] and weak Hopf algebras
\cite[Definition 2.1]{BNS1999} that are hereditary
as associative algebras.
\end{project}

Note that an analogous classification project of hereditary
prime Hopf algebras was finished in a remarkable paper by
Wu-Liu-Ding \cite{WLD2016} a few years ago. Some recent
efforts pertaining on homological aspects of noetherian weak
Hopf algebras were presented in \cite{RWZ2020}. Recall from
\cite{NVY2019} that a {\it monoidal triangulated category}
is a monoidal category $\mathcal{T}$ in the sense of
\cite[Definition 2.2.1]{EGNO2015} that is triangulated and,
for which, the tensor product $\otimes_{\mathcal{T}}:
\mathcal{T} \times \mathcal{T} \to \mathcal{T}$ is an exact
bifunctor. Given a Hopf algebra (respectively, weak and/or
quasi- Hopf algebra/bialgebra) $H$, its comultiplication
induces a monoidal structure on the category of
representations of $H$. The corresponding derived category
has a canonical monoidal triangulated structure. Monoidal
triangulated structures appear naturally in several other
subjects.

{\it Algebraic geometry}.
A classical theorem of Gabriel \cite{Ga1962} states that a
noetherian scheme $\mathbb{X}$ can be reconstructed from the
abelian category of coherent sheaves over $\mathbb{X}$,
denoted by $coh(\mathbb{X})$. Hence the abelian category
$coh(\mathbb{X})$ captures all information about the space
$\mathbb{X}$. Recent development in derived algebraic geometry
suggests that the bounded derived category of coherent sheaves
over $\mathbb{X}$, denoted by $D^b(coh(\mathbb{X}))$, is
sometimes a better category to work with when we are
considering many geometric problems such as moduli problems.
When $\mathbb{X}$ is smooth, $D^b(coh(\mathbb{X}))$ is
equipped with a natural tensor (=symmetric monoidal)
triangulated structure in the sense of
\cite[Definition 1.1]{B2005}.

{\it Tensor triangulated geometry}.
Tensor triangulated categories have been studied by Balmer
\cite{B2005} and many others, where the study of tensor 
triangulated categories has been sometimes
referred to as {\it tensor triangulated geometry}. Balmer
defined the prime spectrum, denoted by ${\rm{Spc}}(\mathcal{T})$,
of a small tensor triangulated category $\mathcal{T}$ by using
the thick subcategories which behave like prime ideals under the
tensor product. Note that ${\rm{Spc}}(\mathcal{T})$ is a locally
ringed space \cite[Remark 6.4]{B2005}. This idea has been shown
to be widely applicable to algebraic geometry, homotopy theory
and representation theory. Recently Vashaw-Yakimov and
Nakano-Vashaw-Yakimov \cite{VY2018, NVY2019} developed a
noncommutative version of the Balmer spectrum, or {\it
noncommutative tensor triangulated geometry} (in the
words of the authors of \cite{NVY2019}).

{\it Noncommutative algebraic geometry}. Following
Grothendieck, {\it to do geometry you really don't need a space,
all you need is a category of sheaves on this would-be space}
\cite[p.78]{Ma2018}. Following \cite{VY2018, NVY2019},
we would like to consider
or view monoidal triangulated categories as appropriate categories
for doing a new kind of noncommutative geometry. For example, if
$T$ is a noetherian Koszul Artin-Schelter regular algebra, then
the bounded derived category of the noncommutative projective
scheme associated to $T$, denoted by ${\mathcal T}:=D^b(\proj T)$,
has at least two different monoidal triangulated structures
[Example \ref{xxex7.9}]. In this situation, it would
be very interesting to understand how the ``geometry'' of
$\proj T$ interacts with ``monoidal triangulated
structures'' on $\mathcal{T}$. Fix a general triangulated category,
still denoted by $\mathcal{T}$, it is common that there are many
different monoidal triangulated structures on $\mathcal{T}$
(with the same underlying triangulated structure) that reflect
on different hidden properties of $\mathcal{T}$. So it is
worth distinguishing different types of monoidal triangulated
structures on $\mathcal{T}$ and finding a definition of the
``size'' of these structures.

{\it Quiver representations}.
A related subject is the representation theory of quivers that
has become a popular topic since Gabriel's work in the 1970s
\cite{Ga1972, Ga197374, Ga1973}. For a given quiver, it is
naturally equipped with a monoidal structure in the category of
its representations, induced by the vertex-wise tensor product
of vector spaces \eqref{E2.1.1}. The monoidal structure of
quiver representations has been studied by Strassen \cite{St2000}
in relation with orbit-closure degeneration in 2000, and later
by Herschend \cite{He2005, He2008a, He2008b, He2009, He2010}
in the relation with the bialgebra structure on the path
algebra during 2005-2012. Herschend solved the Clebsch-Gordan
problem for quivers of type $\mathbb{A}_n$, $\mathbb{D}_n$
and $\mathbb{E}_{6,7,8}$ in \cite{He2008b, He2009}. As for
tame type, Herschend also gave solutions for type
$\mathbb{\tilde{A}}_n$ in \cite{He2005} and the quivers
with relations that correspond to string algebras in
\cite{He2010}. One of our basic objects in this paper is
the bounded derived category $D^b(A-\Modfd)$ for a
finite dimensional weak bialgebra $A$. (We usually consider
hereditary, but not semisimple, algebras.) Since $A$ is a
weak bialgebra, $A-\Modfd$ has an induced monoidal abelian
structure; and hence, $D^b(A-\Modfd)$ is a monoidal
triangulated category in the sense of \cite{NVY2019}. Note
that, even for a finite quiver $Q$, \cite[Proposition 4]{He2008a}
and \cite[Theorem 3.2]{HT2013} give different weak bialgebra
structures on $A:=\Bbbk Q$ which produce different monoidal
triangulated structures on $D^b(A-\Modfd)$.

{\it Connections between geometry, quiver representations,
and weak bialgebras}.
Going back to classical geometry, let $\mathbb{X}$ be a smooth 
projective scheme. If $\mathbb{X}$ is equipped with a full 
strongly exceptional sequence (also called strong full 
exceptional sequence by some authors) $\{\mathcal{E}_1,\cdots,
\mathcal{E}_n\}$ [Definition \ref{xxdef7.8}], then there
is a triangulated equivalence
\begin{equation}
\label{E0.1.1}\tag{E0.1.1}
D^b(coh(\mathbb{X}))\cong D^b(A-\Modfd)
\end{equation}
where $A$ is the finite dimensional algebra
$[\End_{D^b(coh(\mathbb{X}))}(\oplus_{i=1}^n \mathcal{E}_i)]^{op}$
(some details can be found at the end of Section \ref{xxsec7}).
Since $A$ is finite dimensional (and of finite global
dimension), it seems easier to study $A$ than to study
$\mathbb{X}$ in some aspects. Equivalence \eqref{E0.1.1} induces
a monoidal structure on $D^b(A-\Modfd)$ which usually does
not come from any weak bialgebra structures of $A$
[Example \ref{xxex7.9}]; or in some extreme cases, there
is no weak bialgebra structure on $A$ at all. For such an
algebra $A$, it is imperative to understand and even to
classify all possible monoidal triangulated structures of
$D^b(A-\Modfd)$ (though $A$ may not be a weak bialgebra).

Another well-known example of such a connection is from the 
study of weighted projective lines, introduced by Geigle-Lenzing 
\cite{GL1987} in 1985 (see Section \ref{xxsec6}). Since then 
weighted projective lines have been studied extensively by 
many researchers. Let $coh({\mathbb X})$ denote the category 
of coherent sheaves over a weighted projective line $\mathbb{X}$. 
When $\mathbb{X}$ is domestic, a version of \eqref{E0.1.1} holds 
and the representation type of $A$ (appeared in the right-hand 
side of \eqref{E0.1.1}) is tame, see Lemma \ref{xxlem6.1}(2). 
This is one of the key facts that we will use in this paper.

Recently, a new definition of Frobenius-Perron dimension was
introduced in \cite{CGWZZZ2017, CGWZZZ2019} where the authors
extended its original definition from an object in a semisimple
finite tensor category to an endofunctor of any $\Bbbk$-linear
category. (We refer to Definition \ref{xxdef1.3} for some relevant
definitions.) It turns out that new Frobenius-Perron invariants
are sensitive to monoidal structures; as a consequence, these
are crucial to distinguish different monoidal triangulated
structures. One general goal of this paper is to provide evidence
that the Frobenius-Perron invariants are effective to study
monoidal triangulated structures. Some basic properties and
interesting applications of Frobenius-Perron-type invariants
can be found in \cite{CGWZZZ2017, CGWZZZ2019}.

In this paper, we focus on different weak bialgebra structures
on the path algebras of finite quivers and the Frobenius-Perron
theory for finite dimensional hereditary weak bialgebras. As
mentioned above, this is one step beyond the semisimple case.

\begin{definition}
\label{xxdef0.2}
Let ${\mathcal T}$ be a monoidal category and let $\mathcal{P}$
be a function
$$\{{\text{endofunctors of $\mathcal{T}$}}\}
\longrightarrow  {\mathbb{R}}_{\geq 0}.$$
Note that $\mathcal{P}$ could be Frobenius-Perron dimension or 
Frobenius-Perron curvature
as given in Definition \ref{xxdef1.3}(6,7).  For every object
$M$ in ${\mathcal T}$, let $\mathcal{P}(M)$ denote the
$\mathcal{P}$-value of the tensor functor $M\otimes_{\mathcal{T}} -:
\mathcal{T}\to \mathcal{T}$.
\begin{enumerate}
\item[(1)]
We say $\mathcal{T}$ is {\it $\mathcal{P}$-finite} if
$\mathcal{P}(M)<\infty$ for all objects $M$. Otherwise,
$\mathcal{T}$ is called {\it $\mathcal{P}$-infinite}.
\item[(2)]
If $\mathcal{T}$ is $\mathcal{P}$-infinite and if
$\mathcal{P}(M)<\infty$ for all indecomposable objects $M$,
then $\mathcal{T}$ is called {\it $\mathcal{P}$-tame}.
\item[(3)]
If $\mathcal{T}$ is neither $\mathcal{P}$-finite nor
$\mathcal{P}$-tame, then it is called {\it $\mathcal{P}$-wild}.
\end{enumerate}
\end{definition}

Our first main result concerns the trichotomy of $\fpd$-finite/tame/wild
property. Let $\Repr(Q)$ be the category of finite dimensional
representations of a quiver $Q$.

\begin{theorem}
\label{xxthm0.3}
Let $Q$ be a finite acyclic quiver and let $\mathcal{T}$
be the triangulated category $D^b(\Repr(Q))$.
\begin{enumerate}
\item[(1)]
$Q$ is of finite type if and only if $\mathcal{T}$ is
$\fpd$-finite for every monoidal triangulated structure
on $\mathcal{T}$, and if and only if there is one monoidal
triangulated structure on $\mathcal{T}$ such that $\mathcal{T}$
is $\fpd$-finite.
\item[(2)]
$Q$ is of tame type if and only if there is a monoidal
triangulated structure on $\mathcal{T}$ such that $\mathcal{T}$
is $\fpd$-tame. In this case, there must be another monoidal
triangulated structure on $\mathcal{T}$ such that $\mathcal{T}$
is $\fpd$-wild.
\item[(3)]
$Q$ is of wild type if and only if $\mathcal{T}$ is
$\fpd$-wild for every monoidal triangulated structure
on $\mathcal{T}$.
\item[(4)]
If $Q$ is tame or wild, then every monoidal triangulated
structure on $\mathcal{T}$ is $\fpd$-infinite.
\end{enumerate}
\end{theorem}

Note that in part (2) of the above theorem, there are two
different monoidal triangulated structures on $\mathcal{T}$,
one of which is $\fpd$-tame and the other is not. We refer
to Definition \ref{xxdef3.1} for the definition of a discrete
monoidal structure. By the above theorem, it is
rare to have $\fpd$-finite monoidal triangulated structures
on $\mathcal{T}$. When it exists, we can say a bit more. The
canonical weak bialgebra structure on the path algebra
$\Bbbk Q$ is given in Lemma \ref{xxlem2.1}(1).

\begin{theorem}
\label{xxthm0.4}
Let $A$ be a finite dimensional hereditary weak bialgebra
such that the induced monoidal structure on $A-\Modfd$
is discrete. Then the following are equivalent:
\begin{enumerate}
\item[(a)]
$A$ is of finite representation type,
\item[(b)]
$\fpd(M)<\infty$ for every irreducible representation $M\in A-\Modfd$,
\item[(c)]
$\fpd(M)<\infty$ for every indecomposable representation $M\in A-\Modfd$,
\item[(d)]
$\fpd(M)<\infty$ for every representation $M\in A-\Modfd$,
\item[(e)]
$\fpd (X)<\infty$ for every indecomposable object $X\in D^b(A-\Modfd)$,
\item[(f)]
The induced monoidal triangulated structure on $D^b(A-\Modfd)$ is
$\fpd$-finite.
\end{enumerate}
\end{theorem}

Suppose further that $A$ is the path algebra $\Bbbk Q$ with canonical
weak bialgebra structure. It follows from Gabriel's theorem that 
any of conditions {\rm{(a)}} to {\rm{(e)}} is equivalent to
\begin{enumerate}
\item[(g)]
$Q$ is a finite union of quivers of type $\mathbb{ADE}$.
\end{enumerate}

Since condition (a) in the above theorem is an algebra property, the
$\fpd$-finiteness of $D^b(A-\Modfd)$ only depends on the algebra
structure of $A$, though the definition of $\fpd(X)$ uses the
coalgebra structure of $A$. Note that condition (a) is not equivalent
to condition (b) if we remove the hereditary hypothesis in the above
theorem [Remark \ref{xxrem7.3}(3)].

Following BGP-reflection functors \cite{BGP1973}, Happel showed
that, for Dynkin quivers with the same underlying Dynkin diagram,
their derived categories are triangulated equivalent \cite{Ha1987}.
This remarkable theorem is one of most beautiful results in
representation theory of finite dimensional algebras. In contrast,
the story is very different when we are working with {\it monoidal
triangulated} structures of the derived category of Dynkin quivers,
see Theorem \ref{xxthm0.5} below. As indicated in
\cite{CGWZZZ2017, CGWZZZ2019}, Frobenius-Perron-type
invariants are extremely useful to study derived
(or triangulated) categories. Using the Frobenius-Perron
curvature, denoted by $\fpv$, of objects in $D^b(A-\Modfd)$
we can prove the following.

\begin{theorem}
\label{xxthm0.5}
Let $A$ and $B$ be finite dimensional hereditary weak
bialgebras. Assume either $A$ is a bialgebra or
$A-\Modfd$ is discrete. Suppose that the monoidal
triangulated categories $D^b(A-\Modfd)$ and
$D^b(B-\Modfd)$ are equivalent. Then $A-\Modfd$ and
$B-\Modfd$ are equivalent as monoidal abelian categories.
\end{theorem}

As a consequence, we have

\begin{corollary}
\label{xxcor0.6}
Suppose that the bounded derived categories of representations of
two finite acyclic quivers are equivalent as monoidal triangulated
categories. Then the quivers are isomorphic.
\end{corollary}

There is also a result concerning an analogue of a $t$-structure
in the monoidal triangulated setting, see Theorem \ref{xxthm0.7}
below. We introduce the notion of an $mtt$-structure on a monoidal
triangulated category in Section 5. Undefined terminologies
can be found in Sections 4 and 5.

\begin{theorem}
\label{xxthm0.7}
Let $A$ be a finite dimensional weak bialgebra that is hereditary
as an algebra. Suppose that the induced monoidal structure on
$A-\Modfd$ is discrete. Then there is a unique hereditary
$mtt$-structure with deviation zero on the monoidal triangulated
category $D^b(A-\Modfd)$.
\end{theorem}

It is clear that Theorem \ref{xxthm0.7} applies to $D^b(\Repr(Q))$
where $Q$ is a finite acyclic quiver. A $t$-structure on a
triangulated category has been studied extensively since it was
introduced by Beilinson-Bernstein-Deligne in \cite{BBD1981}. It
is natural to study all $mtt$-structures of a monoidal
triangulated category. In fact, $mtt$-structures service
as a compelling system of a monoidal triangulated category. It
is well-known that (hereditary) $t$-structures on $D^b(\Repr(Q))$
are not unique even for quivers of type ${\mathbb A}_n$, \
defined below, for $n\geq 3$. Therefore it is surprising that
certain $mtt$-structures (see Theorem \ref{xxthm0.7}) are unique.
This uniqueness result would have other significant consequences 
than Theorem \ref{xxthm0.5} and Corollary \ref{xxcor0.6}. It is also 
interesting to search for other classes of monoidal triangulated 
categories such that the uniqueness property holds for certain 
$mtt$-structures.

Though there are more than one tensor structures on the path
algebra $\Bbbk Q$ for a quiver $Q$, one of these structures
is from the natural coalgebra structure on $\Bbbk Q$,
similar to group algebras [Lemma \ref{xxlem2.1}(1)].
We will present more results concerning the Frobenius-Perron
dimensions of indecomposable representations under such
tensor structure. Before that we need to introduce some
notation. By definition, a type $\mathbb{A}$ quiver (or more
precisely, type $\mathbb{A}_n$ quiver) is a quiver of the
following form
\begin{equation}
\label{E0.7.1}\tag{E0.7.1}
\xymatrix{
1 \ar@{-}[r]^{\alpha_1}&2\ar@{-}[r]^{\alpha_2}
&\cdots\ar@{-}[r]^{\alpha_{i-1}}&i\ar@{-}[r]^{\alpha_i}
&\cdots\ar@{-}[r]^{\alpha_{n-1}}&n}
\end{equation}
where each arrow $\alpha_i$ is either $\longrightarrow$ or
$\longleftarrow$. For each quiver of type $\mathbb{A}_n$,
the arrows $\alpha_i$ will be specified. It is easy to see
that, for each $n\geq 3$, there are more than one
quivers of type $\mathbb{A}_n$ up to isomorphisms.
Let us fix a quiver of type $\mathbb{A}_n$, say $Q$,
as above. For $1\leq i\leq j\leq n$, we define a thin representation
of $Q$, denoted by $M\{i,j\}$, by
\begin{equation}
\label{E0.7.2}\tag{E0.7.2}
(M\{i,j\})_s=\begin{cases} \Bbbk & i\leq s\leq j,\\
0 & {\text{otherwise}}\end{cases}
\end{equation}
and
\begin{equation}
\label{E0.7.3}\tag{E0.7.3}
(M\{i,j\})_{\alpha_s}=\begin{cases} Id_{\Bbbk} & i\leq s<j,\\
0& {\text{otherwise}}.\end{cases}
\end{equation}
(This thin representation is sometimes called an interval module 
by other researchers.) Then by \cite[p.63]{GR1992},
all such $M\{i,j\}$ form the complete list
of indecomposable representations of $Q$ [Lemma \ref{xxlem1.10}].
For all $i\leq j$, we say
$$M\{i,j\}\;\; {\text{is}} \;\;
\begin{cases}
{\text{a $sink$ $\qquad$ if $\alpha_{i-1}=\longrightarrow$ (or $i=1$) and
$\alpha_{j}=\longleftarrow$ (or $j=n$)}},\\
{\text{a $source\;$ $\quad$ if $\alpha_{i-1}=\longleftarrow$ (or $i=1$) and
$\alpha_{j}=\longrightarrow$  (or $j=n$)}},\\
{\text{a $flow$ $\qquad$ if $\alpha_{i-1}=\alpha_{j}$, and it is either
$\longrightarrow$ or $\longleftarrow$}}.
\end{cases}$$
Since $\Repr(Q)$ is hereditary, every indecomposable
object in the bounded derived category $D^b(\Repr(Q))$
is of the form $M\{i,j\}[m]$ for some $m\in {\mathbb Z}$.
We have the following result for type $\mathbb{A}_n$.
Some computation in the case of type $\mathbb{D}_n$ quivers
is given in \cite{ZWD2020}.

\begin{theorem}
\label{xxthm0.8}
Let $Q$ be a quiver of type $\mathbb{A}_n$ for some
positive integer $n$. Then the following
hold in the bounded derived
category $D^b(\Repr(Q))$ with tensor defined as in \eqref{E2.1.1}.
\begin{enumerate}
\item[(1)]
$\fpd (M\{i,j\}[m])=0$ for all $m<0$ and $m>1$.
\item[(2)]
$\fpd(M\{i,j\}[0])=\begin{cases}
1 & {\text{if $M\{i,j\}$ is a sink}},\\
\min\{i, n-j+1\} & {\text{if $M\{i,j\}$ is a source}},\\
1 &{\text{if $M\{i,j\}$ is a flow}}.
\end{cases}$
\item[(3)]
$\fpd(M\{i,j\}[1])=\begin{cases}
\min\{i-1,n-j\} & {\text{if $M\{i,j\}$ is a sink}},\\
0 & {\text{if $M\{i,j\}$ is a source}},\\
0 &{\text{if $M\{i,j\}$ is a flow}}.
\end{cases}$
\end{enumerate}
\end{theorem}

Related to Project \ref{xxpro0.1} we are also very much interested
in the following questions.

\begin{question}
\label{xxque0.9}
Let $A$ be a finite dimensional weak bialgebra or just an algebra,
or $\Bbbk Q$ where $Q$ is a finite acyclic quiver.
\begin{enumerate}
\item[(1)]
How to determine all monoidal abelian structures on the abelian
category $A-\Modfd$?
\item[(2)]
How to determine all monoidal triangulated structures on the derived
category $D^b(A-\Modfd)$?
\end{enumerate}
\end{question}

The paper is organized as follows. Section 1 contains some basic
definitions. In particular, we recall the definition of the
Frobenius-Perron dimension of an endofunctor. In Section 2 we review
some preliminaries on quiver representations. The notion
of a discrete monoidal abelian category is introduced in
Section 3. A natural example of a discrete monoidal structure
is $\Repr(Q)$ which is the main object in this paper. Theorem
\ref{xxthm0.4} is proved in Section 4. In Section 5, we
introduce the notion of an $mtt$-structure of a monoidal
triangulated category that is a monoidal version of the
$t$-structure of a triangulated category. Theorems \ref{xxthm0.5}
and \ref{xxthm0.7}, and Corollary \ref{xxcor0.6} are
proved near the end of Section 5. Section 6 focuses on
the proof of Theorem \ref{xxthm0.3} which uses some
detailed information about weighted projective lines. Section
7 contains various examples which indicate the richness
of monoidal triangulated structures from different
subjects. Section 8 consists of the proof of Theorem
\ref{xxthm0.8} with some non-essential details left out.

\section{Some basic definitions}
\label{xxsec1}

This section contains several basic definitions which will be used
in later sections.

Recall from \cite[Definition 2.1.1]{EGNO2015} that
a {\it monoidal category} ${\mathcal C}$ consists of the following data:
\begin{enumerate}
\item[($\bullet$)]
a category ${\mathcal C}$,
\item[($\bullet$)]
a bifunctor $\otimes: {\mathcal C}\times {\mathcal C}\to
{\mathcal C}$, called {\it tensor functor},
\item[($\bullet$)]
for each triple $(X,Y, Z)$ in ${\mathcal C}$, a natural isomorphism
$$\alpha_{X,Y,Z}: (X\otimes Y)\otimes Z\xrightarrow{\cong} X\otimes(Y\otimes Z),$$
\item[($\bullet$)]
an object ${\bf 1}\in {\mathcal C}$, called {\it unit object},
\item[($\bullet$)]
natural isomorphisms $l_X: {\bf 1}\otimes X \xrightarrow{\cong} X$
and $r_X: X\otimes {\bf} \xrightarrow{\cong} X$ for each $X$ in ${\mathcal C}$,
\end{enumerate}
such that the pentagon axiom \cite[(2.2)]{EGNO2015} and the triangle axiom
\cite[(2.10)]{EGNO2015} hold. The definitions of a braiding
$\{c_{X,Y}\}_{X,Y\in {\mathcal C}}$ on a monoidal category ${\mathcal C}$
and of a braided monoidal category are given in
\cite[Definitions 8.1.1 and 8.1.2]{EGNO2015} respectively.

By \cite[Definition 8.1.12]{EGNO2015}, a braided monoidal category
${\mathcal C}$ is called {\it symmetric}  if
$$c_{Y,X}\circ c_{X,Y} = id_{X\otimes Y}$$
for all objects $X, Y\in {\mathcal  C}$.

We are usually considering $\Bbbk$-linear categories. Now we
recall some definitions.

\begin{definition}
\label{xxdef1.1}
Let ${\mathcal C}$ be a monoidal category.
\begin{enumerate}
\item[(1)]
We say ${\mathcal C}$ is {\it monoidal $\Bbbk$-linear}
if
\begin{enumerate}
\item[(1a)]
${\mathcal C}$ is $\Bbbk$-linear,
\item[(1b)]
morphisms and functors involving in the definition of
monoidal category are all $\Bbbk$-linear, and
\item[(1c)]
the tensor functor preserves direct sums in each argument.
\end{enumerate}
\item[(2)]
We say ${\mathcal C}$ is {\it monoidal abelian}
if
\begin{enumerate}
\item[(2a)]
${\mathcal C}$ is a $\Bbbk$-linear abelian category,
\item[(2b)]
${\mathcal C}$ is monoidal $\Bbbk$-linear in the sense of
part (1),
\item[(2c)]
the tensor functor preserves exact sequences in each argument.
\end{enumerate}
\item[(3)] \cite{NVY2019}
We say ${\mathcal C}$ is {\it monoidal triangulated}
if
\begin{enumerate}
\item[(3a)]
${\mathcal C}$ is $\Bbbk$-linear triangulated category,
\item[(3b)]
${\mathcal C}$ is monoidal $\Bbbk$-linear in the sense of
part (1),
\item[(3c)]
the tensor functor preserves exact triangles and commutes
with the suspension in each argument.
\item[(3d)]
the suspension satisfies the anti-commuting diagram given
at the end of the definition of a {\it suspended monoidal}
category \cite[Definition 1.4]{Su-Al}.
\end{enumerate}
\end{enumerate}
\end{definition}

By the way, we will not be using the axiom (3d)
in the above definition in this paper.
A tensor triangulated category in the sense of \cite[Definition 1.1]{B2005}
is just a symmetric monoidal triangulated category. We refer the
reader to \cite{EGNO2015} for other details.

Let $({\mathcal C}, \otimes, {\bf 1})$ be a monoidal category
and ${\mathcal A}$ be another category. Following \cite[p.62]{JK2001},
by an {\it action}  of ${\mathcal C}$ on ${\mathcal A}$ we mean a
strong monoidal functor
$$F = (f, \tilde{f}, f^{\circ}): {\mathcal C}\longrightarrow
[{\mathcal A}, {\mathcal A}],$$
where $[{\mathcal A}, {\mathcal A}]$ is the category of
endofunctors of ${\mathcal A}$, provided with a monoidal structure
$([{\mathcal A}, {\mathcal A}], \circ, Id_{\mathcal A})$ which is strict,
wherein $\circ$ denotes composition and $Id_{\mathcal A}$ is the identity endofunctor.
Here, to give the functor $f : {\mathcal C}\to [{\mathcal A}, {\mathcal A}]$
is equally to give a functor $\odot: {\mathcal C}\times {\mathcal A}
\to {\mathcal A}$ where $X\odot A = (fX)A$ for all $X\in {\mathcal C}$
and $A\in {\mathcal A}$; to give
the invertible and natural $\tilde{f}_{X,Y} : (fX) \circ (fY )
\to  f(X \otimes Y )$ (or rather their inverses) is
equally to give a natural isomorphism with components
$$\alpha_{X,Y,A}: (X\otimes Y )\odot A\to X\odot (Y \odot A);$$
to give the invertible $f^{\circ}: Id_{\mathcal A}\to  f {\bf 1}$
(or rather its inverse) is equally to give a natural
isomorphism with components $\lambda_{A} : {\bf 1} \odot A\to  A$;
and the coherence conditions for $F$ become the commutativity
of the three diagrams \cite[(1.1), (1.2) and (1.3)]{JK2001}
which are the pentagon axiom involving the associator of
${\mathcal C}$ and the triangle axioms for the action of the
unit object ${\bf 1}$ on ${\mathcal A}$ compatible with the left
unitor of ${\mathcal C}$ respectively. It is clear that a monoidal
category ${\mathcal C}$ acts on itself by defining $\odot=\otimes$.
We refer to \cite{JK2001} for more details.

\begin{convention}
\label{xxcon1.2}
Let ${\mathcal C}$ be a monoidal category acting on another
category ${\mathcal A}$.
\begin{enumerate}
\item[(1)]
If both ${\mathcal C}$ and ${\mathcal A}$ are {\it $\Bbbk$-linear},
we automatically assume that
\begin{enumerate}
\item[(1a)]
morphisms and functors involving in the definition of
action are all $\Bbbk$-linear, and
\item[(1b)]
$\odot$ preserves direct sums in each argument.
\end{enumerate}
\item[(2)]
If both ${\mathcal C}$ and ${\mathcal A}$ are {\it abelian},
we automatically assume that
\begin{enumerate}
\item[(2a)]
morphisms and functors involving in the definition of
action are all $\Bbbk$-linear,
\item[(2b)]
${\mathcal C}$ is monoidal abelian in the sense of
Definition \ref{xxdef1.1}(2),
\item[(2c)]
$\odot$ preserves exact sequences in each argument.
\end{enumerate}
\item[(3)]
If both ${\mathcal C}$ and ${\mathcal A}$ are {\it triangulated},
we automatically assume that
\begin{enumerate}
\item[(3a)]
morphisms and functors involving in the definition of
action are all $\Bbbk$-linear,
\item[(3b)]
${\mathcal C}$ is monoidal triangulated in the sense of
Definition \ref{xxdef1.1}(3),
\item[(3c)]
$\odot$ preserves exact triangles and commutes
with the suspension in each argument.
\end{enumerate}
\end{enumerate}
\end{convention}

Next we recall some definitions concerning the Frobenius-Perron
dimension of an endofunctor. We refer to \cite{CGWZZZ2017} for
other related definitions. Let $\dim$ be $\dim_{\Bbbk}$.

\begin{definition} \cite{CGWZZZ2017}
\label{xxdef1.3}
Let $\mathcal{C}$ be a $\Bbbk $-linear category.
\begin{enumerate}
\item[(1)]
An object $X$ in $\mathcal{C}$ is called a {\it brick} if
$$\Hom_{\mathcal{C}}(X,X)=\Bbbk.$$
\item[(2)]
Let $\phi:=\{X_1,\ldots, X_n\}$ be a finite subset of nonzero objects
in $\mathcal{C}$. We say that $\phi$ is a {\it brick set} if each
$X_i\in \phi$ is a brick and
$$\Hom_{\mathcal{C}}(X_i,X_j)=0, \forall \; i\neq j.$$
\item[(3)]
Let $\phi:=\{X_1,\ldots, X_n\}$ and let $\sigma$ be an endofunctor
of $\mathcal{C}$. The {\it adjacency matrix} of $(\phi,\sigma)$
is defined to be
$$A(\phi,\sigma)=(a_{ij})_{n\times n},
\quad {\text{where}} \quad
a_{ij}=\dim \Hom_{\mathcal{C}}(X_i, \sigma(X_j))\;\;
\forall \; i,j.$$
\item[(4)]
Let $\Phi_b$ be the collection of all finite brick sets in
$\mathcal{C}$. The {\it Frobenius-Perron dimension} of an
endofunctor $\sigma$ is defined to be
$$\fpd(\sigma)
:= \sup\limits_{\phi\in \Phi_b}\{\rho(A(\phi,\sigma))\}$$
where $\rho(A)$ is the spectral radius of a square matrix $A$
\cite[Section 1]{CGWZZZ2017}, i.e. the largest absolute 
value of $A$.
\item[(5)]
The {\it Frobenius-Perron curvature} of $\sigma$ is defined to be
$$\fpv (\sigma):=\sup_{\phi\in \Phi_{b}} \{\limsup_{n\to\infty} \;
(\rho(A(\phi,\sigma^n)))^{1/n} \}.$$
\item[(6)]
If $\mathcal{C}$ is a monoidal $\Bbbk$-linear category
acting on a $\Bbbk$-linear category ${\mathcal A}$
and $M$ is an object in
$\mathcal{C}$, the {\it Frobenius-Perron dimension} of $M$
is defined to be
$$\fpd(M):=\fpd(M\odot -)$$
where $M\odot -$ is considered as an endofunctor of
${\mathcal A}$ and $\fpd(M\odot -)$ is defined in part (4).
Similarly, the {\it Frobenius-Perron curvature} of $M
\in {\mathcal C}$ is
defined to be
$$\fpv(M):=\fpv(M\odot -)$$
where $M\odot -$ is considered as an endofunctor of
${\mathcal A}$ and $\fpd(M\odot -)$ is defined in part (5).
\item[(7)]
As a special case of (6),
if $\mathcal{C}$ is a monoidal $\Bbbk$-linear category
and $M$ is an object in
$\mathcal{C}$, the {\it Frobenius-Perron dimension} of $M$ is
defined to be
$$\fpd(M):=\fpd(M\otimes -)$$
where $\fpd(M\otimes -)$ is defined in part (4).
Similarly, the {\it Frobenius-Perron curvature} of $M$ is
defined to be
$$\fpv(M):=\fpv(M\otimes -)$$
where $\fpv(M\otimes -)$ is defined in part (5).
\end{enumerate}
\end{definition}

When $\mathcal{C}$ is $R-\Modfd$ for an algebra $R$, a brick set 
is also called a semibrick \cite{As2020}. If both ``full'' and 
``exceptional'' conditions [Definition \ref{xxdef7.8}(2,3)] are 
satisfied, this is also known as a simple-minded collection, see 
\cite[Definition 3.2]{KY2014}.

Now we recall the definition of representation types.

\begin{definition}
\label{xxdef1.4}
Let $A$ be a finite dimensional algebra over $\Bbbk$.
\begin{enumerate}
\item[(1)]
We say $A$ is of {\it finite type} or {\it finite representation
type} if there are only finitely many isomorphism classes of
finite dimensional indecomposable left $A$-modules.
\item[(2)]
We say $A$ is {\it tame} or {\it of tame representation type}
if it is not of finite representation type, and
for every $n\in {\mathbb N}$, all but finitely many
isomorphism classes of $n$-dimensional indecomposables occur in
a finite number of one-parameter families.
\item[(3)]
We say $A$ is {\it wild} or {\it of wild representation type}
if, for every finite dimensional $\Bbbk$-algebra $B$, the
representation theory of $B$ can be representation embedded 
into that of $A$.
\end{enumerate}
\end{definition}

We always assume that the base field $\Bbbk$ is algebraically closed.
A famous trichotomy result due to Drozd \cite{D1979} states
that every finite dimensional algebra is either of finite,
tame, or wild representation type. By classical theorems
of Gabriel \cite{Ga1972} and Nazarova \cite{N1973}, the quivers
of finite and tame representation types correspond to the
$\mathbb{ADE}$ and $\widetilde{\mathbb{A}}\widetilde{\mathbb{D}}
\widetilde{\mathbb{E}}$ diagrams respectively.
By \cite[Theorem 0.3]{CGWZZZ2017}, the representation type
of a quiver $Q$ is indicated by the value of the Frobenius-Perron
dimension of the suspension functor of the derived category
$D^b(\Repr(Q))$.

To show some monoidal structure is $\fpd$-infinite
[Definition \ref{xxdef0.2}(1)], we need the following concepts.

\begin{definition}
\label{xxdef1.5}
Let $\mathcal{C}$ be a $\Bbbk$-linear category.
\begin{enumerate}
\item[(1)]
Let $\phi$ be an infinite set of objects in $\mathcal{C}$.
We say $\phi$ is an {\it infinite brick set} if
$$\Hom_{\mathcal{C}}(X,Y)=\begin{cases} \Bbbk & {\text{  if  }}
X=Y\quad {\text{in $\phi$}},\\
0 & {\text{  if  }}
X\neq Y \quad {\text{in $\phi$}}.\end{cases}$$
\item[(2)]
Suppose $\mathcal{C}$ is abelian or triangulated. A brick set 
$\phi$ (either finite or infinite) is called a {\it connected brick set}
if $\Ext^1_{\mathcal{C}}(X,Y)\neq 0$ for
all $X,Y\in \phi$.
\end{enumerate}
\end{definition}

The next is about the definition of a weak bialgebra.

\begin{definition}
\label{xxdef1.6}
Let $A$ be an algebra with a $\Bbbk$-linear morphism
$\Delta: A\rightarrow A\otimes A$. We say $\Delta$ is a
{\it prealgebra morphism} if
\begin{equation}
\notag
\Delta(ab)=\Delta(a)\Delta(b)
\end{equation}
for all $a,b\in A.$
\end{definition}

A prealgebra morphism is an algebra morphism if and
only if $\Delta(1)=1\otimes 1$ where $1$ is the
identity (or unit) element of $A$.

\begin{definition} \cite[Definition 2.1]{BNS1999}
\label{xxdef1.7}
A {\it weak bialgebra} is a vector space $B$ over the
base field $\Bbbk$ with the structures of
\begin{enumerate}
\item[(a)]
an associative algebra $(B, m, 1 )$ with multiplication
$m: B\otimes B\to B$ and unit $1 \in B$, and
\item[(b)]
a coassociative coalgebra $(B, \Delta, \varepsilon)$ with
comultiplication $\Delta: B\to B\otimes B$ and couint
$\varepsilon: B\to \Bbbk$
\end{enumerate}
satisfying the following conditions.
\begin{enumerate}
\item[(i)]
The comultiplication $\Delta: B\to B\otimes B$ is a
prealgebra morphism.
\item[(ii)]
The unit and counit satisfy
\begin{equation}
\label{E1.7.1}\tag{E1.7.1}
(\Delta(1 )\otimes 1 )(1 \otimes \Delta(1 ))
=(\Delta\otimes Id) \Delta(1 )
=(1 \otimes \Delta(1 ))(\Delta(1 )\otimes 1 )
\end{equation}
and
\begin{equation}
\label{E1.7.2}\tag{E1.7.2}
\varepsilon( xyz)=\sum\varepsilon(x y_{(1)})\varepsilon(y_{(2)}z)=
\sum\varepsilon(x y_{(2)})\varepsilon(y_{(1)}z),
\end{equation}
where $\Delta(y)=\displaystyle\sum y_{(1)}\otimes y_{(2)}$ is the Sweedler notation.
\end{enumerate}
\end{definition}

We refer to \cite{BCJ2011, BNS1999, NTV2003, NV2002} for many
other basic definitions related to weak bialgebras and weak Hopf
algebras. The tensor structure of left modules over a weak
bialgebra \cite[Proposition 2]{NTV2003} is given below.

\begin{definition}
\label{xxdef1.8}
Let $A$ be a weak bialgebra over $\Bbbk$. For two left $A$-modules
$M$ and $N$, define $M\otimes^l N=\Delta(1)(M\otimes_{\Bbbk} N)$
where $\otimes_{\Bbbk}$ is the tensor product over $\Bbbk$.
\end{definition}

The following lemma is clear.

\begin{lemma}
\label{xxlem1.9}
Let $A$ be a weak bialgebra.
\begin{enumerate}
\item[(1)]
With the tensor product $-\otimes^l-$ given in
Definition \ref{xxdef1.8}, both $A-\Modfd$ and $A-\Mod$ are
monoidal abelian categories.
\item[(2)]
Both
$D^b(A-\Modfd)$ and $D^b(A-\Mod)$ are monoidal triangulated.
\end{enumerate}
\end{lemma}

Finally we mention a fact in quiver representations.

\begin{lemma} \cite[p.63]{GR1992}
\label{xxlem1.10}
Let $Q$ be a quiver of type $\mathbb{A}_n$. Then
$M\{i,j\}$, for $1\leq i<j\leq n$, defined as in
\eqref{E0.7.2}-\eqref{E0.7.3}, form the
complete list of indecomposable representations of $Q$,
up to isomorphisms.
\end{lemma}

\begin{convention}
\label{xxcon1.11}
For the rest of the paper, we will use $A$ for an algebra over
$\Bbbk$. It could have a bialgebra or weak bialgebra structure.
We will use $\mathcal{A}$ for the abelian category of finite
dimensional left $A$-modules, also denoted by $A-\Modfd$.
Let $\mathcal{T}$ be a triangulated category that could have
extra monoidal triangulated structure. Sometimes $\mathcal{T}$
denotes the bounded derived category $D^b(\mathcal{A})$.
A general $\Bbbk$-linear or monoidal category is denoted by
$\mathcal{C}$.
\end{convention}

\section{Preliminaries on quiver representations}
\label{xxsec2}

We refer to \cite{ASS2006} for some basic concepts in quiver
representation theory. Here we fix some convention. Let $Q=
(Q_0, Q_1, s, t)$ be a quiver where $Q_0$ is the set of
vertices of $Q$, $Q_1$ is the set of arrows of $Q$, and
$s,t: Q_1\to Q_0$ are source and target maps of $Q$ respectively.
Let $M$ be a representation of $Q$. For each vertex $i\in Q_0$,
let $(M)_i$ denote the vector space at $i$. For each arrow
$\alpha\in Q_1$ from vertex $i:=s(\alpha)$ to vertex
$j:=t(\alpha)$, let $(M)_{\alpha}$ denote the $\Bbbk$-linear map
from $(M)_{i}$ to $(M)_{j}$ corresponding to $\alpha$. Let
${\text{Rep}}(Q)$ be the category of all representations of $Q$
and $\Repr(Q)$ be the full subcategory of ${\text{Rep}}(Q)$
consisting of finite dimensional representations. By
\cite[Theorem 1.7 in Chapter VII]{ASS2006}, every finite
dimensional hereditary algebra $A$ is Morita equivalent to
the path algebra $\Bbbk Q$ of a finite acyclic quiver $Q$.

The definition of a weak bialgebra is given in
Definition \ref{xxdef1.7}.
The path algebra $\Bbbk Q$ is naturally equipped with a
coalgebra structure that makes it a weak bialgebra, see
\cite[Example 2.5]{NV2002} and \cite[Section 3]{He2008a}.
We state this known fact as follows.

\begin{lemma}
\label{xxlem2.1}
Let $Q$ be a finite quiver.
\begin{enumerate}
\item[(1)]
Its path algebra $\Bbbk Q$ is a cocommutative weak bialgebra
whose coalgebra structure is determined by
$$\Delta(p)=p\otimes p \quad \mathrm{and }\quad \varepsilon(p)=1$$
for any path $p=\alpha_1\alpha_2\cdots \alpha_m$ of length
$m\geq 0$.
\item[(2)]
The weak bialgebra structure in part {\rm{(1)}} is a bialgebra if
and only if $|Q_0|=1$.
\end{enumerate}
\end{lemma}

Since $\Bbbk Q$ is a cocommutative weak bialgebra,
$\Repr(Q)(\cong \Bbbk Q-\Modfd)$ is a symmetric monoidal
abelian category where the tensor product is given in Definition
\ref{xxdef1.8}. For two representations $M=((M)_i,(M)_\alpha)$
and $N=((N)_i,(N)_\alpha)$ of $Q$ where $i\in Q_0$ and
$\alpha\in Q_1$, we can define the {\it vertex-wise tensor
product} $M\otimes^{v} N$ by
\begin{equation}
\label{E2.1.1}\tag{E2.1.1}
(M\otimes^{v} N)_i=(M)_i \otimes_{\Bbbk} (N)_i, \quad {\text{and}}\quad
(M\otimes^{v} N)_\alpha=(M)_\alpha \otimes_{\Bbbk} (N)_\alpha,
\end{equation}
for all $i\in Q_0$ and $\alpha\in Q_1$. Then the tensor product
$M\otimes^l N$ given in Definition \ref{xxdef1.8} is exactly equal to
the vertex-wise tensor product $M\otimes^{v} N$ give in \eqref{E2.1.1}.
Therefore, we do not distinguish these two tensors and denote them by
$M\otimes N$. The tensor structure of quiver representations has been
studied by many researchers, see, for example, \cite{He2005,
He2008a, He2008b, He2009, KS2012, Ki2010}. Note that the
bounded derived category $D^b(\Repr(Q))$ is a tensor triangulated
category in the sense of \cite[Definition 1.1]{B2005}; consequently,
it is a monoidal triangulated category.

In this paper we study more than one tensor structures of the quiver
representations. But, in this section, we are only working on the
tensor structure defined by \eqref{E2.1.1}. We start with some
details about quiver representations.

We have defined the Frobenius-Perron dimension, denoted by
$\fpd$, of an object in a monoidal category $\mathcal{C}$
in Definition \ref{xxdef1.3}(4). A nice property of $\fpd$
is a duality property when applied to objects in $\Repr(Q)$.

\begin{definition}
\label{xxdef2.2}
Let $Q=(Q_0,Q_1,s_Q,t_Q)$ be a quiver and $M$ be a finite-dimensional
representation of $Q$.
\begin{enumerate}
\item[(1)]
Define the {\it opposite quiver} of $Q$, denoted by $Q^{op}$,
to be the quiver which reverses all arrows in $Q_1$, that is
$$Q^{op}_0=Q_0, Q^{op}_1=Q_1, s_{Q^{op}}=t_Q, t_{Q^{op}}=s_Q.$$
\item[(2)]
Define the {\it dual} of $M$, denoted by $M^*$, to be the
representation of $Q^{op}$ that is determined by
$$(M^*)_i=((M)_i)^*, (M^*)_{\alpha}=((M)_{\alpha})^*,$$
for all vertices $i$ and arrows $\alpha$.
\end{enumerate}
\end{definition}

We give an easy example.

\begin{example}
\label{xxex2.3}
Let $Q$ be $\xymatrix{1 \ar[r] & 2}$ and $M$ be
$\xymatrix{\Bbbk \ar[r]^{(1,0)^T} & \Bbbk^2}$. Then we have
$$Q^{op}: \xymatrix{1  & 2\ar[l]} \quad
{\text{and}} \quad
M^*= \xymatrix{\Bbbk &  \Bbbk^2\ar[l]_{(1,0)} }.$$
\end{example}

For two finite dimensional $\Bbbk$-vector spaces $U,V$, we
have
$$(V\otimes U)^*= U^*\otimes V^*\cong V^*\otimes U^*.$$
Furthermore, if we have linear maps between finite dimensional
$\Bbbk$-vector spaces, say $f:V \rightarrow V'$ and
$g: U\rightarrow U'$, then we have the commutative diagram
$$
\xymatrix{
(V\otimes U)^* \ar[d]^{\simeq }&
(V'\otimes U')^*\ar[d]^{\simeq }\ar[l]_{(f\otimes g)^*} \\
V^*\otimes U^* & V'^* \otimes U'^*. \ar[l]_{f^*\otimes g^*}}
$$
The above commutative diagram holds for objects in
$\Repr(Q)$ since $\Bbbk Q$ is a commutative weak bialgebra
[Lemma \ref{xxlem2.1}(1)]. It is clear that the $\Bbbk$-linear
dual induces a contravariant equivalence between the abelian
categories $\Repr(Q)$ and $\Repr(Q^{op})$. Combining these
two facts, we have
\begin{align}
\label{E2.3.1}\tag{E2.3.1}
\Hom_{(\Repr(Q))^{op}}(X,M\otimes N)
&\cong \Hom_{\Repr(Q)}(M\otimes N, X)\\
\notag &\cong
\Hom_{\Repr(Q^{op})}(X^*, M^*\otimes N^*)
\end{align}
for $M,N,X\in \Repr(Q)$. Now the following
lemma follows from \eqref{E2.3.1}.

\begin{lemma}
\label{xxlem2.4}
Let $Q$ be a finite quiver and $M$ be a finite
representation of $Q$. Then
$$\fpd(M\otimes_{{\Repr(Q)}^{op}}-)
=\fpd(M^*\otimes_{{\Repr(Q^{op})}}-)$$
where $M$ is considered as an object in the tensor
category ${\Repr(Q)}^{op}$ and $M^*$ an object
in $\Repr(Q^{op})$.

The same statement holds for other Frobenius-Perron
invariants such as $\fpv$.
\end{lemma}

Next we study some brick sets of quiver representations.
Let $S(i)$ denote the simple representation (of $Q$) at
vertex $i$ where
\begin{equation}
\label{E2.4.1}\tag{E2.4.1}
S(i)_j=\begin{cases}
\Bbbk & j=i\\
0 & j\neq i
\end{cases}
\quad \mathrm{and} \quad
S(i)_\alpha=0,\;\; \forall \;\; \alpha\in Q_1,
\end{equation}
and $e_i$ denote the trivial path at vertex $i$.
By the tensor structure of $\Repr(Q)$ \eqref{E2.1.1},
we have the following.

\begin{lemma}
\label{xxlem2.5}
Let $S(i)$ be the simple left $\Bbbk Q$-module defined as above
and $M$ in $\Repr(Q)$. Then $S(i)\otimes M$ is 
isomorphic to a direct sum of finitely many copies of $S(i)$.
\end{lemma}

In the above lemma, $S(i)\otimes M$ could be 0.

\begin{proposition}
\label{xxpro2.6}
Let $M$ be in $\Repr(Q)$. Then
$$\fpd(M)\geq d,$$
where $d=\max\limits_{v\in Q_0}\{\dim ((M)_v)\}$.
\end{proposition}

\begin{proof}
Let $a=\dim (M)_v$ and let $\phi_0=\{S(v)\}$ for $v\in Q_0$. Then
$$\Hom_{\Repr(Q)}(S(v),M\otimes S(v))
=\Hom_{\Repr(Q)}(S(v),S(v)^{\oplus a})=\Bbbk^{\oplus a}$$
which implies that $A(\phi_0, M\otimes -)=(a)_{1\times 1}$. Therefore
$\fpd(M)\geq a$ for all $a$. The assertion follows.
\end{proof}

Note that $\fpd(M)$ may be infinite as the next example
shows (and as predicted by Lemma \ref{xxlem6.4}).

\begin{example}
\label{xxex2.7}
Let $Q$ be the Kronecker quiver $\xymatrix{1 \ar@<1ex>[r]^{\alpha}
\ar[r]_{\beta} & 2}$. Let $S(1)$ be defined as in \eqref{E2.4.1}.
For every $c\in \Bbbk$, we define an object in $\Repr(Q)$:
\begin{equation}
\label{E2.7.1}\tag{E2.7.1}
M_c:=\xymatrix{\Bbbk \ar@<1ex>[r]^{\alpha=Id} \ar[r]_{\beta=c Id} & \Bbbk}.
\end{equation}
Then $M_c$ is a brick object (and such an object is also called a band 
module of $Q$ \cite[pp.160-161]{BR1987}). It is easy to see that
$\Hom(M_c, S(1))\cong \Bbbk$ and that $\{M_c,M_{c'}\}$ is a brick set
if $c\neq c'$. As a consequence, $\{M_c, \mid c\in \Bbbk\}$
is an infinite brick set.

Let $T$ be any finite subset of $\Bbbk$ and let $\phi:=
\{M_c \mid c\in T\}$. The $A(\phi, S(1)\otimes -)$ is
a $|T|\times|T|$ matrix in which all entries are 1. Then
$\rho(A(\phi, S(1)\otimes -))=|T|$. Since $\Bbbk$ is infinite,
we obtain that $\fpd(S(1))=\infty$.
\end{example}

Let us consider a slightly more general situation.

\begin{example}
\label{xxex2.8}
Suppose $Q$ is another quiver and $p_1$ and $p_2$ are two paths from
vertex $i$ to vertex $j$ that do not intersect except at
the two endpoints. Then we can consider a similar brick object
$M_c$ so that
$$\begin{aligned}
(M_c)_v&=\begin{cases}
\Bbbk & \quad {\textrm{if $v$ is in either $p_1$ or $p_2$}},\\
0 & \quad {\textrm{otherwise,}}\end{cases}
\\
(M_c)_{\alpha}&=\begin{cases}
Id & \quad {\textrm{if $\alpha$ is in either $p_1$ or $p_2$, but
not the first arrow in $p_2$}},\\
cId & \quad {\textrm{if $\alpha$ is the first arrow in $p_2$}},\\
0 & \quad {\textrm{otherwise,}}\end{cases}
\end{aligned}
$$
or, similar to \eqref{E2.7.1}, we can write it as
$$M_c:=\xymatrix{\Bbbk \ar@<1ex>[r]^{p_1=Id} \ar[r]_{p_2=cId }
& \Bbbk}.$$
Then $\{M_c, \mid c\in \Bbbk\}$ is an infinite brick set.
\end{example}

We will use this example later.

\section{Discrete categories}
\label{xxsec3}

In this section we will prove some basic lemmas for
monoidal abelian categories that are needed in the
proof of Theorem \ref{xxthm0.4}. We start with a definition.

\begin{definition}
\label{xxdef3.1}
Let $\mathcal{C}$ be a monoidal abelian category.
We say $\mathcal{C}$ is {\it discrete} if
\begin{enumerate}
\item[(a)]
$\mathcal{C}$ is $\Hom$-finite, namely
$\Hom_{\mathcal{C}}(M,N)$ is finite dimensional over $\Bbbk$ for objects
$M,N$ in ${\mathcal C}$,
\item[(b)]
every object in $\mathcal{C}$ has finite length,
\item[(c)]
$\mathcal{C}$ has finitely many simple objects, say
$\{S_1,\cdots,S_n\}$, up to isomorphisms, and
\item[(d)]
for all simple objects $S_i$ and $S_j$ in
$\mathcal{C}$,
\begin{equation}
\label{E3.1.1}\tag{E3.1.1}
S_i\otimes S_j\cong \begin{cases} S_i & {\textrm{ if $i=j$}}\\
0 & {\textrm{ if $i\neq j$}}.\end{cases}
\end{equation}
\end{enumerate}
\end{definition}

Note that an essentially small category $\mathcal{C}$ satisfying 
condition (b) is called a {\it length} category \cite{Ga1973,KV2018}. 

Let $Q$ be a finite quiver. Then there is a canonical monoidal
abelian structure on $\Repr(Q)$ induced by the weak bialgebra
structure defined in Lemma \ref{xxlem2.1}. The following lemma
follows immediately from the definition, see \eqref{E2.1.1}.

\begin{lemma}
\label{xxlem3.2}
Let $Q$ be a finite acyclic quiver. Then the canonical
monoidal abelian structure on $\Repr(Q)$ is discrete.
\end{lemma}

For the rest of this section we assume that $\mathcal{C}$
is discrete.
As a consequence, $\mathcal{C}$ is a Krull-Schmidt category.
For $M\in \mathcal{C}$, let $\ell(M)$ denote the length of
$\mathcal{C}$. Let $IC(M)$ denote the {\it isomorphism class}
of all (possibly repeated) simple subquotients of $M$.
This can be obtained by considering any composition
series of $M$. Even though composition series of $M$ is not
unique, $IC(M)$ is unique, so well-defined.

\begin{lemma}
\label{xxlem3.3}
Let $\mathcal{C}$ be a $\Hom$-finite monoidal abelian category with
finitely many simple objects. Let $\mathbf{1}$ be the unit object.
\begin{enumerate}
\item[(1)]
$\ell(-)$ is additive.
\item[(2)]
For every nonzero object $M$ there is a simple object
$S\in IC(\mathbf{1})$ such that $S\otimes M\neq 0$. By symmetry,
there is a simple object $T\in IC(\mathbf{1})$ such that
$M\otimes T\neq 0$.
\item[(3)]
If $S\in IC(\mathbf{1})$ and $T$ is
a simple object in $\mathcal{C}$,
then $S\otimes T$ is either 0 or a simple object. For each $T$,
there is only one $S\in IC(\mathbf{1})$ such that
$S\otimes T\neq 0$.
\item[(4)]
The multiplicity of any simple object $S$ in $IC(\mathbf{1})$ is 1.
\item[(5)]
If $S, T\in IC(\mathbf{1})$, then
$$S\otimes T\cong \begin{cases} S & {\textrm{ if $S=T$}}\\
0 & {\textrm{ if $S\neq T$}}.\end{cases}
$$
\item[(6)]
$\mathcal{C}$ is discrete if and only if
$S\in IC(\mathbf{1})$ for all simple objects $S$.
\end{enumerate}
\end{lemma}

\begin{proof}
(1) Clear from the definition.

(2) By part (1), we have

\begin{equation}
\label{E3.3.1}\tag{E3.3.1}
\ell(M)=\ell(\mathbf{1}\otimes M)=\sum_{S\in IC(\mathbf{1})}
\ell(S\otimes M).
\end{equation} Therefore there is an $S\in IC(\mathbf{1})$
such that $S\otimes M\neq 0$.

(3) If $S\otimes T\neq 0$, then, by \eqref{E3.3.1}, we have
$$1=\ell(T)=\ell(\mathbf{1}\otimes T)
= \sum_{S'\in IC(\mathbf{1})} \ell(S'\otimes T)
\geq \ell(S\otimes T)\geq 1.$$
Therefore $\ell(S\otimes T)= 1$ and $\ell(S'\otimes T)=0$
for all other $S'\in IC(\mathbf{1})$.

(4) This follows from \eqref{E3.3.1} by taking a simple
object $M$ with $S\otimes M\neq 0$.

(5) It remains to show that $S\otimes T=0$ if
$S$ and $T$ are distinct elements in $IC(\mathbf{1})$.
Suppose on the contrary that $U:=S\otimes T\neq 0$.
By part (3), it is a simple object. Since $U=S\otimes T$ is
a subquotient of $\mathbf{1}\otimes \mathbf{1}$,
$U$ is in $IC(\mathbf{1})$. Since $S\neq T$, we have
either $U\neq S$ or $U\neq T$. By symmetry, we assume
that $U\neq S$. By part (3), there is only one
$W\in IC(\mathbf{1})$ such that $W\otimes U\neq 0$.
This implies that $W\otimes S\neq 0$ as $U=S\otimes T$.
There are two different objects, namely, $S,U\in IC(\mathbf{1})$
such that $W\otimes S\neq 0$ and $W\otimes U\neq 0$.
By the left-version of part (3) this is impossible.
The assertion follows.

(6) If $IC(\mathbf{1})$ contains all simple objects, then
by part (5), $\mathcal{C}$ is discrete.
Conversely, suppose $\mathcal{C}$ is discrete.
For every simple object $T$, by part (2), there is
an $S\in IC(\mathbf{1})$ such that $S\otimes T\neq 0$.
By the definition of discreteness, $T=S$.
So $T\in IC(\mathbf{1})$.
\end{proof}

\begin{proposition}
\label{xxpro3.4}
Let $A$ be a finite dimensional algebra of finite
global dimension. Suppose that $(A-\Modfd, \otimes)$
is a discrete monoidal abelian category. Then, for any simple
left $A$-module $S$ and any $M\in A-\Modfd$,
$$M\otimes S\cong S^{\oplus n}$$
where $n$ is the number of copies of $S$ in the
the composition series of $M$.
\end{proposition}

\begin{proof}
By the 'no loops conjecture', which was proved by Igusa \cite{Ig1990},
\begin{equation}
\label{E3.4.1}\tag{E3.4.1}
\Ext^1_{A}(S,S)=0.
\end{equation}

By definition, $-\otimes-$ is biexact. Hence
$M\otimes S$ has a composition series that is induced
by the composition series of $M$. Let $T$ be a simple
subquotient of $M$. Then $T\otimes S$ is either $S$
when $T\cong S$ or $0$ if $T\not\cong S$. Thus
$M\otimes S$ has a composition series with each
simple subquotient being $S$. The assertion follows
from \eqref{E3.4.1}.
\end{proof}

Recall that $\otimes^v$ is the
canonical tensor given in \eqref{E2.1.1}.
We have an immediate consequence.

\begin{corollary}
\label{xxcor3.5}
Let $Q$ be a finite acyclic quiver.
If $(\Repr(Q), \otimes)$ is another discrete monoidal abelian
structure on $\Repr(Q)$, then for any $M\in \Repr(Q)$
and any simple representation $S$ over $Q$,
$$M\otimes S\cong M\otimes^v S$$
where $\otimes^v$ is defined as in \eqref{E2.1.1}.
\end{corollary}

There are a lot of monoidal categories that are not
discrete. For example, for a finite quiver $Q$,
if $\Repr(Q)$ is equipped with other
bialgebra structure, it may not be discrete,
see Proposition \ref{xxpro7.7}(a-d).
We conclude this section with the definition of
a discrete action.

\begin{definition}
\label{xxdef3.6}
Let $\mathcal{C}$ be a monoidal abelian category acting on
an abelian category ${\mathcal A}$. Assume that both
${\mathcal C}$ and ${\mathcal A}$ satisfy Definition
\ref{xxdef3.1}(a,b,c). Let $\{T_1,\cdots,T_n\}$ (respectively,
$\{S_1,\cdots,S_m\}$) be the complete list of simple objects
in ${\mathcal C}$ (respectively, ${\mathcal A}$), where $m\geq n$.
The action of $\mathcal{C}$ on ${\mathcal A}$ is called
{\it discrete} if
\begin{enumerate}
\item[(d)']
there is a permutation $\sigma\in S_n$ such that
\begin{equation}
\label{E3.6.1}\tag{E3.6.1}
T_i\odot S_j\cong \begin{cases} S_j & {\textrm{ if $j=\sigma(i)$}}\\
0 & {\textrm{ if $j\neq \sigma(i)$}}.\end{cases}
\end{equation}
\end{enumerate}
\end{definition}

\section{Proof of Theorem \ref{xxthm0.4}}
\label{xxsec4}

The aim of this section is to prove Theorem \ref{xxthm0.4}.
We need first recall some facts from representation theory
of quivers.

\begin{proposition}
\cite[Proposition 2.5 in Chapter VII]{ASS2006}
\label{xxpro4.1}
Let $Q$ be a finite, connected, and acyclic quiver and $M$
be a brick such that there exists $a\in Q_0$ with
$\dim (M)_a >1$. Let $Q'$ be the quiver defined as
follows: $Q'=(Q'_0,Q'_1)$, where $Q'_0=Q_0\cup \{b\}$;
$Q'_1=Q_1\cup \{\alpha\}$; and $\alpha:b\rightarrow a$.
Then $\Bbbk Q'$ is of infinite representation type.
\end{proposition}

By duality and Proposition \ref{xxpro4.1}, if $\alpha$
is an arrow of the form $a\rightarrow b$, then $\Bbbk Q'$
is also of infinite representation type.	

\begin{lemma}\cite[Corollary 5.14 in Chapter VII]{ASS2006}
\label{xxlem4.2}
If $Q$ is a quiver of type $\mathbb{ADE}$, see
\cite[p.252]{ASS2006}, then every indecomposable
representation of $Q$ is a brick.
\end{lemma}

Recall from Definition \ref{xxdef1.4}(3) that an algebra $A$
is {\it wild} or {\it of wild representation type}
if there is a faithful exact embedding of abelian categories
\begin{equation}
\label{E4.2.1}\tag{E4.2.1}
Emb : \Bbbk\langle x_1, x_2 \rangle -\Modfd
\longrightarrow  \mathcal{A}:=A-\Modfd
\end{equation}
that preserves indecomposables and respects isomorphism classes
(namely, for all objects $M_1,M_2$ in $\Bbbk\langle
x_1, x_2\rangle-\Modfd$, $Emb(M_1) \cong Emb(M_2)$ if and only if
$M_1\cong M_2$). A stronger notion of wildness is the following.
An algebra A is called {\it strictly wild}, or {\it fully wild},
if $Emb$ in \eqref{E4.2.1} is a fully faithful embedding
\cite[Proposition 5]{Ar2005}. By definition, strictly wild is
wild, but the converse is not true. It is well-known that a wild
path algebra $\Bbbk Q$ is always strictly wild, see a comment of
Gabriel \cite[p.140]{Ga197374} or \cite[Proposition 7]{Ar2005}.

\begin{lemma}
\label{xxlem4.3}
Let $A$ be a finite dimensional algebra that is strictly wild.
Let $\mathcal{C}$ be an abelian category containing
$\mathcal{A}$ as a full subcategory. Then
$\mathcal{C}$ contains an infinite connected brick set. As a
consequence, if $Q$ is a finite acyclic quiver that is wild,
then $\Repr(Q)$ contains an infinite connected brick set.
\end{lemma}

\begin{proof} The consequence follows from the fact that
a wild quiver is strictly wild. So we only prove the main
assertion.

Let $A$ be strictly wild. By definition, there is
a fully faithful embedding
\begin{equation}
\notag
Emb : \Bbbk\langle x_1, x_2 \rangle -\Modfd\longrightarrow
\mathcal{A}\longrightarrow \mathcal{C}.
\end{equation}
For each $c\in \Bbbk$, let $M(c)$ denote the 1-dimensional simple
module $\Bbbk\langle x_1, x_2\rangle/(x_1-c,x_2)$ and let $N_c$
be $Emb(M_c)$. By taking a free resolution $M(c)$, one can check
that $\Ext^1_{\Bbbk\langle x_1, x_2\rangle}(M(c),M(c'))\neq 0$
for all $c,c'$. Hence $\{M(c)\mid c\in\Bbbk\}$ is an infinite
connected brick set in  $\Bbbk\langle x_1, x_2 \rangle-\Modfd$.
Since $Emb$ a fully faithful embedding, $\{N_c\mid c\in\Bbbk\}$
is an infinite connected brick set of $\mathcal{C}$.
\end{proof}

\begin{lemma}
\label{xxlem4.4}
Let $\mathcal{C}$ be an abelian category of finite
global dimension and let $\mathcal{T}$ be the bounded
derived category $D^b(\mathcal{C})$. Suppose that
\begin{enumerate}
\item[(1)]
$\mathcal{T}$ is triangulated equivalent to
$D^b(B-\Modfd)$ for a finite dimensional hereditary
algebra $B$ via tilting object $X$, namely,
$$\RHom_{\mathcal{T}}(X,-):\mathcal{T}
\to D^b(B-\Modfd)$$
is a triangulated equivalence where $B\cong
\RHom_{\mathcal{T}}(X,X)$, and
\item[(2)]
$\mathcal{C}$ contains an infinite {\rm{(}}respectively,
infinite connected{\rm{)}} brick set.
\end{enumerate}
Then $B-\Modfd$ contains an infinite {\rm{(}}respectively,
infinite connected{\rm{)}} brick set.
\end{lemma}

Note that, if $\mathcal{C} =A-\Modfd$ for some
finite dimensional algebra $A$ and if $\mathcal{T}$
is triangulated equivalent to $D^b(B-\Modfd)$,
then, by tilting theory, the existence of $X$ is
automatic.

\begin{proof}[Proof of Lemma \ref{xxlem4.4}]
We only prove the assertion for ``infinite brick set''.
The proof for ``infinite connected brick set''
is similar.

Let
$$F:=\RHom_{\mathcal{T}}(X, -): \mathcal{T}
\longrightarrow D^b(B-\Modfd)$$
be an equivalence of triangulated categories. Let
$\{N(c)\mid c\in U\}$ be an infinite brick set of $\mathcal{C}$
by hypothesis. Then
$$\{F(N(c)) \mid c\in U\}$$
is an infinite brick set of $D^b(B-\Modfd)$.
Since $X$ has finite projective dimension, there is an integer
$n$ independent of $c\in U$ such that
\begin{equation}
\label{E4.4.1}\tag{E4.4.1}
{\text{$H^i(F(N(c)))=0$ for all $|i| > n$.}}
\end{equation}
Note that $B-\Modfd$ is hereditary, which implies that
every indecomposable object in $D^b(B-\Modfd)$ is of the
form $M[i]$ for some indecomposable object $M\in B-\Modfd$
and for some $i$ \cite[Section 2.5]{Ke2007}. By \eqref{E4.4.1},
$F(N(c))=M_c[i_c]$ for some indecomposable object
$M_c\in B-\Modfd$ and some integer $|i_c|\leq n$.
Since $U$ is infinite, there is an infinite subset $U'\subseteq U$
such that $i_c$ is a constant for all $c\in U'$. Let $i_0$ denote
such $i_c$. Thus $\{M_c[i_0]\mid c\in U'\}$ is an infinite brick
set in $D^b(B-\Modfd)$. Since the suspension $[1]$ is an
isomorphism of $D^b(B-\Modfd)$, $\{M_c\mid c\in U'\}$ is an
infinite brick set in $D^b(B-\Modfd)$. Finally, using the
fact that $B-\Modfd$ is a full subcategory of
$D^b(B-\Modfd)$, we obtain that $\{M_c\mid c\in U'\}$ is an
infinite brick set in $B-\Modfd$.
\end{proof}

\begin{lemma}
\label{xxlem4.5}
Let $A$ be a finite dimensional hereditary algebra that is not
of finite representation type. Then
the abelian category $A-\Modfd$ contains an infinite brick
set. As a consequence, if $Q$ is a finite acyclic quiver not of type
$\mathbb{ADE}$, then $\Repr(Q)$ contains an infinite brick set.
\end{lemma}

\begin{proof} By \cite[Theorem 1.7 in Chapter VII]{ASS2006}
every such $A$ is Morita equivalent to a path algebra $\Bbbk Q$ for
some finite acyclic quiver $Q$. By Lemma \ref{xxlem4.4}, we may
assume that $A$ is $\Bbbk Q$.

Since $A$ is not of finite type, $Q$ is not of finite type.
Lemma 4.3 settles the case where Q is of wild representation type.

Case 1: $Q$ is of type $\widetilde{\mathbb{A}}$. Since $Q$ is acyclic,
there exist two different paths $p_1$ and $p_2$ from $v$ to $u$,
where $v\neq u\in Q_0$. We can further assume that
the length $p_1$ is smallest among all such choices. In this case,
$\Repr(Q)$ contains an infinite brick set by Example \ref{xxex2.8}.

Case 2: $Q$ is of type $\widetilde{\mathbb{D}}\widetilde{\mathbb{E}}$.
We consider a slightly more general situation and then apply the
assertion to the special case (see quivers in
\cite[Corollary 2.7 in Chapter VII]{ASS2006}). If there exists a
subquiver $Q'$ of $Q$ and an indecomposable representation $M$
of $Q'$ satisfying:
\begin{enumerate}
\item[(a)]
$Q'$ is a quiver of type $\mathbb{D}$ or $\mathbb{E}$,
\item[(b)]
there exists $x\in Q'_0$, $\dim (M)_x>1$,
\item[(c)]
$\{y\} \in Q_0\setminus Q'_0$,
\item[(d)]
there exists an arrow $\alpha\in Q_1$ such that $\alpha:y\rightarrow x$,
\end{enumerate}
then we construct a new representation $M(\lambda)$ as follows:
\[(M(\lambda))_v=
\begin{cases}
(M)_v & \mathrm{if}~ v\in Q'_0\\
\Bbbk & \mathrm{if}~ v=y\\
0 & \mathrm{otherwise},
\end{cases}
\qquad\qquad
(M(\lambda))_{\beta}=
\begin{cases}
(M)_{\beta} & \mathrm{if}~ \beta\in Q'_1\\
\lambda & \mathrm{if}~ \beta=\alpha\\
0 & \mathrm{otherwise},
\end{cases}
\]where $\lambda: \Bbbk \rightarrow (M)_x$ is a $\Bbbk$-linear map.

Then by the proof of \cite[Proposition 2.5 in Chapter VII]{ASS2006},
each $M(\lambda)$ is a brick and there exists infinitely many
pairwise non-isomorphic bricks of the form $M(\lambda)$. In fact,
the proof of \cite[Proposition 2.5 in Chapter VII]{ASS2006} shows that
there is an infinite set of $U:=\{\lambda: \Bbbk \to M_x\}$ such that
$\Hom_{\Repr(Q)}(M(\lambda), M(\lambda'))=0$ for all
$\lambda,\lambda'\in U$. This means that $\Repr(Q)$ contains an
infinite brick set. Dually, If we change the condition (d) into (d)':
\begin{enumerate}
\item[(d)']
there exists an arrow $\alpha\in Q_1$ such that $\alpha:x\rightarrow y$,
\end{enumerate}
we can still construct an infinite brick set as above.

Now we go back to a quiver of type $\widetilde{\mathbb{D}}\widetilde{\mathbb{E}}$.
By Lemma \ref{xxlem4.4}, we can
assume that
\begin{enumerate}
\item[(e)]
$Q'$ is one of the quivers in
\cite[Corollary 2.6 in Chapter VII.2]{ASS2006}, and that
\item[(f)]
(c) and (d) hold.
\end{enumerate}
Note that (e) implies that (a) holds.
By \cite[Corollary 2.6 in Chapter VII.2]{ASS2006}, (b) holds.
Therefore we proved that $\Repr(Q)$ contains an infinite brick
set.
\end{proof}

\begin{lemma}
\label{xxlem4.6}
Let $\mathcal{C}$ be a monoidal abelian category acting on
an abelian category ${\mathcal A}$. Assume that both
${\mathcal C}$ and ${\mathcal A}$ satisfy Definition
\ref{xxdef3.1}(a,b,c). Suppose that
\begin{enumerate}
\item[(a)]
$\mathcal{A}$ contains an infinite brick set, and that
\item[(b)]
the action of ${\mathcal C}$ on ${\mathcal A}$ is
discrete.
\end{enumerate}
Then there is a simple $T\in {\mathcal C}$ such that
$\fpd(T)=\infty$.
\end{lemma}

Lemma \ref{xxlem4.6} may fail if the action is not discrete.
Let $Q$ be the Kronecker quiver in Example \ref{xxex2.7} and
$A$ be its path algebra equipped with the cocommutative bialgebra
structure in Proposition \ref{xxpro7.7}(a). Then $S(1)$ is the unit object
in $\mathcal{A}$ and $S(2)\otimes M=S(2)^{\oplus \dim(M)}$ for
any $M\in \mathcal{A}$. Since all indecomposables in $\Repr(Q)$
are well-understood, one can check that $\fpd(S(1))=\fpd(S(2))=1$
(details are omitted).

\begin{proof}[Proof of Lemma \ref{xxlem4.6}]
Let $\{N(c)\mid c\in U\}$ is an infinite brick set of
$\mathcal{A}$ and let $\{S_1,\cdots,S_n\}$ be the
complete list of simple objects in $\mathcal{A}$
up to isomorphism. For each $1\leq i\leq n$, define
$$U_i:=\{c\in U\mid \Hom_{\mathcal{A}}(N(c),S_i)\neq 0\}.$$
For each $c\in U$, there is an $i$ such that
$\Hom_{\mathcal{A}}(N(c),S_i)\neq 0$. This implies
that $U=\bigcup_{i=1}^n U_i$. Therefore there is
an $i$ such that $U_i$ is infinite. Without loss of
generality, we may assume that $U=U_1$ is infinite.

Since the action is discrete, there is a simple object
$T\in {\mathcal C}$ such that $T\odot S_1\cong S_1$.
Now $\Hom_{\mathcal A}(N(c), S_1)\neq 0$ implies
that every simple subquotient of $T\odot N(c)$ is
isomorphic to $S_1$. In particular, $T\odot N(c)$
contains a copy of $S_1$ for all $c$.

Let $W$ be any finite subset of $U$ and let
$\phi=\{N(c) \mid c\in W\}$. Using the above paragraph,
$$\Hom_{\mathcal{A}}(N(c), T\odot N(c'))
\neq 0$$
for all $c,c'\in W$. This implies that
$\rho(A(\phi, T\odot -))\geq |W|$ and $\fpd(T)
\geq |W|$. Since $|W|$ can be arbitrarily
larger, $\fpd (T)=\infty$.
\end{proof}

The following is a part of Theorem \ref{xxthm0.4}.

\begin{theorem}
\label{xxthm4.7}
Let $A$ be a finite dimensional hereditary algebra
and let ${\mathcal A}=A-\Modfd$.
Let ${\mathcal C}$ be a monoidal abelian category
satisfying Definition \ref{xxdef3.1}(a,b,c).
Suppose that there is an action of $\mathcal{C}$
on ${\mathcal A}$ that is discrete. Then the following
are equivalent:
\begin{enumerate}
\item[(a)]
$A$ is of finite representation type,
\item[(b)]
$\fpd(M)<\infty$ for every irreducible object $M\in \mathcal{C}$,
\item[(c)]
$\fpd(M)<\infty$ for every indecomposable object $M\in \mathcal{C}$,
\item[(d)]
$\fpd(M)<\infty$ for every object $M\in \mathcal{C}$
\end{enumerate}
\end{theorem}

\begin{proof}
%
(a) $\Longrightarrow$ (d):
If $A$ is of finite representation type, then $\mathcal{A}$
has only finitely many indecomposable objects. This means that there
are only finitely many brick sets. Then, by definition, $\fpd(\sigma)$
is finite for every endofunctor $\sigma$ of $\mathcal{A}$. In particular,
$\fpd(M)$ is finite for every representation $M\in \mathcal{C}$.

(d) $\Longrightarrow$ (c) $\Longrightarrow$ (b): Clear.

(b) $\Longrightarrow$ (a): It suffices to show that if $A$ is not
of finite representation type, then $\fpd(M)=\infty$ for some irreducible
representation $M\in \mathcal{C}$. The assertion follows from Lemmas
\ref{xxlem4.5} and \ref{xxlem4.6}.
\end{proof}

We will use the following lemma concerning a bound of
spectral radius of a matrix.

\begin{lemma}
[Gershgorin Circle Theorem \cite{Ger1931}]
\label{xxlem4.8}
Let $A$ be a complex $n\times n$ matrix, with entries $a_{ij}$.
For $i\in \{1,\dots ,n\},$ let
$R_{i}=\sum\limits_{j\neq i}\left|a_{{ij}}\right|$
be the sum of the absolute values of the non-diagonal entries in
the $i$-th row. Let $D(a_{ii},R_{i})\subseteq \mathbb {C}$ be a
closed disc centered at $a_{ii}$ with radius $R_{i}$. Then every
eigenvalue of $A$ lies within at least one of the Gershgorin
discs $D(a_{ii},R_{i}).$ As a consequence,
$\rho(A)\leq \max_i\{|a_{ii}|+R_i\}$.
\end{lemma}

\begin{proposition}
\label{xxpro4.9}
Suppose $\mathcal{T}$ is a triangulated category
satisfying
\begin{enumerate}
\item[(a)]
$\mathcal{T}$ is $\Hom$-finite and hence Krull-Schmidt,
\item[(b)]
there are objects $\{X_1,\cdots,X_N\}$ such that every
indecomposable object in $\mathcal{T}$ is of the form
$X_i[m]$ for some $1\leq i\leq N$ and $m\in \mathbb{Z}$, and
\item[(c)]
for every two indecomposable objects $X,Y$ in $\mathcal{T}$,
$\Hom_{\mathcal{T}}(X,Y[m])=0$ for $|m|\gg 0$.
\end{enumerate}
Then the following hold.
\begin{enumerate}
\item[(1)]
$\fpd(\sigma)<\infty$ for every endofunctor $\sigma$ of
$\mathcal{T}$.
\item[(2)]
If ${\mathcal C}$ is a monoidal triangulated category acting
on ${\mathcal T}$, then $\fpd(M)<\infty$ for
every object $M\in {\mathcal C}$.
\end{enumerate}
\end{proposition}

\begin{proof}
Let $\sigma$ be an endofunctor of $\mathcal{T}$. Since there are
only finitely many $X_i$ in hypothesis (b),
we can assume that every $\sigma(X_i)$ is a direct summand
of
\begin{equation}
\label{E4.9.1}\tag{E4.9.1}
X=\left(\bigoplus_{i=1}^{N} \bigoplus_{j=-\delta}^{\delta-1}
X_i[j]\right)^{\oplus \xi}
\end{equation}
for some fixed $\delta$ and $\xi$.

We make some definitions. Let
$$
\begin{aligned}
\alpha&=\max\{\dim \Hom_{\mathcal{T}}(X_i[s],X)\mid \;\forall \;\; s,i\},\\
\gamma&=\max\{|s| \mid \; \Hom_{\mathcal{T}}(X_i[s],X)\neq 0
{\text{ for some $i$}}\}.
\end{aligned}
$$

For any given finite brick set $\phi$, it is always is a subset of
$$\Phi:=\bigcup_{j=-D}^{D-1} \{X_1[j],\cdots,X_N[j]\}$$
for some large $D\gg 0$. Since $\phi$ is a subset of $\Phi$,
we have
$$\rho(A(\phi, \sigma))
\leq \rho(A(\Phi, \sigma)).$$
By Definition \ref{xxdef1.3}(4), it is enough to show that
$\rho(A(\Phi, \sigma))$ is uniformly bounded
on $\Phi$ (for each fixed $X$ as given in \eqref{E4.9.1}).
For the next calculation we make a linear order on the
objects in $\Phi$ as
\begin{align}
\label{E4.9.2}\tag{E4.9.2}
\Phi&=\{X_1[-D],\cdots,X_N[-D]\} \cup
\{X_1[-D+1],\cdots,X_N[-D+1]\} \cup \\
&\qquad \cdots \cup
\{X_1[D-2],\cdots,X_N[D-2]\} \cup
\{X_1[D-1],\cdots,X_N[D-1]\} \notag
\end{align}
and write is as $\Phi=\{Y_1,\cdots, Y_{2ND}\}$.
Write the adjacency matrix $A(\Phi, \sigma)$
as $(a_{ij})$. For each pair $(i,j)$, by definition,
$$a_{ij}=\dim \Hom_{\mathcal T}(X_{s_i}[w_i], \sigma(X_{s_j}[w_j]))
\leq \dim \Hom_{\mathcal T}(X_{s_i}[w_i], X[w_j])
\leq \alpha,$$
for some $s_i,s_j,w_i,w_j$; and by the ordering in \eqref{E4.9.2},
we obtain
$$a_{ij}=0 \quad {\text{if $|i-j|>2N\delta +\gamma+2.$}}$$
Then each $R_i$ in the Lemma \ref{xxlem4.8} is bounded
by $(2N\delta +\gamma+2)\alpha$.
By Lemma \ref{xxlem4.8} (Gershgorin Circle Theorem),
there is a bound of $\rho(A(\Phi, \sigma))$
which is independent of $D$. Since every finite brick set
$\phi$ is a subset of $\Phi$ for some large $D$,
$\rho(A(\phi, \sigma))$ has a bound
that is independent of $\phi$. Therefore
$\fpd(\sigma)$ is finite as desired.
\end{proof}

We will use the following special case.
Recall that ${\mathcal A}=A-\Modfd$
and that ${\mathcal T}=D^b({\mathcal A})$.

\begin{corollary}
\label{xxcor4.10}
Let $A$ be a finite dimensional hereditary algebra that
is of finite representation type. Then every
monoidal triangulated structure on $\mathcal{T}$ is
$\fpd$-finite.
\end{corollary}

\begin{proof} Since $A$ is of finite type, we can list
all indecomposable left $A$-modules $\{X_1,\cdots,X_N\}$.
Since $A$ is hereditary, every indecomposable
object in $\mathcal{T}$ is of the form $X_i[s]$
for some $1\leq i\leq N$ and $s\in \mathbb{Z}$
\cite[Lemma 3.3]{CGWZZZ2017}. Finally,
since $A$ is hereditary, then
$\Hom_{\mathcal{T}}(X_i, X_j[m])=0$ for $m\neq 0,1$.
Thus $\mathcal{T}$ satisfies hypotheses (a,b,c) in
Proposition \ref{xxpro4.9}. Then the assertion follows from
Proposition \ref{xxpro4.9}(2) by setting ${\mathcal C}=
{\mathcal T}$ and $\odot=\otimes$.
\end{proof}

\begin{lemma}
\label{xxlem4.11}
Let $A$ be a finite dimensional hereditary algebra.
Let ${\mathcal C}$ be a monoidal abelian category
satisfying Definition \ref{xxdef3.1}(a,b,c).
Suppose that $\mathcal{C}$ acts on ${\mathcal A}$
via $\odot$. Let $\odot_{D}$ be the induced action
of $D^b({\mathcal C})$ on ${\mathcal T}$.
Let $M$ be an object in $\mathcal{C}$, also viewed
as an object in $D^b({\mathcal C})$.
\begin{enumerate}
\item[(1)]
If $n\neq 0,1$, then $\fpd(M[n]\odot_{D}-)=0$.
\item[(2)]
$\fpd(M\odot_{D}-)=\fpd(M\odot -)$.
\end{enumerate}
\end{lemma}

\begin{proof}
(1) Suppose $n\geq 2$.
Let $\phi$ be a (finite) brick set. Since $\mathcal{A}$ is
hereditary, every indecomposable object is of the form $X[m]$. Then
we can write $\phi=\bigcup_{\lambda \in \mathbb{Z}} \phi_{\lambda}$ where
$\phi_{\lambda}$ is either empty or
$\{X_{\lambda,1}[\lambda],X_{\lambda,2}[\lambda],\cdots,
X_{\lambda,t_\lambda}[\lambda]\}$.
Since $\mathcal{A}$ is hereditary,
$$\Hom_{{\mathcal{T}}}(X_{\lambda,s}[\lambda], M[n]\odot_{D}
X_{\delta,s'}[\delta])=
\Hom_{{\mathcal{T}}}(X_{\lambda,s}[\lambda], (M\odot X_{\delta,s'})[n+\delta])
=0$$
for all $\lambda\leq \delta$.
Then $A(\phi, M[n]\odot_{D}-)$ is strictly upper
triangular. Therefore $\rho(A(\phi, M[n]\odot_{D}-))=0$.
As a consequence the assertion follows.

The proof for $n<0$ is similar.

(2) Let $\phi$ be a brick set as in part (1).
Similar to the proof of part (1), also see
\cite[Lemma 6.1]{CGWZZZ2017}, we obtain that
$A(\phi, M\odot_{D}-)$ is a block lower triangular
matrix. So we only need to consider the case that
$\phi=\{X_1[d],X_2[d],\cdots X_t[d]\}$ for the same
$d$. In this case, $A(\phi, M\odot_{D}-)
=A(\phi[-d],M\odot-)$. Therefore the assertion
follows.
\end{proof}

Now we are ready to prove Theorem \ref{xxthm0.4}.
We will use the notation introduced in Theorem
\ref{xxthm4.7} and Lemma \ref{xxlem4.11}.

\begin{theorem}
\label{xxthm4.12}
Let $A$ be a finite dimensional hereditary algebra
and let ${\mathcal A}=A-\Modfd$.
Let ${\mathcal C}$ be a monoidal abelian category
satisfying Definition \ref{xxdef3.1}(a,b,c).
Suppose that there is an action of $\mathcal{C}$
on ${\mathcal A}$ that is discrete. Then the following
are equivalent:
\begin{enumerate}
\item[(a)]
$A$ is of finite representation type,
\item[(b)]
$\fpd(M)<\infty$ for every irreducible object $M\in \mathcal{C}$,
\item[(c)]
$\fpd(M)<\infty$ for every indecomposable object $M\in \mathcal{C}$,
\item[(d)]
$\fpd(M)<\infty$ for every object $M\in \mathcal{C}$,
\item[(e)]
$\fpd(M \odot_{D} -)<\infty$ for every indecomposable object
$M\in D^b(\mathcal{C})$,
\item[(f)]
$\fpd(M \odot_{D} -)<\infty$ for every object $M\in D^b(\mathcal{C})$.
\end{enumerate}
Suppose $A$ is the path algebra $\Bbbk Q$ for some finite
quiver $Q$. Then any of
conditions {\rm{(a)}} to {\rm{(f)}} is equivalent to
\begin{enumerate}
\item[(g)]
$Q$ is a finite union of quivers of type $\mathbb{ADE}$.
\end{enumerate}
\end{theorem}

\begin{proof}
By Theorem \ref{xxthm4.7}, the
first four conditions are equivalent.

(a) $\Longrightarrow$ (f): This follows from
Proposition \ref{xxpro4.9} and the proof of Corollary
\ref{xxcor4.10}.

(f) $\Longrightarrow$ (e): Clear.

(e) $\Longrightarrow$ (c): This follows from Lemma
\ref{xxlem4.11}(2).
\end{proof}

Clearly Theorem \ref{xxthm0.4} is a special case of Theorem
\ref{xxthm4.12}.

\section{$mtt$-structures of a monoidal triangulated category}
\label{xxsec5}
First we recall the definition on a $t$-structure on a triangulated
category. The notion of a $t$-structure was introduced by
Beilinson-Bernstein-Deligne in \cite{BBD1981}. We make a
small change in the definition below.

\begin{definition}
\label{xxdef5.1}
Let $\mathcal{T}$ be a triangulated category.
\begin{enumerate}
\item[(1)]
A {\it $t$-structure} on $\mathcal{T}$ is a pair of full
subcategories $(\mathcal{T}^{\leq 0},
\mathcal{T}^{\geq 0})$ satisfying the following conditions.
\begin{enumerate}
\item[(1a)]
$\mathcal{T}^{\leq 0}\subseteq \mathcal{T}^{\leq 1}$ and
$\mathcal{T}^{\geq 0}\supseteq \mathcal{T}^{\geq 1}$ where
we use notation $\mathcal{T}^{\leq n}=\mathcal{T}^{\leq 0}[-n]$
and $\mathcal{T}^{\geq n}=\mathcal{T}^{\geq 0}[-n]$.
\item[(1b)]
If $M\in \mathcal{T}^{\leq 0}$ and $N\in \mathcal{T}^{\geq 1}$,
then $\Hom_{\mathcal{T}}(M,N)=0$.
\item[(1c)]
For any object $X\in \mathcal{T}$, there is a distinguished
(exact) triangle
$$M\to X\to N\to M[1]$$
with $M\in \mathcal{T}^{\leq 0}$ and $N\in \mathcal{T}^{\geq 1}$.
\end{enumerate}
\item[(2)]
The {\it heart} of the $t$-structure is the
full subcategory
$$\mathcal{T}^{\geq 0}\cap \mathcal{T}^{\leq 0}$$
which is denoted by $\mathcal{H}$ or $\mathcal{H}(\mathcal{T})$.
\item[(3)]
\cite[p.1427]{CR2018}
A $t$-structure is called {\it bounded} if for each $X\in
\mathcal{T}$, there exist $m\leq n$
such that $X\in \mathcal{T}^{\leq n} \cap \mathcal{T}^{\geq m}$.
\item[(4)]
\cite[p.1427]{CR2018}
A bounded $t$-structure is called {\it hereditary} if
$\Hom_{\mathcal{T}}(X, Y[n])=0$ for $n\geq 2$ and
$X,Y\in \mathcal{H}$.
\end{enumerate}
\end{definition}

As a classical example, if $\mathcal{T}$ is the derived category
$D^b(A-\Modfd)$, there is a natural $t$-structure on
$\mathcal{T}$ by setting $\mathcal{T}^{\leq 0}$ to be the complexes
concentrated in degrees less than or equal to 0 (and similarly for
$\mathcal{T}^{\geq 0}$). In this case the heart of this
$t$-structure is $A-\Modfd$.

Note that hereditary $t$-structures are very special. Even for the
path algebra of a quiver $Q$ of type $\mathbb{A}_3$, there is a
$t$-structure in $D^b(\Repr(Q))$ that is not hereditary, see
\cite{KV1988} for a classification of $t$-structures of $D^b(\Repr(Q))$
of a quiver of Dynkin type.

We would like to introduce a version of the $t$-structure in a monoidal
triangulated category. We use $mtt$ for ``monoidal triangulated $t$''
in the next definition.

\begin{definition}
\label{xxdef5.2}
Let $\mathcal{T}$ be a monoidal triangulated category
in parts (1,2,3) and a triangulated category in part (4).
\begin{enumerate}
\item[(1)]
A $t$-structure $(\mathcal{T}^{\leq 0},\mathcal{T}^{\geq 0})$
on $\mathcal{T}$ is called an {\it $mtt$-structure}
if the following conditions hold.
\begin{enumerate}
\item[(a)]
$\mathcal{T}^{\leq 0}\otimes \mathcal{T}^{\leq 0}\subseteq
\mathcal{T}^{\leq 0}$ and
$\mathcal{T}^{\leq 0}\otimes \mathcal{T}^{\leq 0}\not\subseteq
\mathcal{T}^{\leq -1}$.
\item[(b)]
Both $\mathcal{T}^{\leq 0}$ and $\mathcal{T}^{\geq 0}$
are closed under taking direct summands.
\item[(c)]
There is an integer $D\geq 0$ such that
$\mathcal{T}^{\geq D}\otimes \mathcal{T}^{\geq D}\subseteq
\mathcal{T}^{\geq D}$.
\end{enumerate}
\item[(2)]
The minimal integer $D$ in condition (c) is called the
{\it deviation} of the $mtt$-structure of $\mathcal{T}$.
\item[(3)]
The {\it deviation} of $(\mathcal{T}, {\bf 1}, \otimes)$
is defined to be
$$D_{\otimes}({\mathcal T})=
\inf \{ {\text{deviations of all possible $mtt$-structures of
$(\mathcal{T}, {\bf 1}, \otimes)$}}\}.$$
\item[(4)]
Suppose ${\mathcal T}$ is a triangulated category. The
{\it upper deviation} of ${\mathcal T}$ is defined to be
$$UD({\mathcal T})=
\sup \{ D_{\otimes}({\mathcal T})
\mid {\text{all possible monoidal triangulated structures
on ${\mathcal T}$}}\}.$$
The
{\it lower deviation} of ${\mathcal T}$ is defined to be
$$LD({\mathcal T})=
\inf \{ D_{\otimes}({\mathcal T})
\mid {\text{all possible monoidal triangulated structures
on ${\mathcal T}$}}\}.$$
\end{enumerate}
\end{definition}

\begin{example}
\label{xxex5.3} We give two classical examples.
\begin{enumerate}
\item[(1)]
If $A$ is a finite dimensional weak Hopf algebra (or a weak
bialgebra), then $A-\Modfd$ has a natural monoidal abelian
structure, and consequently,
$\mathcal{T}:=D^b(A-\Modfd)$ has an induced
monoidal triangulated structure. It is clear that
${\mathcal T}$ has a canonical $mtt$-structure
by setting $\mathcal{T}^{\leq 0}$ (respectively,
$\mathcal{T}^{\geq 0}$) to be the complexes concentrated in
degrees less than or equal to 0 (respectively, greater than
or equal to 0). In this case the deviation of the $mtt$-structure
is $0$. If $A$ is hereditary as an algebra, then the above
$t$-structure is hereditary.

By definition, $D_{\otimes}(\mathcal{T})=0$ when we consider
the monoidal triangulated structure given above. As a consequence,
$LD({\mathcal T})=0$ when ${\mathcal T}$ is considered as a
triangulated category. A special case is $LD(D^b(\Repr(Q)))=0$
for all finite acyclic quivers $Q$.
\item[(2)]
If $\mathbb{X}$ is a smooth projective scheme of dimension $d$,
then ${\mathcal T}:=D^b(coh(\mathbb{X}))$ has a canonical $mtt$-structure
by setting $\mathcal{T}^{\leq 0}$ (respectively,
$\mathcal{T}^{\geq 0}$) to be the complexes concentrated in
degrees less than or equal to 0 (respectively, greater than
or equal to 0).
If $\mathbb{X}$ is of dimension 1, then the above $t$-structure
is hereditary.

Note that the deviation of the canonical $mtt$-structure of
${\mathcal T}$ is at most $d$.
By definition, $D_{\otimes}(\mathcal{T})\leq d$ with the natural
monoidal triangulated structure. As a consequence,
$LD({\mathcal T})\leq d$ when $\mathcal{T}$ is considered as
a triangulated category.
\end{enumerate}
\end{example}

\begin{lemma}
\label{xxlem5.4}
Let ${\mathcal T}$ be a monoidal triangulated category
with an $mtt$-structure $(\mathcal{T}^{\leq 0},\mathcal{T}^{\geq 0})$
of deviation zero. Suppose that
$(\mathcal{T}^{\leq 0},\mathcal{T}^{\geq 0})$
is a hereditary $t$-structure of $\mathcal{T}$.
Then the heart of the $mtt$-structure is a monoidal
abelian category.
\end{lemma}

\begin{proof}
By \cite[Theorem 1.3.6]{BBD1981}, the heart $\mathcal{H}$
is an abelian category.

Since $\mathcal{T}$ is a monoidal triangulated category, there
is a unit object ${\bf 1}\in \mathcal{T}$. First we claim
that ${\bf 1}\in \mathcal{H}$. By definition, there is a
distinguished triangle
\begin{equation}
\label{E5.4.1}\tag{E5.4.1}
M\to {\bf 1}\to N\to M[1]
\end{equation}
where $M\in \mathcal{T}^{\leq 0}$ and $N\in \mathcal{T}^{\geq 1}$.
For any object $X\in \mathcal{H}$, since $X\otimes-$ is an exact
functor, $$X\otimes M\to X\to X\otimes N\to X\otimes M[1]$$
is a distinguished triangle. However, $X\otimes N\in
\mathcal{T}^{\geq 1}$ as the deviation is zero. Then
$\Hom(X,X\otimes N)=0$ by the definition of $t$-structure,
and
\begin{equation*}
\label{E5.4.2}\tag{E5.4.2}
X\otimes M[1]\cong (X\otimes N)\oplus X[1]=X\otimes (N\oplus {\bf 1}[1]).
\end{equation*}

By hypothesis the $mtt$-structure is hereditary. By
\cite[Lemma 2.1]{CR2018}, \eqref{E5.4.2} holds for all
$X\in \mathcal{T}$.
Take $X={\bf 1}$, then $M[1]\cong N\oplus {\bf 1}[1]$
and in \eqref{E5.4.1}, the morphism from {\bf 1} to $N$ is zero.
Hence ${\bf 1}$ is isomorphic to a direct summand of $M$,
which is in $\mathcal{T}^{\leq 0}$.

Similarly, for $Y\in \mathcal{T}^{\leq 0}$ and $f: Y\to {\bf 1}[-1]$,
there is a distinguished triangle:
$$Y \stackrel{f}{\longrightarrow} {\bf 1}[-1]\to Z\to Y[1].$$
Apply the exact functor $X\otimes-$ on the above triangle for
all $X\in \mathcal{H}$, and then we obtain $f=0$, i.e.
$\Hom(Y, {\bf 1}[-1])=0$ for all $Y\in \mathcal{T}^{\leq 0}$.
Therefore, ${\bf 1}\in \mathcal{T}^{\geq 0}$.
Finally, ${\bf 1}\in \mathcal{T}^{\geq 0}\cap \mathcal{T}^{\leq 0}=\mathcal{H}$.
Thus we proved the claim.

As for the tensor product bifunctor $\otimes$, since the deviation is zero,
$\mathcal{H}$ is closed under $\otimes$. Hence $\mathcal{H}$ is a monoidal
category with the induced tensor product $\otimes$. The exactness of
$\otimes$ in $\mathcal{H}$ follows from the exactness
of $\otimes$ in ${\mathcal T}$, see \cite[p.1426]{CR2018}.
\end{proof}

\begin{lemma}
\label{xxlem5.5}
Let $\mathbb{X}$ be a smooth projective curve
and let ${\mathcal T}$ be the monoidal triangulated
category $D^b(coh(\mathbb{X}))$.
\begin{enumerate}
\item[(1)]
The deviation of every hereditary $mtt$-structure
on ${\mathcal T}$ is positive.
\item[(2)]
For any finite dimensional weak bialgebra $A$,
$D^b(A-\Modfd)$ with canonical monoidal structure
is not isomorphic to
$\mathcal{T}$ as monoidal triangulated categories.
\end{enumerate}
\end{lemma}

\begin{proof}
(1) Suppose on the contrary that there is a hereditary
$mtt$-structure on $\mathcal{T}$ with deviation zero.

Let $\mathcal{H}$ be its heart. By Lemma \ref{xxlem5.4},
${\mathcal H}$ is a monoidal abelian category.
Let ${\mathcal O}_x$ be the skyscraper sheaf
at a point $x\in {\mathbb X}$. There is an integer
$n$ such that $M:={\mathcal O}_x [n]$ is in
${\mathcal H}$. Then $M\otimes M$ is in
${\mathcal H}$. By an easy
computation,
$$M\otimes M\cong {\mathcal O}_x [2n] \oplus {\mathcal O}_x [2n-1]
\cong M[n]\oplus M[n-1]$$
which cannot be in ${\mathcal H}$ for any $n$.
This yields a contradiction. Therefore the assertion
follows.

(2) It is clear that the deviation of the
canonical $mtt$-structure of $D_{\otimes}(D^b(A-\Modfd))$
is zero [Example \ref{xxex5.3}(1)].
This $mtt$-structure is also hereditary. Now
the assertion follows from part (1).
\end{proof}

For the rest of this section, we will use Frobenius-Perron
curvature, see Definition \ref{xxdef1.3}(5), to study the
uniqueness of $mtt$-structures with deviation zero, and
then prove Theorems \ref{xxthm0.5} and \ref{xxthm0.7}.

\begin{definition}
\label{xxdef5.6}
Let ${\mathcal C}$ be a monoidal abelian category and
$M\in {\mathcal C}$. The {\it curvature} of $M$ is defined to be
$$v(M)=\overline{\lim\limits_{n\rightarrow \infty}}
( \ell(M^{\otimes n}))^{\frac{1}{n}}$$
where $\ell(-)$ denotes the length of an object.
\end{definition}

\begin{lemma}
\label{xxlem5.7}
Let ${\mathcal C}$ be a monoidal abelian category
satisfying Definition \ref{xxdef3.1}(a,b,c).
Let $A$ be a finite dimensional weak bialgebra
and $\mathcal{A}$ be $A-\Modfd$.
Let $M$ be
an object in $\mathcal{C}$ or $\mathcal{A}$.
\begin{enumerate}
\item[(1)]
If $M$ is in ${\mathcal C}$, then
\begin{equation}
\notag
\fpv(M)\leq v(M)<\infty.
\end{equation}
\item[(2)]
If $M$ is in ${\mathcal A}$, then
\begin{equation}
\label{E5.7.1}\tag{E5.7.1}
\fpv(M)\leq v(M)\leq \dim M.
\end{equation}
\item[(3)]
If $A=\Bbbk Q$ for some finite acyclic quiver $Q$ with the
tensor defined as in \eqref{E2.1.1}, then
\begin{equation}
\label{E5.7.2}\tag{E5.7.2}
\fpv(M)=v(M)=\max_{i\in Q_0} \{ \dim (M)_{i}\}.
\end{equation}
\item[(4)]
If $\mathcal{C}$ is discrete, then,
for every nonzero object $M\in \mathcal{C}$,
$\fpv(M)$ is positive.
\item[(5)]
Suppose that ${\mathcal C}$ acts on a general abelian
category ${\mathcal A}$ such that the action is
discrete in the sense of Definition \ref{xxdef3.6}.
Then, for every object $M$ in ${\mathcal C}$,
$$1\leq \fpv(M)< \infty.$$
\end{enumerate}
\end{lemma}

\begin{proof}
(1)
Let $\Hom$ denote $\Hom_{\mathcal{C}}$.
Let
$$\alpha:=\max\{\ell(X_i\otimes X_j)\mid {\text{
$X_i$ and $X_j$ are simple}}\},$$
and
$$\beta:=\max\{\dim \End(X_i) \mid {\text{
$X_i$ is simple}}\}.$$
Then, for any objects $X$ and $Y$ in
${\mathcal C}$, we have
\begin{equation}
\label{E5.7.3}\tag{E5.7.3}
\ell(X\otimes Y)\leq \alpha \ell(X)\ell(Y),
\end{equation}
and
\begin{equation}
\label{E5.7.4}\tag{E5.7.4}
\dim \Hom(X, Y)\leq \beta\ell(X)\ell(Y).
\end{equation}
By induction,
$\ell(X^{\otimes n})\leq \alpha^{n-1}
\ell(X)^n$ which implies that
$v(X)\leq \alpha \ell(X)<\infty$.

Given a brick set $\phi=\{X_1,\cdots, X_r\}$, define
$\ell (\phi):=\max\limits_{X\in \phi} \{\ell( X)\}$. By
Lemma \ref{xxlem4.8},
$$\rho(A(\phi, M^{\otimes n}\otimes_{\mathcal{C}}-))
\leq \max\limits_{i=1,\cdots, r} \{\sum_{j=1}^r \dim
\Hom(X_i, M^{\otimes n}\otimes X_j)\}.$$
By \eqref{E5.7.3} and \eqref{E5.7.4}, we have
\begin{eqnarray*}
\dim \Hom(X_i, M^{\otimes n}\otimes X_j)
&\leq & \alpha \beta \ell( X_i) (\ell(M^{\otimes n}) \ell( X_j))\\
&\leq & \alpha \beta (\ell(\phi))^2 (\ell(M^{\otimes n}))\\
&\leq & \alpha \beta (\ell(\phi))^2 (v(M)+\varepsilon)^n)
\end{eqnarray*}
for arbitrary small $\varepsilon>0$ and for $n\gg 0$.
Therefore,
$$\rho(A(\phi, M^{\otimes n}\otimes_{\mathcal{C}}-))
\leq \alpha \beta r (\ell( \phi))^2 (v(M)+\varepsilon)^n),$$
which implies that
\begin{equation}
\label{E5.7.5}\tag{E5.7.5}
\rho(A(\phi, M^{\otimes n}\otimes_{\mathcal{C}}-))^{\frac{1}{n}}
\leq (\alpha \beta r (\dim \phi)^2)^{\frac{1}{n}}(v(M)+\varepsilon),
\end{equation}
for $n\gg 0$. When $n\rightarrow \infty$, the limit of right side of
inequality \eqref{E5.7.5} is $v(M)+\varepsilon$, so $\fpv(M)\leq v(M)+
\varepsilon$ for every small $\varepsilon$. The assertion follows.

(2) It follows from Definition \ref{xxdef1.8} that
\begin{equation}
\label{E5.7.6}\tag{E5.7.6}
\dim M\otimes N\leq (\dim M)(\dim N)
\end{equation}
for all $M,N\in \mathcal{A}$. It is also clear that
\begin{equation}
\label{E5.7.7}\tag{E5.7.7}
\dim \Hom_{\mathcal{A}}(M, N)\leq \dim \Hom_{\Bbbk}(M,N)=(\dim M)
(\dim N).
\end{equation}
By \eqref{E5.7.6}, $\dim M^{\otimes n}
\leq (\dim M)^n$, which implies that
$v(M)\leq \dim M$. Now the assertion follows from
part (1).

(3) Let $\phi=\{S(i)\}$ where $i$ is a vertex of $Q$.
Write $\dim (M)_{i}=d_i$. Then $\rho(A(\phi, M^{\otimes n}))$ is the
integer $d_i^n$ and
$\lim\limits_{n\rightarrow \infty} \rho(A(\phi, M^{\otimes n}))^{\frac{1}{n}}=d_i.$
Hence $\fpv(M)\geq d_i$ for all $i$. It is clear that $v(M)=
\max\{d_i\mid i\in Q_0\}$. Therefore part (1) implies that $\fpv(M)=v(M).$

(4) Suppose $\mathcal{A}$ is discrete. Then
there is a simple object $S$ such that $M\otimes S\neq 0$
and $\Hom_{\mathcal{A}}(S, M\otimes S)\neq 0$. By induction,
one can show that $\Hom_{\mathcal{A}}(S, M^{\otimes n}\otimes S)\neq 0$
for all $n$. Therefore $\fpv(M)\geq 1$.

(5) Using a similar proof of part (1), one sees that
$\fpv(M)< \infty$. Using the proof of part (4),
one can show that
$\fpv(M)\geq 1$. Details are omitted.
\end{proof}

\begin{remark}
\label{xxrem5.8}
\begin{enumerate}
\item[(1)]
Let $\mathcal{C}$ be a monoidal abelian category acting on
an abelian category ${\mathcal A}$. Assume that ${\mathcal C}$ 
satisfies Definition \ref{xxdef3.1}(a,b,c). The action of 
$\mathcal{C}$ on ${\mathcal A}$ is called {\it $\fpv$-positive} if
\begin{enumerate}
\item[(e)]
$\fpv(M)>0$ for every nonzero object
$M$ in $\mathcal{C}$. We say $\mathcal{C}$ is
$\fpv$-positive if the natural action of $\mathcal{C}$
on itself is $\fpv$-positive.
\end{enumerate}
\item[(2)]
By Lemma \ref{xxlem5.7}(5) if an action of $\mathcal{C}$
on ${\mathcal A}$ is discrete, then it is {\it $\fpv$-positive}.
\item[(3)]
Suppose an action of $\mathcal{C}$
on ${\mathcal A}$ is $\fpv$-positive.
Let $\mathcal{C}'$ be a monoidal abelian subcategory
$\mathcal{C}$. Then the induced action of $\mathcal{C}'$
on ${\mathcal A}$ is $\fpv$-positive. In general,
such an action is not discrete.
\item[(4)]
There are other natural examples that the action of $\mathcal{C}$
on ${\mathcal A}$ is not discrete, but $\fpv$-positive, see
below.

If $A$ is a finite-dimensional bialgebra and let
${\mathcal C}=A-\Modfd$, then $\mathcal{C}$
is $\fpv$-positive. Let $0\neq M\in \mathcal{C}$ and
let $S_0$ be a simple submodule of $M$. Then $M\otimes S_0\neq 0$
as $\dim M\otimes S_0=\dim M \dim S_0$ (when $A$ is
a bialgebra). For each $i \geq 1$, we define $S_{i}$
inductively to be a simple module of $M\otimes S_{i-1}$.
So $S_i\subseteq M\otimes S_{i-1}$ for all $i\geq 1$.
Continuing this process, we will obtain a set of simple object
$\Gamma=\{S_0,S_1,S_2,\cdots\}$.
Since $A$ has finite many simples, $|\Gamma|<\infty$. Hence,
there exists $m<n\in \Z^{+}$, such that $S_n\cong S_m$.
For all $i\geq n$, we redefine $S_i$ to be $S_{i-k(n-m)}$
where $k$ is an integer such that $m\leq i-k(m-n)<n$.

By the construction, we have, for all $i\geq m$ and all
$s\geq 0$,
$$\dim \Hom(S_{i+1}, M\otimes S_{i})\geq 1 \text{ and }
\dim \Hom(S_{i+s}, M^{\otimes s}\otimes S_{i})\geq 1.$$
Therefore, by taking the brick set $\phi=\{S_m,S_{m+1},\cdots, S_{n}\}$,
one sees that, for each $s$,
$A(\phi, M^{\otimes s}\otimes-)$ is a non-negative matrix
that contains a permutation matrix. As a consequence,
$\rho(A(\phi, M^{\otimes s}\otimes-))\geq 1$ for all $s\geq 1$,
which implies that $\fpv(M)\geq 1$.
\item[(5)]
In $\Repr(Q)$ where the tensor defined as in \eqref{E2.1.1},
$\fpv(M)$, unlike $\fpd(M)$, is an invariant only dependent on the
the dimension vector of $M$ (which is independent of the orientations
of arrows in the quiver).
\end{enumerate}
\end{remark}

Next we investigate $mtt$-structures on
$D^b(\mathcal{C})$.

\begin{lemma}
\label{xxlem5.9}
Suppose that a monoidal abelian category ${\mathcal C}$
acts on an arbitrary abelian category ${\mathcal A}$.
Let ${\mathcal T}=D^b(\mathcal{C})$. Assume that
\begin{enumerate}
\item[(a)]
the above action is either discrete or $\fpv$-positive,
\item[(b)]
${\mathcal C}$ is hereditary, and
\item[(c)]
$(\mathcal{T}^{\leq 0},\mathcal{T}^{\geq 0})$
is any hereditary $mtt$-structure of deviation zero
on ${\mathcal T}$.
\end{enumerate}
Let $\mathcal{H}$ be the heart of the above $mtt$-structure.
Then the following hold.
\begin{enumerate}
\item[(1)]
\cite[Lemma 2.1]{CR2018}
If $M$ is an indecomposable object in $\mathcal{T}$,
then $M$ is in $\mathcal{T}^{\leq b}\cap
\mathcal{T}^{\geq b}$ for some integer $b$.
\item[(2)]
If $M\in \mathcal{C}$, then $M$ is in the heart $\mathcal{H}$.
\item[(3)]
The $mtt$-structure $(\mathcal{T}^{\leq 0},\mathcal{T}^{\geq 0})$
given in (c) is the canonical $mtt$-structure of ${\mathcal T}$.
\end{enumerate}
\end{lemma}

\begin{proof}

(2,3) We only prove this when the action is discrete.
First we claim that $\fpv(M \odot_{\mathcal T} -)>0$ if
$0\neq M\in {\mathcal C}$. It is clear that
$$\fpv(M \odot_{\mathcal T} -)\geq \fpv(M \odot_{\mathcal C} -).$$
Now the claim follows from Lemma \ref{xxlem5.7}(5).

Let $M$ be an indecomposable object in $\mathcal{C}$.
Then $M\in \mathcal{T}^{\leq b}\cap\mathcal{T}^{\geq b}$ for
some $b$. If $b\neq 0$, then $M^{\otimes n}\in \mathcal{T}^{\leq nb}
\cap \mathcal{T}^{\geq nb}$ by Definition \ref{xxdef5.2}(a,c).
For any fixed brick set $\phi$, $A(\phi, M^{\otimes n}
\odot_{\mathcal T} -)$ is zero for $n\gg 0$ by the hereditary
property of Definition \ref{xxdef5.1}(4). Therefore
$\fpv(M\odot_{\mathcal T} -)=0$. By the first paragraph,
$\fpv(M\odot_{\mathcal T} -)>0$, yielding a contradiction.
Therefore $b=0$, or equivalently, $M\in \mathcal{T}^{\leq 0}
\cap\mathcal{T}^{\geq 0}=:\mathcal{H}$. This implies that ${\mathcal C}
\subseteq {\mathcal H}$. By \cite[Lemma 2.1]{CR2018}
and \cite[Lemma 3.6]{SR2016}, ${\mathcal C} ={\mathcal H}$,
and consequently, the $mtt$-structure $(\mathcal{T}^{\leq 0},
\mathcal{T}^{\geq 0})$ in hypothesis (c) must be the canonical
$mtt$-structure of ${\mathcal T}$.
\end{proof}

The following is basically Theorem \ref{xxthm0.7}.

\begin{theorem}
\label{xxthm5.10}
Let $A$ be a finite dimensional hereditary weak bialgebra.
Suppose that the monoidal abelian category $\mathcal{A}$ is
either discrete or $\fpv$-positive.
\begin{enumerate}
\item[(1)]
There is a unique hereditary $mtt$-structure with deviation
zero on $D^b(\mathcal{A})$.
\item[(2)]
The $\mathcal{A}$ is the heart of any hereditary $mtt$-structure
with deviation zero on $D^b(\mathcal{A})$.
\item[(3)]
The $\mathcal{A}$ is uniquely determined by the monoidal
triangulated structure on $D^b(\mathcal{A})$.
\end{enumerate}
\end{theorem}

\begin{proof}
Let ${\mathcal C}={\mathcal A}$. Then we can easily check
all hypotheses in Lemma \ref{xxlem5.9}. Then part (1) follows
from Lemma \ref{xxlem5.9}(3).

(2,3) Follow directly from part (1).
\end{proof}

\begin{proof}[Proof of Theorem \ref{xxthm0.5}]
If $A$ is a bialgebra, by Remark \ref{xxrem5.8}(4),
$\mathcal{A}$ is $\fpv$-positive. Therefore
the hypothesis of Theorem \ref{xxthm5.10} is
satisfied. Now the assertion follows from the
uniqueness of hereditary $mtt$-structure with
deviation zero in Theorem \ref{xxthm5.10}.
\end{proof}

\begin{proof}[Proof of Corollary \ref{xxcor0.6}]
Let $Q$ and $Q'$ be two quivers such that
$D^b(\Repr(Q))$ and $D^b(\Repr(Q'))$ are
equivalent as monoidal triangulated categories.
By Theorem \ref{xxthm0.5}, this equivalent induces
an equivalence between $\Repr(Q)$ and $\Repr(Q')$.

Recall that $Q$ and $Q'$ are acyclic. For each acyclic
quiver there are finitely many simple representations,
say $\{S_i\}_{i=1}^n$, that are associated to vertices
$\{1,\cdots,n\}$ of the quiver. The correspondence
between those simple representations gives rise to a
bijective map $f: Q_0 \rightarrow Q'_0$. By
\cite[Lemma 2.12, p.84]{ASS2006}, $\dim \Ext^1 (S_i, S_j)$
is the number of arrows from vertex $i$ to vertex $j$.
Therefore the number of arrows from $f(i)$ to $f(j)$ is
the same as that from $i$ to $j$. Thus $Q\cong Q'$.
\end{proof}

\section{Proof of Theorem \ref{xxthm0.3}}
\label{xxsec6}

The proof of Theorem \ref{xxthm0.3} uses several results
about weighted projective lines and takes several pages
in total. The final step of the proof is given
at the end of this section. First we recall some basic
definitions concerning weighted projective lines. Details
can be found in \cite[Section 1]{GL1987}.

For $t\geq 1$, let ${\bf p}:=(p_0,p_1,\cdots,p_t)$ be a
$(t+1)$-tuple of positive integers, called the {\it weight}
or {\it weight sequence}. Let
${\bf D}:=(\lambda_0, \lambda_1,\cdots, \lambda_t)$ be a
sequence of distinct points of the projective line
${\mathbb P}^1$ over $\Bbbk$. We normalize ${\bf D}$
so that $\lambda_0=\infty$, $\lambda_1=0$ and
$\lambda_2=1$ (if $t\geq 2$). Let $R$ denote the commutative
algebra
\begin{equation}
\label{E6.0.1}\tag{E6.0.1}
\Bbbk[X_0,X_1,\cdots,X_t]/(X_i^{p_i}-X_1^{p_1}+\lambda_i X_0^{p_0},
i=2,\cdots,t).
\end{equation}
The image of $X_i$ in $R$ is denoted by $x_i$ for all $i$.
Let ${\mathbb L}$ be the abelian group of rank 1
generated by $\overrightarrow{x_i}$ for $i=0,1,\cdots,t$
and subject to the relations
$$p_0 \overrightarrow{x_0}= \cdots =p_i \overrightarrow{x_i}=\cdots
=p_t \overrightarrow{x_t}=: \overrightarrow{c}.$$
The algebra $R$ is ${\mathbb L}$-graded by setting $\deg x_i=
\overrightarrow{x_i}$. The corresponding
{\it weighted projective line}, 
denoted by ${\mathbb X}({\bf p},{\bf D})$ or simply ${\mathbb X}$,
is a noncommutative space whose category of coherent sheaves is
given by the quotient category
$$coh({\mathbb X}):=\frac{\gr^{\mathbb L}-R}{\gr_{f.d.}^{\mathbb L}-R},$$
see \cite[p.155]{Le2011}.

The weighted projective lines are classified into the following
three classes:
\begin{equation}
\label{E6.0.2}\tag{E6.0.2}
{\mathbb X} \;\; {\rm{is}}\;\;
\begin{cases} domestic \;\; & {\rm{if}} \;\; {\bf p}
\;\; {\rm{is}}\; (p, q), (2,2,n), (2,3,3), (2,3,4), (2,3,5);\\
tubular \;\; & {\rm{if}} \;\; {\bf p}
\;\; {\rm{is}}\; (2,3,6), (3,3,3), (2,4,4), (2,2,2,2);\\
wild \;\; & {\rm{otherwise}}.
\end{cases}
\end{equation}
In \cite[Section 4.4]{Sc2012}, domestic (respectively, tubular,
wild) weighted projective lines are called {\it parabolic}
(respectively, {\it elliptic, hyperbolic}). Let ${\mathbb X}$
be a weighted projective line. A sheaf $F\in coh({\mathbb X})$
is called {\it torsion} if it is of finite length in
$coh({\mathbb X})$. Let $Tor({\mathbb X})$ denote the full
subcategory of $coh({\mathbb X})$ consisting of all torsion
objects. By \cite[Lemma 4.16]{Sc2012}, the category $Tor({\mathbb X})$
decomposes as a direct product of orthogonal blocks
\begin{equation}
\label{E6.0.3}\tag{E6.0.3}
Tor({\mathbb X})=\prod_{x\in {\mathbb P}^1\setminus
\{\lambda_0,\lambda_1,\cdots,\lambda_{t}\}} Tor_{x}
\; \times \; \prod_{i=0}^{t} Tor_{\lambda_i}
\end{equation}
where $Tor_{x}$ is equivalent to the category of nilpotent
representations of the Jordan quiver (with one vertex and
one arrow) over the residue field $\Bbbk_{x}$ and where
$Tor_{\lambda_i}$ is equivalent to the category of nilpotent
representations over $\Bbbk$ of the cyclic quiver of length
$p_i$. A simple object in
$coh({\mathbb X})$ is called {\it ordinary simple} (see
\cite{GL1987}) if it is the skyscraper sheaf ${\mathcal O}_x$ of
a closed point $x\in {\mathbb P}^1\setminus
\{\lambda_0,\lambda_1,\cdots,\lambda_{t}\}$.

Let $Vect({\mathbb X})$ be the full subcategory of
$coh({\mathbb X})$ consisting of all vector bundles. Similar
to the elliptic curve case \cite[Section 4]{BB2007}, one can
define the concepts of {\it degree}, {\it rank} and
{\it slope} of a vector bundle on a weighted projective
line ${\mathbb X}$; details are given in \cite[Section 4.7]{Sc2012}
and \cite[Section 2]{LM1994}.  For each
$\mu\in {\mathbb Q}$, let $Vect_{\mu}({\mathbb X})$ be the
full subcategory of $Vect({\mathbb X})$ consisting of all
semistable vector bundles of slope $\mu$. By convention,
$Vect_{\infty}({\mathbb X})$ denotes $Tor({\mathbb X})$.
By \cite[Comments after
Corollary 4.34]{Sc2012}, every indecomposable object in
$coh({\mathbb X})$ is in
\begin{equation}
\notag
\bigcup_{\mu\in {\mathbb Q}\cup\{\infty\}} Vect_{\mu}({\mathbb X}).
\end{equation}

The {\it dualizing element} of $\mathbb{X}$ is denoted by
\begin{equation}
\label{E6.0.4}\tag{E6.0.4}
\omega_0:= (t-2)\overrightarrow{c}
-\sum_{i=1}^n \overrightarrow{x}_i \in \mathbb{L}.
\end{equation}

Below we collect some nice properties of weighted projective lines.
The definition of a stable tube (or simply tube) was introduced
in \cite{Ri1984}.

\begin{lemma} \cite[Lemma 7.9]{CGWZZZ2017}
\label{xxlem6.1}
Let ${\mathbb X}={\mathbb X}({\bf p}, {\bf D})$ be a weighted
projective line.
\begin{enumerate}
\item[(1)]
$coh({\mathbb X})$ is noetherian and hereditary.
\item[(2)]
$$D^b(coh({\mathbb X})) \cong
\begin{cases}
D^b(\Repr( \widetilde{\mathbb A}_{p, q})) & {\rm{if}}\;\; {\bf p}=(p,q),\\
D^b(\Repr( \widetilde{\mathbb D}_n)) & {\rm{if}}\;\; {\bf p}=(2,2,n),\\
D^b(\Repr( \widetilde{\mathbb E}_6)) & {\rm{if}}\;\; {\bf p}=(2,3,3),\\
D^b(\Repr( \widetilde{\mathbb E}_7)) & {\rm{if}}\;\; {\bf p}=(2,3,4),\\
D^b(\Repr( \widetilde{\mathbb E}_8)) & {\rm{if}}\;\; {\bf p}=(2,3,5).
\end{cases}
$$
\item[(3)]
Let ${\mathcal S}$ be an ordinary simple
object in $coh({\mathbb X})$.
Then $\Ext^1_{\mathbb X}({\mathcal S},{\mathcal S})=\Bbbk$.
\item[(4)]
If ${\mathbb X}$ is tubular or domestic, then $\Ext^1_{\mathbb X}(X,Y)=0$
for all $X\in Vect_{\mu'}({\mathbb X})$ and $Y\in Vect_{\mu}({\mathbb X})$
with $\mu'< \mu$.
\item[(5)]
If ${\mathbb X}$ is domestic, then $\Ext^1_{\mathbb X}(X,Y)=0$
for all $X\in Vect_{\mu'}({\mathbb X})$ and $Y\in Vect_{\mu}({\mathbb X})$
with $\mu'\leq \mu<\infty$.
\item[(6)]
Suppose ${\mathbb X}$ is tubular or domestic.
Then every indecomposable vector bundle ${\mathbb X}$
is semistable.
\item[(7)]
Suppose ${\mathbb X}$ is tubular
and let $\mu\in {\mathbb Q}$.
Then each $Vect_{\mu}({\mathbb X})$ is a uniserial category.
Accordingly indecomposables in $Vect_{\mu}({\mathbb X})$
lies in Auslander-Reiten components, which all are
stable tubes of finite rank. In fact, for every
$\mu\in {\mathbb Q}$, $$Vect_{\mu}({\mathbb X})\cong
Vect_{\infty}({\mathbb X})=Tor({\mathbb X}).$$
\end{enumerate}
\end{lemma}

\begin{lemma}
\label{xxlem6.2}
Let ${\mathbb X}={\mathbb X}({\bf p}, {\bf D})$ be a weighted
projective line.
\begin{enumerate}
\item[(1)]
\cite[Theorem 2.2(ii)]{Le2011}
Let $\mathcal{T}$ be $D^b(coh(\mathbb{X}))$. Then
$\mathcal{T}$ has Serre duality in the form of
$$\Hom_{\mathcal{T}}(X,Y)^{\ast}
\cong \Hom_{\mathcal{T}}(Y,S(X)),$$
where the Serre functor $S$ is $-(\omega_0)[1]$
and where the dualizing element $\omega_0$ is
in \eqref{E6.0.4}.
\item[(2)]
\cite[Proposition 1.10]{LR2006}
Each indecomposable vector bundle has a nonzero
morphism to $Tor_{x}$ for every point $x$ in
$\mathbb{P}^1$.
\end{enumerate}
\end{lemma}

The following linear algebra lemma is needed to estimate
the spectral radius of some matrices.

\begin{lemma}
\label{xxlem6.3}
Let $\Gamma$ be the $n\times n$-matrix $(a_{ij})_{n\times n}$
where
\begin{equation}
\label{E6.3.1}\tag{E6.3.1}
a_{ij}=\begin{cases} 1 & {\text{if $i=1$, or $j=1$,}}\\
0 & {\text{otherwise.}}\end{cases}.
\end{equation}
Then the spectral radius $\rho(\Gamma)\geq \sqrt{n}$.
\end{lemma}

\begin{proof}
It is not hard to check that the characteristic polynomial
of $\Gamma$ is
\[
f(x)=x^n-x^{n-1}-(n-1)x^{n-2}=x^{n-2}(x^2-x-(n-1)).
\]
Then
$$\rho(\Gamma)=\frac{1+\sqrt{4n-3}}{2}\geq \sqrt{n}.$$
\end{proof}

\begin{lemma}
\label{xxlem6.4}
Suppose $\mathcal{T}$ be a triangulated category
satisfying
\begin{enumerate}
\item[(a)]
there is an infinite brick set $\phi$,
\item[(b)]
there is a brick object $B$ in $\mathcal{T}$ such that
$\Hom_{\mathcal{T}}(B,X)\neq 0$ for all $X\in \phi$,
\item[(c)]
there is an integer $m$ such that
$\Hom_{\mathcal{T}}(B[s],X)=\Hom_{\mathcal{T}}(X, B[s])= 0$
for all $X\in \Phi$ and for all $|s|\geq m$,
\item[(d)]
$\mathcal{T}$ has a Serre functor $S$, and
\item[(e)]
there is an integer $m_0$ such that
$\Hom_{\mathcal{T}}(B[m_0], S(X))\neq 0$ for all
$X\in \phi$.
\end{enumerate}
Let $\mathcal{C}$ be a monoidal triangulated category
acting on $\mathcal{T}$. Then there is an object
$M\in \mathcal{C}$ such that $\fpd(M)=\infty$.
\end{lemma}

\begin{proof} In the following proof let $\odot$ denote the
action of $\mathcal{C}$ on $\mathcal{T}$ and $\Hom$ denote
$\Hom_{\mathcal{T}}$.

By condition (d), $\mathcal{T}$ has a Serre functor
$S:\mathcal{T}\to \mathcal{T}$ such that
\begin{equation}
\label{E6.4.1}\tag{E6.4.1}
\Hom(X,Y)^{\ast} \cong \Hom(Y,S(X))
\end{equation}
for all $X,Y$ in $\mathcal{T}$.

Let ${\bf 1}\in \mathcal{C}$ be the unit object with
respect to the monoidal tensor of $\mathcal{C}$. Let $m$
and $m_0$ be the integers given in conditions (c) and (e),
and let $M$ be the object ${\bf 1}[m]\oplus {\bf 1}\oplus {\bf 1}[m_0-m]$ in
$\mathcal{C}$. It is enough to show that $\fpd(M)=\infty$.
Let $\phi_n$ be a brick set consisting of $(n-1)$ objects in
$\phi$ and one extra special object, namely $B[m]$, where $m$
is in condition (c). Write
$$\phi_n=\{X_1:=B[m], X_2, X_3,\cdots,X_n\}$$
where $X_i\in \phi$ for all $i=2,3,\cdots,n$. Let
$A:=(a_{ij})$ denote the adjacency matrix
$A(\phi_n, M\odot -)$.
We claim that $a_{1i}\neq 0$ and $a_{j1}\neq 0$ for all $i,j$.

Case 1:
$$\begin{aligned}
a_{11}&= \dim \Hom(B[m], M\odot B[m])\\
&\geq \dim \Hom(B[m], {\bf 1}\odot B[m])\\
&= \dim \Hom(B, B)\\
&=\dim \Bbbk=1  \qquad\qquad\qquad\qquad {\text{by condition (b)}}.
\end{aligned}
$$

Case 2: for every $i\geq 2$,
$$\begin{aligned}
a_{1i}&= \dim \Hom(B[m], M\odot X_i)\\
&\geq \dim \Hom(B[m], {\bf 1}[m]\odot X_i)\\
&= \dim \Hom(B[m], X_i[m])\\
& \geq \dim \Bbbk=1 \qquad\qquad\qquad\qquad {\text{by condition (c)}}.
\end{aligned}
$$

Case 3: for every $j\geq 2$,
$$\begin{aligned}
a_{j1}&= \dim \Hom(X_j, M\odot B[m])\\
&\geq \dim \Hom(X_j, {\bf 1}[m_0-m]\odot B[m])\\
&= \dim \Hom(X_j, B[m_0])\\
&= \dim \Hom(B[m_0], S(X_j)) \quad\;\; {\text{by \eqref{E6.4.1}}}\\
& \geq \dim \Bbbk=1 \qquad\qquad\qquad\qquad {\text{by condition (e)}}.
\end{aligned}
$$

Therefore we proved the claim. This means that
every entry in $A$ is larger than or equal to
the corresponding entry in $\Gamma$ as given in
Lemma \ref{xxlem6.3}. By linear algebra,
$$\rho(A)\geq \rho(\Gamma)\geq \sqrt{n}$$
where the last inequality is Lemma \ref{xxlem6.3}.
Then, by definition, $\fpd(M)\geq \sqrt{n}$ for all $n$.
Thus $\fpd(M)=\infty$ as desired.
\end{proof}

Now we are ready to show that every monoidal structure
on weighted projective line is $\fpd$-infinite.

\begin{proposition}
\label{xxpro6.5}
Let $\mathbb{X}$ be a weighted projective line and
let $\mathcal{T}$ be $D^b(coh(\mathbb{X}))$.
\begin{enumerate}
\item[(1)]
Let $\mathcal{C}$ be a monoidal triangulated category
acting on $\mathcal{T}$. Then there is an object
$M\in \mathcal{C}$ such that $\fpd(M)=\infty$.
\item[(2)]
Every monoidal structure on $\mathcal{T}$ is
$\fpd$-infinite.
\end{enumerate}
\end{proposition}

\begin{proof} Since part (2) is a special case of
part (1), it suffices to show part (1).
We need to verify hypotheses (a)-(e) in Lemma
\ref{xxlem6.4}.

Let $\phi$ be the set $\{\mathcal{O}_x\mid x\in
\mathbb{P}^1\setminus\{\lambda_0,\cdots,\lambda_t\}\}$
and let $B$ be the trivial bundle $\mathcal{O}_{\mathbb{X}}$.
It is clear that $\phi$ is infinite, so (a) holds.
By Lemma \ref{xxlem6.2}(2), (b) holds. Since $coh(\mathbb{X})$
has global dimension 1, (c) holds. By Lemma
\ref{xxlem6.2}(1), $D^b(coh(\mathbb{X}))$ has a Serre
functor $S$ which is
$\mathcal{O}_{\mathbb{X}}(\omega_0)[1]\otimes_{\mathbb{X}}-$.
Then $S(\mathcal{O}_x)=\mathcal{O}_x[1]$ for all
$x\in \mathbb{P}^1\setminus\{\lambda_0,\cdots,\lambda_t\}$.
Therefore (e) holds. Finally the assertion follows from
Lemma \ref{xxlem6.4}.
\end{proof}

It is not hard to check that Proposition \ref{xxpro6.5} also
holds if $\mathbb{X}$ is an irreducible smooth projective
scheme of dimension at least 1.

We still need quite a few lemmas before we can prove
Theorem \ref{xxthm0.3}. Recall that the definition of
$\fpd$-wild is given in Definition \ref{xxdef0.2}(3).

\begin{lemma}
\label{xxlem6.6}
Let $\mathcal{T}$ be a triangulated category. Suppose that,
for each $n$, there is a connected brick set $\phi$ with
$|\phi|>n$.
\begin{enumerate}
\item[(1)]
Let $\mathcal{C}$ be a $\Hom$-finite Krull-Schmidt
monoidal triangulated category acting on $\mathcal{T}$.
Then there is an indecomposable object $M\in \mathcal{C}$
such that $\fpd(M)=\infty$.
\item[(2)]
Suppose further that $\mathcal{T}$ is $\Hom$-finite
Krull-Schmidt. Then every monoidal triangulated structure
on $\mathcal{T}$ is $\fpd$-wild.
\end{enumerate}
\end{lemma}

\begin{proof} Since part (2) is a special case of
part (1), it suffices to show part (1).

Let $(\mathcal{C},\otimes, \mathbf{1})$ be a monoidal
triangulated category acting on $\mathcal{T}$
where $\mathbf{1}$ is the unit object of $\mathcal{C}$.
Write $\mathbf{1}$ as a direct sum of indecomposable objects
$$\mathbf{1}=\bigoplus_{i=1}^d M_i.$$
By hypothesis, for each $n$, there is a connected brick set
$\phi^n$ with $|\phi^n|>dn$. Define
$$\phi^n_i:=\{ X\in \phi^n \mid M_i\odot X\neq 0\}.$$
Since $X=\mathbf{1}\odot X=\bigoplus_{i=1}^d (M_i\odot X)$
and $X$ is indecomposable, there is exactly one $i$ such that
$M_i\odot X\neq 0$, and for that $i$, we have $M_i\odot X=X$.
Hence, for each $n$, $\phi^n$ is a disjoint union of $\phi^n_i$
for $i=1,\cdots, d$. By the pigeonhole principle, there is at least
$i$ such that $|\phi^n_i|>n$. This implies that there is at least
one $j$ such that, with this fixed $j$, there is an infinite
sequence $n_j$ such that $|\phi^{n_j}_j|>n_j$. Using this
sequence of brick sets, one sees that
$$\Hom_{\mathcal{T}}(M_j[-1]\odot X,Y)=\Hom_{\mathcal{T}}(X,Y[1])
\neq 0$$
for all $X,Y\in \phi^{n_j}_j$.
By definition, $\fpd(M_j[-1])\geq n_j$ as $|\phi^{n_j}_j|\geq n_j$.
Since $n_j$ goes to infinity, $\fpd(M_j[-1])=\infty$ as desired.
\end{proof}

Next we recall more detailed structures concerning
weighted projective lines.
Let ${\bf p}$ be the weight of $\mathbb{X}$ and
$B_0=\gcd(p_i\in {\bf p})$.
Define $\nu$ to be the group homomorphism from $\mathbb{L}$ to
$\mathbb{Z}$ such that
$\nu(\overrightarrow{x}_i)=\prod_{s\neq i} p_s$. It is easy to
see that the image of $\nu$ is $B_0\mathbb{Z}$. In fact,
we can assume that $B_0=1$, so $\nu: \mathbb{L}\to \mathbb{Z}$
is a surjective morphism. Since $\rank(\ker(\nu))=0$,
the kernel of $\nu$ is finite.

\begin{lemma}
\label{xxlem6.7}
Let $\mathbb{X}$ be a weighted projective line and let
$\mathcal{T}$ be $D^b(coh(\mathbb{X}))$.
\begin{enumerate}
\item[(1)]
There is a positive integer $B_1$, only dependent on
$\mathbb{X}$, such that, if $\omega_1,\omega_2$ are in
$\mathbb{L}$ satisfying $\nu(\omega_2-\omega_1)\geq B_1$,
then $\Hom_{\mathbb{X}}(\mathcal{O}_{\mathbb{X}}(\omega_1),
\mathcal{O}_{\mathbb{X}}(\omega_2)) \neq 0$.
\item[(2)]
For every $N$, there is a positive integer $B_3(N)$,
only dependent on $\mathbb{X}$ and $N$ such that
$$\dim
\Hom_{\mathbb{X}}(\mathcal{O}(\omega_1),\mathcal{O}(\omega_2))
\leq B_3(N)$$
for all $\omega_1,\omega_2$ in
$\mathbb{L}$ satisfying $0\leq \nu(\omega_2-\omega_1)\leq N$.
\end{enumerate}
\end{lemma}

\begin{proof} (1) We may assume that $\omega_1=0$. Let
$B_1=(t-1)\prod_{s=0}^t p_i$. For $\omega_2 \in \mathbb{L}$
with $\nu(\omega_2)\geq B_1$, write $\omega_2=\sum_{s=0}^{t-1} a_s
\overrightarrow{x}_s + a_t \overrightarrow{x}_t$ where
$0\leq a_s\leq p_s$ for all $0\leq s\leq t-1$. Since
$\nu(\omega_2)\geq B_1$, $a_t\geq 0$. Then the $\omega_2$-degree
component of $R$ (see \eqref{E6.0.1}) is not zero and hence
$\Hom_{\mathbb{X}}(\mathcal{O}_{\mathbb{X}},
\mathcal{O}_{\mathbb{X}}(\omega_2))=R_{\omega_2}\neq 0$.

(2) Again we can assume that $\omega_1=0$. Since there are
only finitely many $\omega_2$ such that $\nu(\omega_2)$ is
in between $0$ and $N$. Let $B_3(N)$ be the
maximum of all possible
$$\dim \Hom_{\mathbb{X}}(\mathcal{O},\mathcal{O}(\omega_2))$$
where $\omega_2$ runs over all $\omega_2\in \mathbb{L}$
such that $0\leq \nu(\omega_2)\leq N$. Then the
assertion follows.
\end{proof}

The next lemma concerns domestic weighted projective lines.
Some un-defined terms can be found in \cite{KLM2013}. Let
$\omega_0$ be the dualizing element defined in \eqref{E6.0.4}.

\begin{lemma}
\label{xxlem6.8}
Let $\mathbb{X}$ be a weighted projective line.
\begin{enumerate}
\item[(1)]
\cite[Proposition 5.1(ii)]{KLM2013}
Suppose that the weight ${\bf p}$ is either $(2, 2, n)$, or
$(2, 3, 3)$, or $(2, 3, 4)$ or $(2, 3, 5)$. Let $\Delta$
be the attached Dynkin diagram and $\widetilde{\Delta}$
its extended Dynkin diagram. The Auslander-Reiten quiver
$\Gamma(Vect(\mathbb{X}))$ of $Vect(\mathbb{X})$ consists
of a single standard component having the form
$\mathbb{Z} \widetilde{\Delta}$. Moreover, the category of
indecomposable vector bundles on $\mathbb{X}$, denoted by
$ind(Vect(\mathbb{X}))$, is equivalent to the mesh category
of $\Gamma(Vect(\mathbb{X}))$.
\item[(2)]
Under the hypotheses of part {\rm{(1)}}, there is a finite
set of indecomposable vector bundles $\{V_i\}_{i\in I}$ such
that every indecomposable vector bundle is of the form $V_i(n
\omega_0)$ for some $n\in \mathbb{Z}$ and some $i\in I$.
\item[(3)]
\cite[Sect. 5.1, page 217]{KLM2013}
If the weight ${\bf p}$ is of the form $(p,q)$, then each
indecomposable vector bundle is a line bundle $\mathcal{O}(\omega)$
for $\omega\in \mathbb{L}$.
\item[(4)]
Under the hypotheses of part {\rm{(3)}}, there is a finite set of
indecomposable vector bundles $\{V_i\}_{i\in I}$ such that
every indecomposable vector bundle is of the form $V_i(n
\omega_0)$ for some $n\in \mathbb{Z}$.
\end{enumerate}
\end{lemma}

\begin{proof}
(2) There is a ($[-1]$-shifted) Serre functor $F:=-(\omega_0)$ which
is also a functor from $ind(Vect(\mathbb{X}))$ to itself. It is easy
to check that $\nu(\omega_0)<0$. Then $F$ induces an automorphism of the
Auslander-Reiten quiver $\Gamma(Vect(\mathbb{X}))$ by shifting
forward a distance $\nu(\omega_0)$.
Therefore there is a finite set of indecomposable vector bundles
$\{V_i\}_{i\in I}$ such that every indecomposable vector bundle
is of the form $V_i(n\omega_0)$ for some $n\in \mathbb{Z}$
and some $i\in I$.


(4) Since the map $\nu: \mathbb{L}
\to \mathbb{Z}$ is a group homomorphism with finite
kernel, there are only finitely many
$\omega$ such that $\nu(\omega)=0$.
Similarly, there are only finitely many $\omega\in \mathbb{L}$
such that $\nu(\omega)=0,1,\cdots,-\nu(\omega_0)-1$. Then
the set $\{\mathcal{O}(\omega)\mid 0\leq \nu(\omega)\leq -\nu(\omega_0)-1\}$
has the desired property.
\end{proof}

We introduce some temporary notation. By Lemma
\ref{xxlem6.8}(2,4), if $\mathbb{X}$ is domestic,
then there is a finite set of indecomposable vector
bundles, say $\mathbb{K}:=\{K_1, \cdots K_{B_4}\}$, such that
every indecomposable vector bundle is of the form
$K_s(n\omega_0)$ for some $1\leq s\leq B_4$ and
some $n\in \mathbb{Z}$. (Here $\omega_0 \in \mathbb{L}$ is
the dualizing element given in \eqref{E6.0.4}.) For
each $K_s$ we fix a sequence of sub-bundles
\begin{equation}
\label{E6.8.1}\tag{E6.8.1}
0=:V_{s,0}\subset V_{s,1}\subset V_{s,2}\subset \cdots
\subset V_{s,Y_s}:=K_s
\end{equation}
such that each subquotient $V_{s,i}/V_{s,i-1}$ is a line
bundle of the form $\mathcal{O}_{\mathbb{X}}(\omega_{s,i})$
for some $\omega_{s,i}\in \mathbb{L}$. Let $\Omega(\mathbb{X})$
be the collection of all such $\omega_{s,i}$'s. Hence
$\Omega(\mathbb{X})$ is finite. Let
$$\begin{aligned}
\max(\Omega)&= \max\{\nu(\omega)\mid \omega\in \Omega(\mathbb{X})\},\\
\min(\Omega)&= \min\{\nu(\omega)\mid \omega\in \Omega(\mathbb{X})\}.
\end{aligned}
$$
For every vector bundle $V$, we write $V=K_s(n\omega_0)$ for some
$s$ and $n$. Then we fix a sequence of sub-bundles of
$V:=K_s(n\omega_0)$ by applying $-(n\omega_0)$ to \eqref{E6.8.1}.
We have a series of subquotients
$$V_{s,i}(n\omega_0)/V_{s,i-1}(n\omega_0)
\cong \mathcal{O}_{\mathbb{X}}(\omega_{s,i}+n\omega_0)$$
induced by \eqref{E6.8.1}. Let $\nu(V)$ denote the positive
different between the largest of all $\nu(\omega_{s,i}+n\omega_0)$
and the smallest of all $\nu(\omega_{s,i}+n\omega_0)$. Then it is
clear that $\nu(V)\leq \max(\Omega)-\min(\Omega)$. So we have
proved part (1) of the follows proposition.

\begin{proposition}
\label{xxpro6.9}
Let $\mathbb{X}$ be a domestic weighted projective line.
\begin{enumerate}
\item[(1)]
Let $V$ be an indecomposable vector bundle on $\mathbb{X}$.
Then the $\nu(V)$ is uniformly bounded by
$B_5:=\max(\Omega)-\min(\Omega)$.
\item[(2)]
Let $V$ be an indecomposable vector bundle on $\mathbb{X}$.
Then the rank $V$ is uniformly bounded by an integer $B_6$
{\rm{(}}only dependent on $\mathbb{X}${\rm{)}}.
\item[(3)]
Suppose $\phi$ is a brick set consisting of vector bundles
on $\mathbb{X}$. Then the size of $\phi$ is uniformly bounded
by $B_7$ {\rm{(}}only dependent on $\mathbb{X}${\rm{)}}.
\item[(4)]
Suppose $\phi$ is a brick set consisting of vector bundles
on $\mathbb{X}$. Then, up to a degree shift, $\phi$ is a
subset of $\bigcup_{n=-N}^{N} \mathbb{K}(n\omega_0)$ for some
integer $N$. As a consequence, $\sum_{V\in \phi} \nu(V)$ is
uniformly bounded, say, by $B_8$ {\rm{(}}only dependent on
$\mathbb{X}${\rm{)}}.
\item[(5)]
Fix a vector bundle $V$ on $\mathbb{X}$. For every
brick set consisting of vector bundles
$\{X_1,\cdots,X_n\}$,
$\dim \Hom_{\mathbb{X}}(X_i, V\otimes_{\mathbb{X}} X_j)$
is uniformly bounded by $B_9(V)$ for all $i,j$
{\rm{(}}only dependent on $V$ and $\mathbb{X}${\rm{)}}.
\item[(6)]
Fix a vector bundle $V$ on $\mathbb{X}$. For every
brick set consisting of vector bundles
$\{X_1,\cdots,X_n\}$,
$\dim \Hom_{\mathbb{X}}(V\otimes_{\mathbb{X}} X_i,  X_j)$
is uniformly bounded by $B_{10}(V)$ for all $i,j$
{\rm{(}}only dependent on $V$ and $\mathbb{X}${\rm{)}}.
\end{enumerate}
\end{proposition}

\begin{proof}
(2) This is part of \cite[Theorem 6.1]{LR2006}. It also can be shown
directly as follows.

Since every indecomposable vector bundle $V$ is of the form
$K_s(\omega)$ for $1\leq s\leq B_4$, the rank of $V$ is uniformly
bounded, say by $B_6$.

(3) Since $\nu(\omega_0)$ is negative, there is an $N_1$ such that
for all $n\geq N_1$ and for all $s_1,s_2$,
$$\nu(\omega_{s_2,Y_{s_2}})-\nu(\omega_{s_1,1}-n \omega_0)\geq B_1$$
where $B_1$ is the constant given in Lemma \ref{xxlem6.7}(1).
By Lemma \ref{xxlem6.7}(1), for such $n$, $s_1,s_2$,
$$\Hom_{\mathbb{X}}(\mathcal{O}_{\mathbb{X}}(\omega_{s_2,Y_{s_2}})),
\mathcal{O}_{\mathbb{X}}(\omega_{s_1,1}-n\omega_0))\neq 0.$$
By \eqref{E6.8.1},
\begin{equation}
\label{E6.9.1}\tag{E6.9.1}
\Hom_{\mathbb{X}}(K_{s_2},K_{s_1}(-n\omega_0))\neq 0
\end{equation}
for all $s_1,s_2$ and all $n\geq N_1$.

Let $\phi$ be a brick set of vector bundles. We claim that
$|\phi|\leq N_1|\mathbb{K}|=:B_7$. If not, by the pigeonhole
principle, there is an $s$ such that $\phi$ contains
a subset
$$\{K_s(n_1\omega_0),\cdots,K_s(n_q\omega_0)\}$$
for some $q> N_1$ where $n_1<n_1<\cdots <n_q$. Then,
by \eqref{E6.9.1},
$$
\Hom_{\mathbb{X}}(K_s(n_q\omega_0),K_s(n_1\omega_0))
=\Hom_{\mathbb{X}}(K_s,K_s((n_1-n_q)\omega_0))
\neq 0.$$
This contradicts that $\phi$ is a brick set. Therefore
we proved the claim.

(4) Without loss of generality, we may assume that
$\phi$ contains $K_1$. Let $K_s(n\omega_0)$ be
any other object in $\phi$. By \eqref{E6.9.1},
$|n|< N_1$ where $N_1$ is given in the proof of
part (3). Therefore $\phi$ is a subset of
$\bigcup_{n=-N_1}^{N_1} \mathbb{K}(n\omega_0)$.
As a consequence, $\sum_{X\in \phi} \nu(X)$ is
uniformly bounded, say by $B_8$.

(5) By part (4), up to a degree shift, we can assume that
$\phi$ is a subset of $\bigcup_{n=-N}^{N} \mathbb{K}
(n\omega_0)$ for a fixed integer $N$. Note that
the global degree shift will not change the assertion.
Then the assertion follows by the fact that
$\bigcup_{n=-N}^{N} \mathbb{K}(n\omega_0)$ is a fixed
set.

(6) Similar to the proof of part (5).
\end{proof}

\begin{lemma}
\label{xxlem6.10}
Let $\mathbb{X}$ be a weighted projective line.
Let $\mathcal{T}$ be $D^b(coh(\mathbb{X}))$.
\begin{enumerate}
\item[(1)]
Let $M$ be a brick object in $\mathcal{T}$.
Then $M\cong N[n]$ where $n\in \mathbb{Z}$ and there
$N\in coh(\mathbb{X})$ is either a vector bundle, or
an ordinary simple $\mathcal{O}_x$, or an indecomposable
object in $Tor_{\lambda_i}$.
\item[(2)]
If a brick set $\phi$ consists of indecomposable
objects in $Tor_{\lambda}$ for some $\lambda\in \mathbb{P}^1$,
then $|\phi|$ is uniformly bounded by $B_{11}$
{\rm{(}}only dependent on $\mathbb{X}${\rm{)}}.
\item[(3)]
If $M$ is a brick object in $Tor(\mathbb{X})$, then
$\dim M$ is uniformly bounded by $B_{12}$
{\rm{(}}only dependent on $\mathbb{X}${\rm{)}}.
\end{enumerate}
\end{lemma}

\begin{proof}
(1) It is well-known that every indecomposable object in
$coh(\mathbb{X})$ is either a vector bundle or a torsion
sheaf. The assertion follows by \eqref{E6.0.3} and the fact that
$coh(\mathbb{X})$ is hereditary.

(2) This is trivial if $\lambda \in \mathbb{P} \setminus
\{\lambda_0,\cdots,\lambda_t\}$. If $\lambda=\lambda_i$ for some $i$,
$Tor_{\lambda_i}$ is a standard tube of
rank $p_i$ with $p_i^2$ brick objects, see
\cite[Section 2.2]{CGWZZZ2019}. So the assertion follows.

(3) By \eqref{E6.0.3}, $M\in Tor_{\lambda}$ for
some $\lambda \in \mathbb{P}$. It is trivial if
$\lambda \in \mathbb{P} \setminus
\{\lambda_0,\cdots,\lambda_t\}$. Now assume that $\lambda=\lambda_i$.
All brick objects in $Tor_{\lambda_i}$ are given in
\cite[Corollary 2.8]{CGWZZZ2019}. As a consequence,
$\dim M\leq p_i$. The assertion follows.
\end{proof}

Since $R$ in \eqref{E6.0.1} is commutative, there is a
natural tensor product on $coh(\mathbb{X})$, denoted
by $\otimes_{\mathbb{X}}$. Note that $\otimes_{\mathbb{X}}$
is not (bi)exact. The derived category $\mathcal{T}:=
D^b(coh(\mathbb{X}))$ has a canonical monoidal structure
where the tensor functor is defined by
$$ - \otimes_{\mathcal{T}} - : = -\otimes_{\mathbb{X}}^L -$$
(the derived tensor product). Note that
$\otimes_{\mathcal{T}}$ is biexact so that $\mathcal{T}$ is
a monoidal triangulated category. Next we show that
this monoidal triangulated structure is $\fpd$-tame
when $\mathbb{X}$ is domestic.

\begin{theorem}
\label{xxthm6.11}
Retain the notation introduced above. If $\mathbb{X}$
is domestic, then the canonical monoidal triangulated
structure on $D^b(coh(\mathbb{X}))$ is
$\fpd$-tame.
\end{theorem}

\begin{proof} Let $\mathcal{T}$ denote
$D^b(coh(\mathbb{X}))$. By Proposition \ref{xxpro6.5},
$\mathcal{T}$ is $\fpd$-infinite. By definition, it
remains to show that $\fpd(M)<\infty$ for every
indecomposable object $M$ in $\mathcal{T}$.

Since $M$ is indecomposable and $coh(\mathbb{X})$ is
hereditary, by \cite[Lemma 3.3]{CGWZZZ2017}, $M$ is of
the form $N[n]$ for some $N\in coh(\mathbb{X})$ and
$n\in \mathbb{Z}$. By Lemma \ref{xxlem6.10}(1), $N$ is
either a vector bundle or a torsion. So we fix an $N$
and consider the following two cases.

Case 1: $N$ is a vector bundle. In this case
$N\otimes_{\mathbb{X}}-$ is exact and
$N\otimes_{\mathcal{T}} Y=N\otimes_{\mathbb{X}} Y$
for all $Y\in coh(\mathbb{X})$.

If $n\neq 0,1$, by the proof of Lemma \ref{xxlem4.11},
$\fpd(N[n]\otimes_{\mathcal{T}}-)=0$. Now we deal with the case
$n=0$ or $M=N$.
Let $\phi$ be a brick set. By Lemma \ref{xxlem6.10}(1),
we can write $\phi=\bigcup_{\delta \in \mathbb{Z}} \phi_{\delta}$,
with $\delta$ integers ranging from small to large, where
$\phi_{\delta}$ is either empty or of the form
$$\{X_{\delta,1}[\delta],X_{\delta,2}[\delta],\cdots,
X_{\delta,t_\delta}[\delta]\}$$
for some $X_{\delta,s}\in coh(\mathbb{X})$. Since
$$\Hom_{\mathcal{T}}(X_{\delta,s}[\delta],N\otimes_{\mathcal{T}}
X_{\delta',s'}[\delta'])=0$$
for all $\delta>\delta'$, the adjacency matrix
$A(\phi, N\otimes_{\mathcal{T}}-)$ is a upper triangular
block matrix. Now the idea of \cite[Lemma 6.1]{CGWZZZ2017} implies
that we only need to consider blocks, namely, we can assume that $\phi=
\phi_{\delta}$ for some $\delta$. For each block associated to
$\phi_{\delta}$, we can further assume that $\delta=0$ and
$\phi_0=\{X_1,\cdots,X_t\}$ for some $X_s\in coh(\mathbb{X})$.
Without loss of generality, we assume that
$$\phi=\phi_0=\{X_1,\cdots,X_t\}$$
for some $X_1,\cdots, X_t\in coh(\mathbb{X})$. If $\phi$ contains
an ordinary simple $\mathcal{O}_x$, then, by Lemma \ref{xxlem6.2}(2),
$\phi$ does not contain any vector bundle. In this case, one can
further decompose $\phi$ according to \eqref{E6.0.3} so that
$A(\phi, N\otimes_{\mathcal{T}}-)$ is a block diagonal matrix.
For each block, $\phi$ is either $\{\mathcal{O}_x\}$ or consisting of
objects in $Tor_{\lambda_i}$. So we consider these two subcases.
If $\phi=\{\mathcal{O}_x\}$, it is easy to see that
$\Hom_{\mathcal{T}}(\mathcal{O}_x, N\otimes_{\mathcal{T}}
\mathcal{O}_x)$ has dimension bounded by the rank of $N$.
This is uniformly bounded. If $\phi$ is a subset of
$Tor_{\lambda_i}$, then there are only finitely many
possibilities [Lemma \ref{xxlem6.10}(2)]. Hence entries and size of
the $A(\phi, N\otimes_{\mathcal{T}}-)$ is uniformly bounded.
Therefore $\rho(A(\phi, N\otimes_{\mathcal{T}}-))$ is
uniformly bounded. The second case is when $\phi$
does not contain any ordinary simple $\mathcal{O}_x$.
Then the size of $\phi$ is uniformly bounded by
Proposition \ref{xxpro6.9}(3) and Lemma \ref{xxlem6.10}(2).
We claim that each entry in $A(\phi, N\otimes_{\mathcal{T}}-)$
is uniformly bounded, or $d_{ij}:=\dim \Hom_{\mathbb{X}}(X_i,N
\otimes_{\mathbb{X}} X_j)$ is uniformly
bounded for all $X_i, X_j$ in $\phi$. If both
$X_i$ and $X_j$ are vector bundles, the assertion follows
from Proposition \ref{xxpro6.9}(5). If $X_i$ is in $Tor_{\lambda_i}$
and $X_j$ is a vector bundle, then $d_{ij}=0$. If
$X_i$ is a vector bundle and $X_j$ is in $Tor_{\lambda_i}$,
then $d_{ij}$ is bounded by $\rank(X_i)\rank(N)\dim X_j$,
which is uniformly bounded by Proposition \ref{xxpro6.9}(2)
and Lemma \ref{xxlem6.10}(3). If $X_i$ and $X_j$ are
both in  $Tor_{\lambda_i}$, then $d_{ij}$ is bounded by
$(\dim X_i)\rank(N) (\dim X_j)$ which is uniformly bounded.
Combining all these cases, one proves that $\fpd(N)$ is finite
by Lemma \ref{xxlem4.8} (Gershgorin Circle Theorem).

Next we deal with the case $n=1$ (namely, $M=N[1]$) and re-cycle
some notation used in the previous paragraphs. By Lemma
\ref{xxlem6.10}(1), we can write
$\phi=\bigcup_{\delta \in \mathbb{Z}} \phi_{\delta}$,
with $\delta$ being integers ranging from small to large, where
$\phi_{\delta}$ is either empty or of the form
$\{X_{\delta,1}[\delta],X_{\delta,2}[\delta],\cdots,
X_{\delta,t_\delta}[\delta]\}$.
Since $coh(\mathbb{X})$ is hereditary,
$$\Hom_{\mathcal{T}}(X_{\delta,s}[\delta],N[1]\otimes_{\mathcal{T}}
X_{\delta',s'}[\delta'])=0$$
for all $s,s'$ and all $\delta<\delta'$. Therefore the adjacency
matrix $A(\phi, N[1]\otimes_{\mathcal{T}}-)$ is a lower triangular
block matrix. For each block we can assume that $\delta=0$
and $\phi=\{X_1,\cdots,X_t\}$ as in the case $n=0$. If $\phi$
contains an  ordinary simple $\mathcal{O}_x$, then, by
Lemma \ref{xxlem6.2}(2), $\phi$ does not contain any
vector bundle. In this case, one can further decompose $\phi$
according to \eqref{E6.0.3} so that
$A(\phi, N[1]\otimes_{\mathcal{T}}-)$ is a block
diagonal matrix. For each block, $\phi$ is either
$\{\mathcal{O}_x\}$ or consisting of objects in $Tor_{\lambda_i}$.
So we consider these two subcases. If $\phi=\{\mathcal{O}_x\}$,
then
$$\Hom_{\mathcal{T}}(\mathcal{O}_x, N[1]\otimes_{\mathcal{T}}
\mathcal{O}_x)=\Ext^1_{\mathbb{X}}(\mathcal{O}_x, N\otimes_{\mathcal{T}}
\mathcal{O}_x)$$
which is bounded by the $\rank(N)$. If $\phi$ is a subset of
$Tor_{\lambda_i}$, then there are only finitely many
possibilities, see the proof of Lemma \ref{xxlem6.10}(2).
Hence the entries and the size of the $A(\phi, N[1]\otimes_{\mathcal{T}}-)$
are uniformly bounded. Therefore
$\rho(A(\phi, N[1]\otimes_{\mathcal{T}}-))$ is uniformly bounded.
The second case is when $\phi$ does not contain any ordinary
simple $\mathcal{O}_x$. Then the size of $\phi$ is uniformly
bounded by Proposition \ref{xxpro6.9}(3) and Lemma
\ref{xxlem6.10}(2). We claim that each entry in
$A(\phi, N[1]\otimes_{\mathcal{T}}-)$
is uniformly bounded, or
$$\begin{aligned}
d_{ij}:&=\dim \Hom_{\mathbb{X}}(X_i,N[1]\otimes_{\mathbb{X}} X_j)
=\dim \Ext^1_{\mathbb{X}}(X_i,N\otimes_{\mathbb{X}} X_j)\\
&=\dim \Hom_{\mathbb{X}}(N\otimes_{\mathbb{X}} X_j, X_i(\omega_0))
=\dim \Hom_{\mathbb{X}}(N(-\omega_0)\otimes_{\mathbb{X}} X_j, X_i)
\end{aligned}
$$
is uniformly bounded for all $X_i, X_j$ in $\phi$. Note that
the third equality is Serre duality. If both
$X_i$ and $X_j$ are vector bundles, the assertion follows
and Proposition \ref{xxpro6.9}(6). If $X_i$ is in $Tor_{\lambda_i}$
and $X_j$ is a vector bundle, we obtain
that
$$d_{ij}\leq \rank(X_j)\rank(N(-\omega_0))\dim X_i,$$
which is uniformly bounded by Proposition \ref{xxpro6.9}(2)
and Lemma \ref{xxlem6.10}(3). If
$X_i$ is a vector bundle and $X_j$ is in $Tor_{\lambda_i}$,
then $d_{ij}=0$. If $X_i$ and $X_j$ are
both in  $Tor_{\lambda_i}$, then
$$d_{ij}\leq \dim(X_j)\rank(N(-\omega_0))\dim X_i,$$
which is uniformly bounded.
Combining all these cases, one proves that $\fpd(N[1])$ is finite
by Lemma \ref{xxlem4.8} (Gershgorin Circle Theorem).

Case 2: $N$ is a torsion. By definition,
$N\otimes_{\mathcal{T}}-=N\otimes_{\mathbb{X}}^L -$.
If $n\neq -1, 0, 1$, a proof similar to Lemma \ref{xxlem4.11}(1)
shows that $\fpd(N[n])=0$. We need to analyze the cases
$n=-1,0,1$. The following proof is independent of $n$.

Since $N$ is torsion and indecomposable, by \eqref{E6.0.3},
$N$ is either in $Tor_{x}$ or $\Tor_{\lambda_i}$. We will use
Gershgorin Circle Theorem [Lemma \ref{xxlem4.8}].
Let $\phi=\{X_1,\cdots,X_m\}$ be any brick set in
$\mathcal{T}$ and let $(d_{ij})_{m\times m}$ denote
the adjacency matrix $A(\phi, N[n]\otimes_{\mathcal{T}}-)$
where
$$d_{ij}=\dim \Hom_{\mathcal{T}}(X_i, N[n]\otimes_{\mathcal{T}} X_j).$$
By Lemma \ref{xxlem4.8}, it suffices to show
\begin{enumerate}
\item[(a)]
each $d_{ij}$ is uniformly bounded (only dependent on $M:=N[n]$).
\item[(b)]
For each $j$, there are only uniformly-bounded-many $i$ such that
$d_{ij}\neq 0$.
\end{enumerate}

\noindent
{\bf Proof of (a):} For each $j$, write $X_j=Y_j[s_j]$ for some
$Y_j\in coh(\mathbb{X})$ and $s_j\in \mathbb{Z}$. Since
$N\in Tor_{\lambda}$, $H^s_N(X_j):=H^s(N[n]\otimes_{\mathcal{T}} X_j)$
is zero for $s\neq n+s_j-1,n+s_j$ and $H^s_N(X_j)$ is in
$Tor_{\lambda}$ for $s=n+s_j-1,n+s_j$. Since $coh(\mathbb{X})$ is
hereditary,
$$N[n]\otimes_{\mathcal{T}} X_j
=\sum_{s} H^s(N[n]\otimes_{\mathcal{T}} X_j)[-s],$$
see \cite[Lemma 2.1]{CR2018}. If $Y_j$ is a vector
bundle, then
$$\dim H^s_N(X_j)\leq (\dim N)(\rank(Y_j))$$
for all $s$. If $X_j$ is torsion, then
$$\dim H^s_N(X_j)\leq (\dim N)(\dim Y_j)$$
for all $s$. In both cases, $\dim H^s(X_j)$ is uniformly
bounded by Proposition \ref{xxpro6.9}(2) and Lemma
\ref{xxlem6.10}(3). Using the Serre duality and Proposition
\ref{xxpro6.9}(2) and Lemma \ref{xxlem6.10}(3) again, one sees that
$$\sum_{s,t\in \mathbb{Z}}\dim \Hom_{\mathcal{T}}(X_i[t], H^s(N[n]
\otimes_{\mathcal{T}}X_j)[s])=
\sum_{s,t}\dim \Hom_{\mathcal{T}}(X_i[t], H^s_N(X_j)[s])$$
is uniformly bounded.  Hence
$$d_{ij}=\Hom_{\mathcal{T}}(X_i, N[n]\otimes_{\mathcal{T}} X_j)=
\sum_{s}\Hom_{\mathcal{T}}(X_i, H^s(N[n]\otimes_{\mathcal{T}} X_j)[s])$$
is uniformly bounded.

\noindent
{\bf Proof of (b):} As noted before, $\fpd(N[n])=0$ when
$n\neq -1,0,1$. So, in this proof, we assume that $n$ is
$-1$ or $0$ or $1$. Without loss of generality, we only prove
that there are only uniformly-bounded-many $i$ such that
$d_{i1}\neq 0$. By a complex shift, we can assume that
$X_1\in coh(\mathbb{X})$. Since $coh(\mathbb{X})$ is hereditary,
one can check that, if $X_i\in coh(\mathbb{X})[m]$ for $|m|\geq 3$,
then $d_{i1}=0$.

For each $m$ with $|m|\leq 2$, let $\phi_m$ consist of
$Y_i\in coh(\mathbb{X})$ such that $X_i=Y_i[m]\in \phi$
and $d_{i1}\neq 0$. If $\phi_m$ does not contain any
ordinary simple $\mathcal{O}_x$, then, by
Proposition \ref{xxpro6.9}(3) and Lemma 6.10(2), $|\phi_m|$
is uniformly bounded. If $\phi_m$ contain an
ordinary simple $\mathcal{O}_x$, then $d_{i1}\neq 0$
implies that $x$ is in the support of $N\otimes_{\mathbb{X}}
X_1$. Therefore there are only finitely many
possible $x$. Further, $X_1$ is either $\mathcal{O}_x$
or a vector bundle, and in the latter case, $d_{i1}\neq 0$
implies that $N$ must be $\mathcal{O}_x$.
In both case, $N[n]\otimes_{\mathcal{T}} X_1$ is
supported at $x$. Therefore $\phi_m$ consists of
a single element $\mathcal{O}_x$. Combining above, we obtain
that $\sum_{|m|\leq 2} |\phi_m|$ is uniformly bounded. As a consequence,
(b) holds.

Now it follows by Lemma \ref{xxlem4.8}, $\fpd(N[n])<\infty$.
Combining Cases 1 and 2, we finish the proof.
\end{proof}

Now we are ready to prove Theorem \ref{xxthm0.3}.

\begin{proof}[Proof of Theorem \ref{xxthm0.3}]
(1) If $Q$ is of finite type, by
Corollary \ref{xxcor4.10} every monoidal
triangulated structure on $D^b(\Repr(Q))$ is
fpd-finite. The converse follows from
Lemmas \ref{xxlem6.1}(2), \ref{xxlem4.3} and \ref{xxlem4.5}
and Proposition \ref{xxpro6.5}.

(2) Suppose $Q$ is tame. By Lemma
\ref{xxlem6.1}(2) and Theorem \ref{xxthm6.11},
there is a $\fpd$-tame monoidal structure on
$\mathcal{T}$. Applying Lemma \ref{xxlem4.6}
to $\mathcal{A}=\mathcal{C}=\Repr(Q)$,
there is a $\fpd$-wild monoidal structure.

(3) This follows from parts (1,2), Lemmas \ref{xxlem4.3} and
\ref{xxlem6.6}.

(4) This follows from part (1).
\end{proof}

\begin{corollary}
\label{xxcor6.12}
Let $Q$ be a finite acyclic quiver.
\begin{enumerate}
\item[(1)]
$Q$ is of finite type if and only if $\Repr(Q)$ does not
contain an infinite brick set.
\item[(2)]
$Q$ is of tame type if and only if $\Repr(Q)$
contains an infinite brick set and does not contain an
infinite connected brick set.
\item[(3)]
$Q$ is of wild type if and only if $\Repr(Q)$
contains an infinite connected brick set.
\end{enumerate}
\end{corollary}

\begin{proof} (1) If $Q$ is of finite type,
$\Repr(Q)$ contains only finitely many indecomposable
objects. So $\Repr(Q)$ does not contains an infinite
brick set.

For the converse, we assume that $\Bbbk Q$ is of tame or
wild type. By Lemmas \ref{xxlem4.3} and \ref{xxlem4.5},
$\Repr(Q)$ contains an infinite brick set. This yields
a contradiction. Therefore the assertion follows.

(3) If $Q$ is of wild type, by Lemmas \ref{xxlem4.3},
$\Repr(Q)$ contains an infinite connected brick set.
Conversely suppose $\Repr(Q)$ contains an infinite connected
brick set. By Lemma \ref{xxlem6.6}, every monoidal
triangulated structure on $D^b(\Repr(Q))$ is
$\fpd$-wild. By Theorem \ref{xxthm0.3}(3),
$Q$ is of wild type.

(2) Follows from parts (1,3).
\end{proof}

\section{Examples}
\label{xxsec7}

The natural construction of weak bialgebras associated
to quivers, given in Lemma \ref{xxlem2.1}, produces many
monoidal triangulated categories by Lemma \ref{xxlem1.9}(2).
The main goal of this section is to construct other
examples of (weak) bialgebras most of which are related
to finite quivers. We will see that, given a quiver $Q$,
there are different weak bialgebra structures on $\Bbbk Q$
such that the induced tensor products over $\Repr(Q)$ are
different from \eqref{E2.1.1}. As a consequence, there are
several different monoidal abelian structures on $\Repr(Q)$
generally. We will also see that there are monoidal
triangulated structures on derived categories
associated to noncommutative projective schemes.
The first example comes from \cite{HT2013}.

\begin{example}
\label{xxex7.1}
This example follows some ideas from \cite[Theorem 3.2]{HT2013}.
Let $Q$ be a quiver with $n$ vertices.
We label vertices of $Q$ as $1,2,\cdots,n$.
Suppose that $1$ is either a source or a sink, namely, $Q$
satisfies the following condition, either
\begin{enumerate}
\item[(1)]
there is no arrows from $1$ to $j$ for every $j$, or
\item[(2)]
there is no arrows from $j$ to $1$ for every $j$.
\end{enumerate}
Let $e_i$ be the idempotent corresponding
to the vertex $i$, and we use $p$ for a path of length
at least 1.

First we define a bialgebra structure on $\Bbbk Q$ by
$$\begin{aligned}
\varepsilon(e_1)&=1,\quad \Delta(e_1)=e_1\otimes e_1,\\
\varepsilon(e_i)&=0, \quad \Delta(e_i)=
\sum_{s<i} (e_i\otimes e_s+e_s\otimes e_i) +e_i\otimes e_i,\\
\varepsilon(p)&=0,\quad \Delta(p)=e_1\otimes p+p\otimes e_1\\
\end{aligned}
$$
for all $i>1$ and all paths $p$ of length at least 1.
It is routine to check that this defines a cocommutative
bialgebra structure on $\Bbbk Q$.

By the above definition, $\Delta(x)=e_1\otimes x+x\otimes e_1$
for all $x$ in the ideal $J$ generated by arrows of $Q$ (this
is also the graded Jacobson radical of $\Bbbk Q$). Let $I$ be any
sub-ideal of $J$. Then it is clear that $I$ is a bialgebra
ideal of $\Bbbk Q$. Therefore there is an induced bialgebra
structure on $\Bbbk Q/I$.
\end{example}

Let $Q$ is a finite acyclic quiver. Let $(\Delta,\varepsilon)$
be a coalgebra structure on $\Bbbk Q$. Suppose $|Q_0|=n$,
then $\Delta$ is called a {\it partitioning morphism}
(cf. \cite[p.460]{He2008a}) if
\begin{enumerate}
\item[(1)]
there are $E_1,\cdots,E_n$ which are subsets of
$E=\{(i,j)\mid1\leq i,j\leq n\}$,
\item[(2)]
$E_i\cap E_j=\emptyset$ if $i\neq j$, and
\item[(3)]
for every $1\leq k\leq n$,
$\Delta(e_k)= \sum\limits_{(i,j)\in E_k} e_i\otimes e_j$.
\end{enumerate}

Let $Q(i,j)$ be the set of paths from vertex $i$ to
vertex $j$, then:

\begin{proposition}\cite[Proposition 4]{He2008a}
\label{xxpro7.2}
Let $Q$ be a finite acyclic quiver. Suppose $\Bbbk Q$ has
a coalgebra structure $(\Bbbk Q, \Delta, \varepsilon)$. Then
$\Bbbk Q_0$ is a subcoalgebra of $\Bbbk Q$ and
$\Delta$ is a prealgebra map if and only if
\begin{enumerate}
\item[(1)]
$\Delta$ is a partitioning morphism,
\item[(2)]
$\Delta(\alpha_1\cdots\alpha_m)=\Delta(\alpha_1)\cdots\Delta(\alpha_m)$
where $\alpha_i\in Q_1$,
\item[(3)]
$\Delta(\alpha)\in
\bigoplus\limits_{(i,j)\in E_k,(i',j')\in E_l}
\Bbbk Q(i,i')\otimes \Bbbk Q(j,j')$
for any $\alpha:k\rightarrow l$.
\end{enumerate}
\end{proposition}
\begin{proof}
Note the fact that if $\Bbbk Q_0$ is a subcoalgebra of $\Bbbk Q$ and $\Delta$
is a prealgebra morphism, then $\Delta$ is a partitioning
morphism. The rest of the proof is similar to \cite[Proposition 4]{He2008a}, and
we omit it here.
\end{proof}

\begin{remark}
\label{xxrem7.3}
\begin{enumerate}
\item[(1)]
Following Proposition \ref{xxpro7.2}, our first step is to
understand all weak bialgebra structures on $\Bbbk^{\oplus n}$.
This is already a non-trivial task and we post it as a question.
$$\text{\it Can we classify all weak bialgebra structures
on $\Bbbk^{\oplus n}$?}$$
When $n=2$, see Lemma \ref{xxlem7.5} below.
\item[(2)]
There are algebras $A$ which do not admit any weak bialgebra 
structure. Let $A$ be the algebra $\Bbbk[x]/(x^n)$ for some $n$. 
Then $A$ admits a (weak) bialgebra structure if and only if 
$n=p^t$ where $p={\rm{char}}\; \Bbbk>0$ and $t\geq 1$. We give 
a sketch proof of one implication. Suppose that $A:=\Bbbk[x]/(x^n)$ 
is a weak bialgebra. Note that $A$ is local which implies that 
both the target and source counital subalgebras of $A$ are 
$\Bbbk$. As a consequence, $A$ is a bialgebra. So the augmentation 
ideal $J:=\ker \epsilon$ is the Jacobson radical of $A$. So the 
associated graded Hopf algebra ${\text{gr}}_{J} A$, which is 
isomorphic to $A$ as an algebra, is the restricted enveloping 
algebra of a restricted Lie algebra. Therefore the 
$\Bbbk$-dimension of $A$ is $p^t$ for some $t\geq 1$. The 
assertion follows.
\item[(3)]
Suppose ${\textrm{char}}\; \Bbbk=p>0$.
Let $A$ be the finite dimensional Hopf algebra
$$\Bbbk [x_1,\cdots,x_n]/(x_1^p,\cdots,x_n^p)$$
for some $n\geq 2$. The coalgebra structure of $A$ is
determined by
$$\Delta(x_i)=x_i\otimes 1+1\otimes x_i$$
for all $i$. Since $A$ is local, the only
brick object in $\mathcal{A}$ is the
trivial module $\Bbbk$. Therefore $\fpd(M)<\infty$
for every object in $\mathcal{A}$.
On the other hand, $A$ is wild when $n\geq 2$.
Therefore conditions (a) and (b) in Theorem \ref{xxthm0.4}
are not equivalent
if we remove the hereditary hypothesis.
\end{enumerate}
\end{remark}

\begin{definition}
\label{xxdef7.4}
Let $A$ be an algebra. Two (weak) bialgebra structures $(\Delta_1,
\varepsilon_1)$ and $(\Delta_2,\varepsilon_2)$ on $A$ are called
{\it equivalent} if there is an algebra automorphism
$\sigma$ of $A$ such that $\Delta_1 \sigma =(\sigma\otimes \sigma)
\Delta_2$ and $\varepsilon_1 \sigma=\varepsilon_2$.
\end{definition}

\begin{lemma}
\label{xxlem7.5}
Let $B=\Bbbk^{\oplus 2}=\Bbbk e_1\oplus \Bbbk e_2$. Then there are
five different weak bialgebra structures on $B$:
\begin{enumerate}
\item[(a)]
$\Delta(e_1)=e_1\otimes e_1, \Delta(e_2)=e_2\otimes e_2+e_1\otimes e_2
+e_2\otimes e_1$, $\varepsilon(e_1)=1$ and $\varepsilon(e_2)=0$.
\item[(b)]
$\Delta(e_1)=e_1\otimes e_1+e_2\otimes e_2, \Delta(e_2)=e_2\otimes e_1
+e_1\otimes e_2$, $\varepsilon(e_1)=1$ and $\varepsilon(e_2)=0$.
\item[(c)]
$\Delta(e_2)=e_2\otimes e_2, \Delta(e_1)=e_1\otimes e_1+e_1\otimes e_2
+e_2\otimes e_1$, $\varepsilon(e_1)=0$ and $\varepsilon(e_2)=1$.
\item[(d)]
$\Delta(e_2)=e_2\otimes e_2+e_1\otimes e_1, \Delta(e_1)=e_1\otimes e_2
+e_2\otimes e_1$, $\varepsilon(e_1)=0$ and $\varepsilon(e_2)=1$.
\item[(e)]
$\Delta(e_1)=e_1\otimes e_1, \Delta(e_2)=e_2\otimes e_2$,
$\varepsilon(e_1)=1$ and $\varepsilon(e_2)=1$.
\end{enumerate}
Note that {\rm{(a)}} and {\rm{(c)}} are equivalent bialgebra
structures {\rm{(}}and so are {\rm{(b)}} and {\rm{(d)}}{\rm{)}}.
The fifth one is a weak bialgebra, but not a bialgebra.
\end{lemma}

Note that (e) in the above lemma is the direct sum of two 
copies of trivial Hopf algebra $\Bbbk$. Consequently, 
it is a weak Hopf algebra. Other bialgebra algebras in
the above lemma are not (weak) Hopf algebras.

\begin{proof} Fix a (weak) bialgebra structure $(\Delta,\varepsilon)$
on $B$. Let $B_t$ and $B_s$ be target and source counital
subalgebras  of $B$, see \cite[Definition 2.2.3]{NV2002}.

Case 1: $\dim B_t=1$, then $B_t=B_s=\Bbbk 1_B$. In this case,
$B$ is a bialgebra. As a consequence,
$\varepsilon(e_1)+\varepsilon(e_2)=\varepsilon(e_1+e_2)=\varepsilon(1)=1$.
Since $e_i$ are idempotents, $\varepsilon(e_i)$ is 1 or 0.
First we assume that $\varepsilon(e_1)=1$ and $\varepsilon(e_2)=0$.
Write $\Delta(e_1)=\sum\limits_{i,j} a_{ij} e_i\otimes e_j$.
By the counital axiom, we obtain that $\Delta(e_1)=e_1\otimes e_1$ or
$\Delta(e_1)=e_1\otimes e_1+e_2\otimes e_2$. If
$\Delta(e_1)=e_1\otimes e_1$, we obtain case (a); if
$\Delta(e_1)=e_1\otimes e_1+e_2\otimes e_2$, we obtain
case (b). The other situation is $\varepsilon(e_1)=0$ and
$\varepsilon(e_2)=1$. By symmetric, we have (c) and (d).

Case 2: $\dim B_t=2$. Then $B_t=B_s=B$. By
\cite[Lemma 2.7]{BXYZZ2020}, $\Delta(e_1)=e_1\otimes e_1$
and $\Delta(e_2)=e_2\otimes e_2$. Then it is easy to check
that we obtain (e).
\end{proof}

\begin{lemma}
\label{xxlem7.6}
Let $A$ be a bialgebra and $J$ be its Jacobson radical. Suppose
that $J$ is nilpotent. If $B:=A/J\cong \Bbbk^{\oplus n}$ as an
algebra for some positive integer $n$, then $B$ is a quotient
bialgebra of $A$.
\end{lemma}

\begin{proof}
Let $\pi$ be the canonical quotient map from $A$ to $B$. It's
clear that $\pi$ is an algebra map. Consider the composition
of algebra maps:
$$A \xrightarrow{\Delta} A\otimes A
\xrightarrow{\pi \otimes \pi} B\otimes B.$$
Since $B\otimes B$ doesn't have nilpotent elements and $J$ is
nilpotent, the above algebra map from $A\to B\otimes B$ factors
through the quotient map $\pi$, that is, there exists a unique
algebra map $\Delta_B$ from $B\to B\otimes B$, such that the
following diagram commutes
$$\xymatrix{
A \ar[rr]^{\Delta}\ar[d]_{\pi} & & A\otimes A\ar[d]^{\pi\otimes \pi}\\
B \ar[rr]^{\Delta_{B}} & & B\otimes B.}$$
Furthermore, $(\Delta_B\otimes Id)\Delta_B$ and
$(Id\otimes \Delta_B)\Delta_B$ are the algebra maps induced
by algebra maps
$(\pi\otimes \pi\otimes \pi)(\Delta\otimes Id)\Delta$
and $(\pi\otimes \pi\otimes \pi)(Id\otimes \Delta)\Delta$
respectively from $A\to B\otimes B\otimes B$. Then $\Delta_B$ is
coassociative since $\Delta$ is coassociative.

Similarly, let $\varepsilon_B: B\to \Bbbk$ be the algebra map
induced by $\varepsilon:A\to \Bbbk$. It is not hard to verify
that $\varepsilon_B$ satisfies the counital axiom. Consequently,
$B$ is a quotient bialgebra of $A$ and $J$ is a bi-ideal.
\end{proof}

Now we are ready to classify (weak) bialgebras on a small quiver.

\begin{proposition}
\label{xxpro7.7}
Suppose $Q$ is the quiver with two vertices $\{1,2\}$ and
$w$ arrows from $1$ to $2$ with $w\geq 1$. Let $A$ be the
path algebra $\Bbbk Q$. Then there are 5 types of weak bialgebra
structures on $A$ up to equivalences.
\begin{enumerate}
\item[(a)]
$\Delta(e_1)=e_1\otimes e_1, \Delta(e_2)=e_2\otimes e_2+e_1\otimes e_2
+e_2\otimes e_1$, $\varepsilon(e_1)=1, \varepsilon(e_2)=0$, and for any
arrow $r$ from 1 to 2, $\Delta(r)= e_1\otimes r+r\otimes e_1$ and
$\varepsilon(r)=0$.
\item[(b)]
$\Delta(e_1)=e_1\otimes e_1+e_2\otimes e_2, \Delta(e_2)=e_2\otimes e_1
+e_1\otimes e_2$, $\varepsilon(e_1)=1, \varepsilon(e_2)=0$,
and for any arrow $r$ from 1 to 2, $\Delta(r)=r\otimes e_1
+e_1\otimes r$ and $\varepsilon(r)=0$.
\item[(c)]
$\Delta(e_2)=e_2\otimes e_2, \Delta(e_1)=e_1\otimes e_1+e_1\otimes e_2
+e_2\otimes e_1$, $\varepsilon(e_2)=1, \varepsilon(e_1)=0$, and for any
arrow $r$ from 1 to 2, $\Delta(r)= e_2\otimes r+r\otimes e_2$ and
$\varepsilon(r)=0$.
\item[(d)]
$\Delta(e_2)=e_1\otimes e_1+e_2\otimes e_2, \Delta(e_1)=e_2\otimes e_1
+e_1\otimes e_2$, $\varepsilon(e_2)=1, \varepsilon(e_1)=0$,
and for any arrow $r$ from 1 to 2, $\Delta(r)=r\otimes e_2
+e_2\otimes r$ and $\varepsilon(r)=0$.
\item[(e)]
$\Delta(e_i)=e_i\otimes e_i$, $\varepsilon(e_i)=1$ for $i=1,2$,
and the Jacobson radical $J$ of $A$ is a subcoalgebra of $A$.
\end{enumerate}
\end{proposition}

\begin{proof}
Let $J$ be the Jacobson radical of $A$, which is the ideal
generated by the arrows from $1$ to $2$. It is clear that
$J^2=0$ and $A/J\cong B$ where $B$ is as given in Lemma
\ref{xxlem7.5}.

We first consider bialgebra structures on $A$.

By Lemma \ref{xxlem7.6}, $A/J$ is a quotient bialgebra of
$A$ and $J$ is a bi-ideal of $A$. All bialgebra structures
on $B\cong A/J$ are classified in Lemma \ref{xxlem7.5}. We
will use this classification to analyze the bialgebra
structures on $A$.

Case 1: Suppose the bialgebra structure on $B$ is as in
Lemma \ref{xxlem7.5}(a). Lifting the bialgebra structure on
$B$ to $A$, we have
$$\Delta(e_1)=e_1\otimes e_1+e_1\otimes t_1+ e_2\otimes t_2
+t_3\otimes e_1+t_4\otimes e_2+ T,$$
$$\Delta(e_2)=e_1\otimes e_2+e_2\otimes e_1+e_2\otimes e_2-
e_1\otimes t_1-e_2\otimes t_2
-t_3\otimes e_1-t_4\otimes e_2- T,$$
where $T\in J\otimes J$ and $t_i\in J$ for $1\leq i\leq 4$, and
$$\varepsilon (e_1)=1, \quad \varepsilon(e_2)=0,
\quad \varepsilon(r)=0 {\text{ for all }} r\in J.$$
By counital axiom, we have $t_1=t_3=0$. By using the equation
$\Delta(e_1e_2)=0$, we have $t_2=t_4=0$.
In the bialgebra structure of $A$, we have, for every arrow
$r$ from $1$ to $2$,
$$\Delta(r)=e_1\otimes r+r\otimes e_1+ f(r)\otimes e_2+e_2\otimes
g(r)+w(r)$$
where $f(r),g(r)\in J$ and $w(r)\in J\otimes J$. Using the fact
that $r=re_1$, we obtain that $f(r)=g(r)=0$ for all $r$.

Pick any $\Bbbk$-basis of $J$, say $\{r_i\}$, we can write,
\begin{eqnarray*}
\Delta(e_1)&= &e_1\otimes e_1+\sum_{i,j} a_{ij} r_i \otimes r_j,\\
\Delta(e_2)&= &e_1\otimes e_2+e_2\otimes e_1
+e_2\otimes e_2-\sum_{i,j} a_{ij} r_i \otimes r_j,\\
\Delta(r_i)&= &e_1\otimes r_i+r_i\otimes e_1+\sum_{j,k} c^{jk}_i r_j\otimes r_k.
\end{eqnarray*}

Suppose $\deg (e_1)=\deg(e_2)=0$ and $\deg(r_i)=1$.
Let $\equiv$ denote $=$ modulo higher degree terms. Then the
coalgebra structure above can be written as
$$\begin{aligned}
\Delta(e_1)&\equiv  e_1\otimes e_1 \\
\Delta(e_2)&\equiv e_1\otimes e_2+e_2\otimes e_1+ e_2\otimes e_2\\
\Delta(r_i)&\equiv e_1\otimes r_i +r_i \otimes e_1.
\end{aligned}
$$
By \cite[Lemma 3.1]{HT2013}, if two different bialgebra structures
on $A$ both satisfy the above equations, then they are isomorphic.

Therefore, in this case, there exists a unique bialgebra structure
on $A$ up to isomorphism, that is,
\begin{eqnarray*}
\Delta(e_1)&= &e_1\otimes e_1,\\
\Delta(e_2)&= &e_1\otimes e_2+e_2\otimes e_1+e_2\otimes e_2,\\
\Delta(r_i)&= &e_1\otimes r_i+r_i\otimes e_1,
\end{eqnarray*}
which is exactly  (a).

Case 2: Suppose the bialgebra structure on $B$ is as in
Lemma \ref{xxlem7.5}(b). Lifting the bialgebra structure on
$B$ to $A$, we have
$$\Delta(e_1)=e_1\otimes e_1+ e_2\otimes e_2+ e_1\otimes t_1+
e_2\otimes t_2+ t_3\otimes e_1+t_4\otimes e_2+ T,$$
where $T\in J\otimes J$ and $t_i\in J$ for $1\leq i\leq 4$, and
$$\Delta(e_2)=e_1\otimes e_2+e_2\otimes e_1-e_1\otimes t_1-
e_2\otimes t_2-t_3\otimes e_1-t_4\otimes e_2- T,$$
$$\varepsilon (e_1)=1, \quad \varepsilon(e_2)=0,
\quad \varepsilon(r)=0 {\text{ for all }} r\in J.$$
By counital axiom, we have $t_1=t_3=0$.
By the fact $e_i$ is an idempotent, we have $T=0$.
In the bialgebra structure of $A$, for every arrow $r$ from
1 to 2, we have
$$\Delta(r)=e_1\otimes r+r\otimes e_1+ f(r)\otimes e_2+e_2\otimes
g(r)+w(r)$$
where $f(r),g(r)\in J$ and $w(r)\in J\otimes J$.
Using the fact
that $e_1r=0$, we obtain that $f(r)=g(r)=0$ for all $r$ and
$w(r)+r\otimes t_2+t_4\otimes r=0$.
Hence, for all $t\in J$,
$$\Delta(t)=e_1\otimes t+ t\otimes e_1- t\otimes t_2-
t_4\otimes t.$$
Moreover, the coassociative axiom, $(Id\otimes \Delta)\Delta(e_2)
=(\Delta\otimes Id)\Delta(e_2)$, implies $t_2=t_4=0.$
We obtain (b).

Case 3 and 4: When the bialgebra structure on $B$ is as in
Lemma \ref{xxlem7.5}(c) and (d), it's similarly to case 1 and case 2
respectively, and we obtain (c) and (d).

Next, we consider weak bialgebra, but not bialgebra, structures on $A$.

Let $A_t$ and $A_s$ be the target and source counital subalgebras.
By \cite[(2.1) and Proposition 2.4]{BNS1999},
$\dim A_t=\dim A_s$ and $A_s$ commutes with $A_t$.
If $\dim A_t=\dim A_s=1$, by \cite[Lemma 8.2]{N1998},
$A$ is a bialgebra since $\Delta(1)=1\otimes 1$, which is the case
we have just finished above. So $\dim A_t=\dim A_s\geq 2$.
Since $A_s$ is separable (hence semisimple), $A_s\cap J=\{0\}$.
Thus there is an injective map
$$A_s\longrightarrow  A \xrightarrow{\pi} B$$
which implies that $\dim A_t=\dim A_s=2$ and
that $\pi(A_s)=B$.

Now we claim that $A_t=A_s\cong B$. Since $\pi(A_s)=B$,
we can write $A_s=span\{1,e_1+p\}$ where $p\in J$. In this case,
$A_t=A_s$ since the space of elements which commute with $e_1+p$
is $A_s$ itself, and they both are weak bialgebras. Let $l$
denote the idempotent $e_1+p$. Assume that
$$\Delta(1)=a_1 1\otimes 1+a_2 l\otimes 1+a_3 1\otimes l+a_4 l\otimes l$$
where $a_i\in \Bbbk$.
By \cite[Equations (2.7a) and (2.7b)]{BNS1999}, $a_1+a_2=0$ and
$\Delta(l)=(a_2+a_4)l\otimes l$.
By $\Delta(1)=\Delta(1^2)$, $a_2=a_3$ and one of following equalities hold:
\begin{enumerate}
\item[(i)]
$\Delta(1)=l\otimes l$,
\item[(ii)]
$\Delta(1)=1\otimes 1- l\otimes 1- 1\otimes l+ l\otimes l$,
\item[(iii)]
$\Delta(1)=1\otimes 1- l\otimes 1- 1\otimes l+2 l\otimes l$.
\end{enumerate}
However, (i) implies that $l$ is a scalar multiple of $1$ and
(ii) implies that $1-l$ is a scalar multiple of $1$,
which both are impossible. So (iii) holds and
$\Delta(1)=(1-l)\otimes (1-l)+ l\otimes l$,
which means $A_t=A_s\cong B$ as weak bialgebra,
where the weak bialgebra structure on $B$
as in Lemma \ref{xxlem7.5}(e).

Re-write $l_1=l$ and $l_2=1-l$.  Then
$\Delta(l_1)=l_1\otimes l_1$ and $\Delta(l_2)=l_2\otimes l_2$.
Note that $A=A_t\oplus J$ as vector space.
Then for any arrow $r$ from 1 to 2, we have
$$\Delta(r)=f(r)\otimes l_1+g(r)\otimes l_2+
l_1\otimes p(r)+ l_2\otimes q(r)+w(r),$$
where $f(r),g(r),p(r),q(r)\in J$ and $w(r)\in J\otimes J$.
By $rl_1=r$ and $l_2r=r$, $f(r)=g(r)=p(r)=q(r)=0$ for all $r$.
That is $J$ is a subcoalgebra of $A$.
It's not hard to check any coalgebras structure over $J$
satisfy conditions in Definition \ref{xxdef1.7}.

Moreover, let $\sigma: (A,\Delta,\varepsilon)\to (A,\Delta',\varepsilon')$
via $\sigma(l_i)=e_i$ and $\sigma(r)=r$ for $r\in J$, where
$(A,\Delta',\varepsilon')$ is the weak bialgebra as in (e).
Then $\sigma$ is an algebra automorphism and $(\Delta,\varepsilon)$
is equivalent to $(\Delta',\varepsilon')$.
\end{proof}

We finish this section with examples related to both
commutative projective varieties and noncommutative
projective schemes in the sense of \cite{AZ1994}.

\begin{definition} \cite[p. 1230]{HP2011}
\label{xxdef7.8}
Let ${\mathbb X}$ be a smooth projective scheme.
\begin{enumerate}
\item[(1)]
A coherent sheaf ${\mathcal E}$ on ${\mathbb X}$ is called
{\it exceptional} if $\Hom_{\mathbb X}({\mathcal E},{\mathcal E})
\cong \Bbbk$ and $\Ext^i_{\mathbb X}({\mathcal E},{\mathcal E}) =0$
for every $i \geq  0$.
\item[(2)]
A sequence ${\mathcal E}_1, \cdots, {\mathcal E}_n$ of exceptional
sheaves is called an {\it exceptional sequence} if $\Ext^k_{\mathbb X}
({\mathcal E}_i,{\mathcal E}_j) = 0$ for all $k$ and for all $i > j$.
\item[(3)]
If an exceptional sequence generates $D^b(coh({\mathbb X}))$, then
it is called {\it full}.
\item[(4)]
If an exceptional sequence satisfies
$$\Ext^k_{\mathbb X}({\mathcal E}_i,{\mathcal E}_j) = 0$$
for all $k > 0$ and all $i, j$, then it is called a {\it strongly
exceptional sequence}.
\end{enumerate}
\end{definition}

The above concepts are extended to an arbitrary triangulated category
in \cite[Definition 4.1]{Mo2013}. The existence of a full (strongly)
exceptional sequence has been proved for many smooth projective schemes.
However, on Calabi-Yau varieties there are no exceptional sheaves.
When ${\mathbb X}$ has a full exceptional sequence
${\mathcal E}_1, \cdots, {\mathcal E}_n$, then there is a triangulated
equivalence
\begin{equation}
\label{E7.8.1}\tag{E7.8.1}
\RHom_{\mathbb{X}}(\oplus_{i=1}^n {\mathcal E}_i,-):\quad
D^b(coh({\mathbb X}))\cong D^b(\Modfd-A)
\end{equation}
where $A$ is the finite dimensional algebra
$\End_{\mathbb X}(\oplus_{i=1}^n {\mathcal E}_i)$, see
\cite[Theorem 4.2]{Mo2013} (or \cite[Theorem 3.1.7]{BVdB2003}).
By Example \ref{xxex5.3}(2), there is a canonical
monoidal triangulated structure on $D^b(coh({\mathbb X}))$
induced by $\otimes_{\mathbb{X}}$. Then we obtain a monoidal
triangulated structure on $D^b(\Mod_{f.d}-A)$ via
\eqref{E7.8.1}. By Example \ref{xxex5.3}(1), if $A$ is a
weak bialgebra, there is a (different) canonical monoidal
triangulated structure on $D^b(\Mod_{f.d}-A)$ (or
equivalently, on $D^b(coh({\mathbb X}))$). In short, there
are possibly many different monoidal triangulated structures
on a given triangulated category.

Next we give an explicit example related to noncommutative
projective schemes.

\begin{example}
\label{xxex7.9}
Let $T$ be a connected graded noetherian Koszul
Artin-Schelter regular algebra of global dimension at least
2. If $T$ is commutative, then $T$ is the polynomial ring
$\Bbbk[x_0,x_1,\cdots,x_n]$ for some $n\geq 1$. Let
$\mathbb{X}$ be the noncommutative projective scheme
associated to $T$ in the sense of \cite{AZ1994}. In
\cite{AZ1994} $\proj T$ denotes the category of coherent
sheaves on $\mathbb{X}$, but here we use $coh(\mathbb{X})$
instead. When $T$ is the commutative polynomial ring
$\Bbbk[x_0,x_1,\cdots,x_n]$, then $\mathbb{X}$ is the
commutative projective $n$-space $\mathbb{P}^n$. On the
other hand, there are many noetherian Koszul Artin-Schelter
regular algebras $T$ that are not commutative. Let $r$ be
the global dimension of $T$ and $\mathcal{O}$ be the
structure sheaf of $\mathbb{X}$. Then
$$\{\mathcal{O}(-(r-1)),\mathcal{O}(-(r-2)),
\cdots,\mathcal{O}(-1), \mathcal{O}\}$$
is a full strongly exceptional sequence for $\mathbb{X}$
in the sense of \cite[Definition 4.1]{Mo2013}.
By \eqref{E7.8.1} or \cite[Theorem 4.2]{Mo2013},
\begin{equation}
\label{E7.9.1}\tag{E7.9.1}
D^b(coh(\mathbb{X}))\cong D^b(A-\Modfd)
\end{equation}
where $A$ is the opposite ring of
$\End_{\mathbb X}(\oplus_{i=1}^n {\mathcal O}_i)$.
By \cite[Definition 4.6 and Theorem 4.7]{Mo2013},
$A$ is the opposite ring of the Beilinson algebra
(which is denoted by $R$ in \cite[Definition 4.6]{Mo2013}).
By the description in \cite[Definition 4.7]{MM2011}, the
Beilinson algebra is an upper triangular matrix with
diagonal entries being $\Bbbk$. Then $A$ can be written as
$\Bbbk Q/I$ where $Q$ is a quiver with $r$ vertices and
the number of arrows from vertex $i$ to vertex $j$ equals the
dimension of $T_{j-i}$. It is clear that $Q$ satisfies
condition (2) in Example \ref{xxex7.1}. By Example
\ref{xxex7.1}, there is a cocommutative bialgebra structure
on $A$. Similarly, vertex $r$ in $Q$ satisfies condition
(1) in Example \ref{xxex7.1}, which implies that there is
another cocommutative bialgebra structure on $A$. Via
\eqref{E7.9.1}, $D^b(coh(\mathbb{X}))$ has at least two
different monoidal triangulated structures induced by
two different bialgebra structures on $A$.

Now let $T$ be the polynomial ring $\Bbbk[x_0,x_1]$.
Then $\mathbb{X}=\mathbb{P}^1$ and
\begin{equation}
\notag
D^b(coh(\mathbb{P}^1))\cong D^b((B)^{op}-\Modfd)
\end{equation}
where $B$ is the Beilinson algebra associated to $T$.
By \cite[Definition 4.7]{MM2011},
$$B=\begin{pmatrix} \Bbbk & \Bbbk x+\Bbbk y\\
0& \Bbbk\end{pmatrix}.$$
It is clear that $B$ is the path algebra of the Kronecker
quiver given in Example \ref{xxex2.7}. In this case we have
two monoidal triangulated structures on
$D^b(coh(\mathbb{P}^1))$. One is the monoidal structure
induced by $\otimes_{\mathbb{P}^1}$, and the other comes
from the canonical weak bialgebra structure of $B=\Bbbk Q$
[Lemma \ref{xxlem2.1}(1)]. Together with two bialgebra
structures on $B$, see the above paragraph, we obtain four
different monoidal triangulated structures on
$D^b(coh(\mathbb{P}^1))$. To show these monoidal triangulated
structures are not equivalent, one need to use some arguments
in the proof of Lemma \ref{xxlem5.9} (details are omitted).
\end{example}

\section{Proof of Theorems \ref{xxthm0.8}}
\label{xxsec8}

It is important and interesting to calculate explicitly
$\fpd(M)$ of some objects $M$ in
a monoidal abelian (or triangulated) category. Generally
this is very difficult task and dependent on complicated
combinatorial structures of the brick sets. In this section
we will work out one example. Note that some non-essential
details are omitted.

A type $\mathbb{A}_n$ quiver is defined to be a quiver of
form \eqref{E0.7.1}:
\begin{equation}
\notag
\xymatrix{
1 \ar@{-}[r]^{\alpha_1}&2\ar@{-}[r]^{\alpha_2}
&\cdots\ar@{-}[r]^{\alpha_{i-1}}&i\ar@{-}[r]^{\alpha_i}
&\cdots\ar@{-}[r]^{\alpha_{n-1}}&n}
\end{equation}
where each arrow $\alpha_i$ is either $\longrightarrow$ or
$\longleftarrow$. For each $n\geq 3$, there are
more than one isomorphism classes of type $\mathbb{A}_n$
quivers with $n$ vertices, though we denote all of them
by $\mathbb{A}_n$. In this section we provide
fairly detailed computation of $\fpd(M)$ for every indecomposable
object in the monoidal abelian category $\Repr(\mathbb{A}_n)$.
Using Lemma \ref{xxlem4.11}, we obtain $\fpd(M)$ for every
indecomposable object $M$ in the monoidal triangulated
category $D^b(\Repr (\mathbb{A}_n))$. The result is
summarized in Theorem \ref{xxthm0.8}. Throughout this
section, the tensor product is defined as in \eqref{E2.1.1}.

First we try to understand brick sets in $\Repr(\mathbb{A}_n)$.
Recall that $M\{i,j\}$, for $i\leq j$, denotes the representation
of $\mathbb{A}_n$ defined by
\begin{align}
\notag
(M\{i,j\})_s&=\begin{cases} \Bbbk & i\leq s\leq j,\\
0 & {\text{otherwise,}}\end{cases} \\
\notag
(M\{i,j\})_{\alpha_s}&=\begin{cases} Id_{\Bbbk} & i\leq s<j,\\
0& {\text{otherwise}}.\end{cases}
\end{align}

We start with easy observations.

\begin{lemma}
\label{xxlem8.1}
If $\{M\{1,m\},M\{k,l\}\}$ is a brick set and $m\geq k \geq 3$,
then $\{M\{2,m\}$, $M\{k,l\}\}$ also is a brick set.
\end{lemma}

\begin{proof}
This is clear since $k\geq 3$.
\end{proof}

\begin{lemma}
\label{xxlem8.2}
For any $1\leq i<j\leq n$, $\{M\{1,i\},M\{1,j\}\}$ is not a
brick set.
\end{lemma}

\begin{proof}
There are two cases.

Case 1: $s(\alpha_i)=i$.
Let $f:M\{1,j\}\rightarrow M\{1,i\}$ be
$(f)_k=\begin{cases}
Id & k\leq i\\
0 & k>i
\end{cases}.$
Then it is clear that $f\in \Hom(M\{1,j\},M\{1,i\})$
and $\Hom(M\{1,j\},M\{1,i\})\neq 0$.

Case 2: $t(\alpha_i)=i$.
Let $g:M\{1,i\}\rightarrow M\{1,j\}$ be
$(g)_k=\begin{cases}
Id & k\leq i\\
0 & k>i
\end{cases}.$
Then $g\in \Hom(M\{1,i\},M\{1,j\})$ and $\Hom(M\{1,i\},M\{1,j\})\neq 0$.

Combining these two cases, one sees that $\{M\{1,i\},M\{1,j\}\}$
is not a brick set.
\end{proof}

In the above, we can replace $1$ by any positive integer
no more that $i$.

\begin{lemma}
\label{xxlem8.3}
Suppose $i\leq j\leq k$. Then one of spaces
$\Hom(M\{i,j\},M\{i,k\})$ and $\Hom(M\{i,k\},M\{i,j\})$ is
isomorphic to $\Bbbk $ while the other is zero.
\end{lemma}

\begin{proof}
An idea similar to the proof of Lemma \ref{xxlem8.2} shows that
one of spaces is nonzero and the other is zero. For the one that
is nonzero, it must be $\Bbbk$ by Lemma \ref{xxlem4.2}.
\end{proof}

\begin{lemma}
\label{xxlem8.4}
If $f:M\{i,k\}\rightarrow M\{i,l\}$ is a non-zero morphism and $k\neq l$,
then for any $j\leq i$, $\Hom(M\{i,l\}, M\{j,k\})=0$ and
$\Hom(M\{j,l\}, M\{i,k\})=0$.
\end{lemma}

\begin{proof}
Assume that $g:M\{i,l\}\rightarrow M\{j,k\}$ is non-zero morphism, then
it can induce a non-zero morphism $\hat g:M\{i,l\}\rightarrow M\{i,k\}$.
By Lemma \ref{xxlem8.3}, $f=0$ which contradicts the assumption.
Therefore, $\Hom(M\{i,l\}, M\{j,k\})=0$. Similarly,
$\Hom(M\{j,l\}, M\{i,k\})=0$.
\end{proof}

Next we define a binary relation, denoted by $\succ$, that
does not necessarily satisfy the usual axioms of an order.

\begin{definition}
{\label{xxdef8.5}}
For $N,N'\in \Repr(\mathbb{A}_n)$, we write $N\succ N'$ if
$\Hom(N,N')\cong \Bbbk$. Usually we only consider
indecomposable objects $N,N'$.
\end{definition}

Another easy observation, following from Lemma \ref{xxlem8.3}, is

\begin{lemma}
\label{xxlem8.6}
Let $I\subset \{1,2,\cdots,n\}$ and
$\mathcal{S}_I=\{X_i\mid X_i=M\{1,i\}, i\in I\}$.
Then  $(\mathcal{S}_I,\succ)$ is a totally ordered set.
Similarly, $\{Y_i\mid Y_i=M\{i,n\}, i\in I\}$
is a totally ordered set.
\end{lemma}

\begin{lemma}
{\label{xxlem8.7}}
Let $N=M\{i,j\}$, $N'=M\{k,l\}$ and $k\leq j< l$.
\begin{enumerate}
\item[(1)]
If $s(\alpha_j)=j$ and $i\leq k$,
then $\Hom(N',N\otimes N')\cong \Bbbk$ where
$N\otimes N'=M\{k,j\}$.
\item[(2)]
If $t(\alpha_j)=j$,
then for all $m\leq j$, $\Hom(N', M\{m,j\})=0.$
\end{enumerate}
\end{lemma}

\begin{proof}
(1) In this case, we have $i\leq k\leq j$. By definition,
$N'=M\{k,l\}$ and $N\otimes N'=M\{k,j\}$.
Let $f:N'\rightarrow N\otimes N$ be defined by
$(f)_s=\begin{cases}
Id & \mathrm{if} ~k\leq s\leq j\\
0 & \mathrm{otherwise}
\end{cases}.$
Then it is not hard to check $0\neq f\in \Hom(N', N\otimes N')$.

If $f'\in \Hom(N',N\otimes N')$, then there is a scalar $c\in \Bbbk$
such that $(f')_s=c Id$ for all $k\leq s\leq j$.
Then $f'=c f$ and $\Hom(N', N\otimes N')\cong \Bbbk$.

(2) Since $k\leq j<l$, $(N')_{j+1}=\Bbbk$. Let $f\in
\Hom(N',M\{m,j\})$. Then, for every $s>j$, $f_{s}=0$
as $(M\{m,j\})_s=0$. So we have
$$(f)_j (N')_{\alpha_j}=(M\{m,j\})_{\alpha_j} (f)_{j+1}=0.$$
Since $(N')_{\alpha_j}=Id_{\Bbbk}$, we obtain $(f)_j=0$.
Using a similar equation as above and induction, one sees
that $f_s=0$ for all $s<j$. Therefore $f=0$ as desired.
\end{proof}

For the rest of this section we use $\phi$ for a brick set in
$\Repr (\mathbb{A}_n)$. Given a brick set $\phi$ and an
indecomposable representation $M\{i,j\}$, we define three subsets
of $\phi$ according to $\{i,j\}$:
\begin{enumerate}
\item[(1)]
$\phi_i=\{N\in \phi \mid (N)_i\cong \Bbbk, (N)_j=0\},$
\item[(2)]
$\phi_j=\{N\in \phi \mid (N)_i=0, (N)_j\cong \Bbbk\},$
\item[(3)]
$\phi_{ij}=\{N\in \phi \mid (N)_i\cong \Bbbk, (N)_j\cong \Bbbk\}.$
\end{enumerate}
It is clear that $\phi$ contain the disjoint union of $\phi_i$,
$\phi_j$ and $\phi_{ij}$. Note that $\phi_l$, for $l$ being either
$i$ or $j$, can be divided into the following two parts:
\begin{equation*}
\begin{split}
\hat\phi_l  &= \{N\in \phi_l \mid M\{i, j\}\otimes N\succ M\{i, j\}\},\\
\tilde\phi_l &= \{N\in \phi_l \mid M\{i, j\}\succ M\{i, j\}\otimes N\}.
\end{split}
\end{equation*}

\begin{lemma}
\label{xxlem8.8}
Let $N$ be an object in $\phi$ that satisfies either
$M\{i,j\}\otimes N=0$ or $M\{i,j\}\otimes N=N$. Then
$$\rho(A(\phi,M\{i,j\}\otimes -))=
\max\{a, \quad \rho(A(\phi\setminus \{N\},M\{i,j\}\otimes -))\}$$
where $a=\begin{cases}
0 & \text{if } M\{i,j\}\otimes N=0,\\
1 & \text{if } M\{i,j\}\otimes N=N.
\end{cases}$
\end{lemma}

\begin{proof}
Write $\phi=\{N_1, \cdots, N_m\}$ where $N_1=N$. By the
hypothesis on $N$,
$$\dim \Hom(N_k,M\{i,j\}\otimes N)=0$$
for $2\leq k\leq m$. Hence, in the matrix
$A(\phi,M\{i,j\}\otimes -)$, $a_{k1}=0$ for all $k\geq 2$.
As a consequence,
$$\rho(A(\phi,M\{i,j\}\otimes -))=
\max\{a, \quad \rho(A(\phi\setminus \{N_1\},M\{i,j\}\otimes -))\}$$
where $a:=a_{11}$ is the $(1,1)$-entry in $A(\phi,M\{i,j\}\otimes -)$.
Clearly $a$ has the desired property.
\end{proof}

\begin{lemma}
{\label{xxlem8.9}}
Let $N\in \phi_i$ and $N'\in \phi_j$. Then
$\{N, M\{i,j\}\otimes N'\}$ and $\{M\{i,j\}\otimes N, N'\}$ are
brick sets.
\end{lemma}

\begin{proof}
Write $N$ as $M\{i',j'\}$. Then $i'\leq i$ and $j'<j$.
Similarly, $N'=M\{k, l\}$ for some $k>i$ and $l\geq j$,
and consequently, $M\{i,j\}\otimes N'=M\{k, j\}$. A version
of Lemma \ref{xxlem8.1} shows that $\{M\{i',j'\},M\{k,l\}\}$
being a brick set implies that $\{M\{i',j'\}, M\{k,j\}\}$ is a
brick set. Therefore $\{N, M\{i,j\}\otimes N'\}$ is a
brick set. A similar argument shows that
$\{M\{i,j\}\otimes N, N'\}$ is brick set.
\end{proof}

\begin{lemma}
\label{xxlem8.10}
Let $j$ be a positive integer no more than $n$. If
$j=n$ or $s(\alpha_j)=j$, then $A(\phi_j, M\{i,j\}\otimes -)$ is similar
to an upper triangular matrix in which all diagonal entries are 1.
\end{lemma}

\begin{proof}
If $j=n$, then $|\phi_j|=1$ and $A(\phi_j, M\{i,j\}\otimes -)=(1)_{1\times 1}$
by Lemma \ref{xxlem8.2}.

If $j<n$ and $s(\alpha_j)=j$, by Lemma \ref{xxlem8.6}, the set
$(\{M\{i,j\}\otimes N \mid N\in \phi_j\}, \succ)$ is a totally
ordered set. Let $|\phi_j|=m$ and we can label the objects in
$\phi_j$ so that
$$M\{i,j\}\otimes N_1\succ \cdots\succ M\{i,j\}\otimes N_{m}.$$
By Definition \ref{xxdef8.5},
\begin{equation}
\label{E8.10.1}\tag{E8.10.1}
\Hom(M\{i,j\}\otimes N_k, M\{i,j\}\otimes N_l)\cong
\begin{cases}
0 & \mathrm{if} ~l< k,\\
\Bbbk & \mathrm{if} ~l\geq k.
\end{cases}
\end{equation}
And, by Lemma \ref{xxlem8.7}(1),
\begin{equation}
\label{E8.10.2}\tag{E8.10.2}
\Hom(N_k, M\{i,j\}\otimes N_k)\cong \Bbbk.
\end{equation}

Combine \eqref{E8.10.1} and \eqref{E8.10.2}, then
$$\dim \Hom(N_k,M\{i,j\}\otimes N_l)=
\begin{cases}
0 & \mathrm{if} ~l<k,\\
1 & \mathrm{if} ~l\geq k.
\end{cases}$$
The assertion follows.
\end{proof}

The next theorem is Theorem \ref{xxthm0.8}(2).

\begin{theorem}
\label{xxthm8.11}
Let $Q$ be a quiver of type $\mathbb{A}_n$ given in
\eqref{E0.7.1} for some $n\geq 2$. Then the following
hold in $\Repr(Q)$:
$$\fpd(M\{i,j\})=\begin{cases}
1 & {\text{if $M\{i,j\}$ is a sink}},\\
\min\{i, n-j+1\} & {\text{if $M\{i,j\}$ is a source}},\\
1 &{\text{if $M\{i,j\}$ is a flow}}.
\end{cases}$$
\end{theorem}

\begin{proof} First we show that
\begin{equation}
\label{E8.11.1}\tag{E8.11.1}
\fpd(M\{i,j\})\geq \begin{cases}
1 & {\text{if $M\{i,j\}$ is a sink}},\\
\min\{i, n-j+1\} & {\text{if $M\{i,j\}$ is a source}},\\
1 &{\text{if $M\{i,j\}$ is a flow}}.
\end{cases}
\end{equation}
Let $\phi$ be the singleton consisting of $M\{i,j\}$.
It is clear that $A(\phi, M\{i,j\}\otimes -)$ is $(1)_{1\times 1}$.
Hence $\fpd(M\{i,j\})\geq 1$. Now suppose that
$M\{i,j\}$ is a source. Let $d=\min\{i, n-j+1\}$.
We construct a brick set with $d$ elements as follows.
By Lemma \ref{xxlem8.6},
$(\{M\{k,i\}\mid 1\leq k\leq i\},\succ)$ and
$(\{M\{j,m\}\mid j\leq m\leq n\}, \succ)$
are two totally ordered sets. We list elements
in these two sets as
\begin{equation}
\label{E8.11.2}\tag{E8.11.2}
{\text{
$M\{k_1,i\}\succ  \cdots\succ M\{k_i,i\}$ and
$M\{j,m_1\}\succ \cdots\succ M\{j,m_{n-j+1}\}$}}
\end{equation}
where $\{k_l\}_{l=1}^i$ and $\{m_l\}_{l=1}^{n-j+1}$
are distinct integers from $1$ to $i$ and
from $j$ to $n$ respectively.
Since $d=\min\{i,n-j+1\}$, we have a set of $d$ elements
$$\phi=\{M\{k_1,m_{n-j+1}\}, M\{k_2,m_{n-j}\}, \cdots,
M\{k_d, m_{n-j+2-d}\}\}.$$
We claim that $\phi$ is a brick set. If there is a nonzero
map from $M\{k_s,m_{n-j+2-s}\}$ to $M\{k_t,m_{n-j+2-t}\}$
for some $s<t$, then, when restricted to vertices
$\{j,j+1,\cdots,n\}$, we obtain a nonzero map from
$M\{j,m_{n-j+2-s}\}$ to $M\{j,m_{n-j+2-t}\}$. This
contradicts the second half of \eqref{E8.11.2}. Therefore
there is no nonzero morphism from $M\{k_s,m_{n-j+2-s}\}$ to
$M\{k_t,m_{n-j+2-t}\}$ for $s<t$. Similarly, there is
no nonzero morphism from $M\{k_t,m_{n-j+2-t}\}$ to
$M\{k_s,m_{n-j+2-s}\}$ for $s<t$, by using the first
half of \eqref{E8.11.2}. Thus we prove our claim.
Using this brick set, one see that every entry in
the matrix $A(\phi, M\{i,j\}\otimes -)$ is $1$,
consequently, $\rho(A(\phi, M\{i,j\}\otimes -))=d$.
Therefore $\fpd(M\{i,j\})\geq d$ if $M\{i,j\}$ is a
source. Combining with the inequality $\fpd(M\{i,j\})\geq 1$,
we obtain \eqref{E8.11.1}.

It remains to show the opposite inequality of
\eqref{E8.11.1}, or equivalently, to show that
\begin{equation}
\label{E8.11.3}\tag{E8.11.3}
\rho(A(\phi, M\{i,j\}\otimes -))\leq \begin{cases}
1 & {\text{if $M\{i,j\}$ is a sink}},\\
\min\{i, n-j+1\} & {\text{if $M\{i,j\}$ is a source}},\\
1 &{\text{if $M\{i,j\}$ is a flow}},
\end{cases}
\end{equation}
for every brick set $\phi$ in $\Repr(Q)$. We use induction
on the integer $|\phi|+n$. If $|\phi|+n$ is 1, nothing needs to
be proved. So we assume that $|\phi|+n\geq 2$.
If $|\phi|=1$, then $A(\phi, M\{i,j\}\otimes -)$
is either $(0)_{1\times 1}$ or $(1)_{1\times 1}$.
It is clear that the assertion holds. Now we assume that
$|\phi|\geq 2$. This forces that $n\geq 3$ (but we will
not use this fact directly). If there is an object $N\in
\phi$ such that either $M\{i,j\}\otimes N=0$ or
$M\{i,j\}\otimes N=N$, then \eqref{E8.11.3} follows
from Lemma \ref{xxlem8.8} and the induction
hypothesis.

For the rest of the proof we can assume that
$$N\not\cong M\{i,j\}\otimes N\neq 0$$
for every object $N\in \phi$. Note that the
above condition implies that $\phi$ is the disjoint
of $\phi_i$, $\phi_j$ and $\phi_{ij}$. Now it
suffices to consider $\phi$ satisfying the following
conditions:
\begin{enumerate}
\item[(*)] $\phi=\phi_i\cup \phi_j\cup \phi_{ij}$,
\item[(**)] for every $N\in \phi$, $M\{i,j\}\otimes N\not\cong N.$
\end{enumerate}

Let $w$ be the number of objects in $\phi$. Suppose that
$\phi_j$ is not empty. If there is an $N\in \phi_j$ such that
$N\otimes M\{i,j\} \succ M\{i,j\}$, we let $N_w$ be the object in
$\phi_j\cup \phi_{ij}$ such that $N_w\otimes M\{i,j\}$
is largest in the set
$$\{ N\otimes M\{i,j\}\mid N\in \phi_j \cup \phi_{ij}\}.$$
Such an object $N_w$ exists by a version of Lemma
\ref{xxlem8.6}. It is easy to see that $N_w\in \phi_j$.
By the choice of $N_w$, one can show that,
for every $N_k\in \phi_j \cup \phi_{ij}$ with $k\neq w$,
$$\Hom(N_k, N_w\otimes M\{i,j\})=0.$$
If $N_k\in \phi\setminus (\phi_j\cup\phi_{ij})$,
then, by Lemma \ref{xxlem8.9},
$$\Hom(N_k, N_w\otimes M\{i,j\})=0.$$
Therefore
$a_{kw}=0$ for all $k<w$ as an entry in the adjacency
matrix $A(\phi, M\{i,j\}\otimes -)$. As a consequence,
$$\rho(A(\phi, M\{i,j\}\otimes -))=\max\{1, \rho(A(\phi\setminus\{N_w\},
M\{i,j\}\otimes -))\}.$$
Assertion \eqref{E8.11.3} follows by induction hypothesis.
The other possibility is that for every $N\in \phi_j$ we
have $M\{i,j\}\succ N\otimes M\{i,j\}$. Now let $N_1$ be the object in
$\phi_j\cup \phi_{ij}$ such that $N_1\otimes M\{i,j\}$
is smallest in the set
$$\{ N\otimes M\{i,j\}\mid N\in \phi_j \cup \phi_{ij}\}.$$
Such an object $N_1$ exists by a version of Lemma
\ref{xxlem8.6}. It is easy to see that $N_1\in \phi_j$.
By the choice of $N_1$, one sees,
for every $N_k\in \phi_j \cup \phi_{ij}$ with $k\neq 1$,
$$\Hom(N_1, N_k\otimes M\{i,j\})=0.$$
If $N_k\in \phi\setminus (\phi_j\cup\phi_{ij})$,
then, by Lemma \ref{xxlem8.9}
$$\Hom(N_1, N_k\otimes M\{i,j\})=0.$$
Therefore
$a_{1k}=0$ for all $k>1$ as an entry in the adjacency
matrix $A(\phi, M\{i,j\}\otimes -)$. As a consequence,
$$\rho(A(\phi, M\{i,j\}\otimes -))=\max\{1, \rho(A(\phi\setminus\{N_1\},
M\{i,j\}\otimes -))\}.$$
Assertion \eqref{E8.11.3} follows by induction hypothesis.
Combining these two cases, we show that \eqref{E8.11.3}
holds by induction when $\phi_j$ is not empty.

Similarly, \eqref{E8.11.3} holds by induction when $\phi_i$
is not empty. The remaining case is when $\phi_i$ and
$\phi_j$ are empty, or
\begin{enumerate}
\item[$(^{\ast\ast\ast})$]
$\phi=\phi_{ij}$.
\end{enumerate}
We divide the rest of the proof into 5 small subcases.

Subcase 1: $t(\alpha_{i-1})=i$.
Pick any object in $\phi$, say $N$.
Suppose that $(N_1)_{i-1}\neq 0$. Then
$\Hom(N_1, M\{i,j\})=0$. Note that in
this case $M\{i,j\}=M\{i,j\}\otimes N$
for all $N\in \phi$. Therefore
$a_{1k}=0$ for all $k>1$ as an entry in the adjacency
matrix $A(\phi, M\{i,j\}\otimes -)$. As a consequence,
$$\rho(A(\phi, M\{i,j\}\otimes \rho(A(\phi\setminus\{N_1\},
M\{i,j\}\otimes -)).$$
Assertion \eqref{E8.11.3} follows by induction hypothesis.
Therefore, without loss of generality, we can assume that
$(N_1)_{i-1}=0$ for all $N_1\in \phi$. Now everything can be computed
in the subquiver quiver $Q\setminus\{1\}$. Then we reduce the number
of vertices from $n$ to $n-1$. Again the assertion follows from the
induction hypothesis.

Subcase 2: $t(\alpha_j)=j$. This is equivalent to Subcase 1 after one
relabels vertices of $Q$ by setting $i'=n+1-i$ for all
$1\leq i\leq n$.

Subcase 3: $i=1$. Since $\phi=\phi_{ij}$, by Lemma \ref{xxlem8.6},
$\phi$ consists of single object. As a consequence,
$A(\phi, M\{i,j\}\otimes -)$ is either $(0)_{1\times 1}$
or $(1)_{1\times 1}$. Then $\rho(A(\phi, M\{i,j\}\otimes -))
\leq 1$ and the assertion follows trivially.

Subcase 4: $j=n$. This is equivalent to Subcase 3 after one
re-labels vertices of $Q$ by setting $i'=n+1-i$ for all
$1\leq i\leq n$.

Subcase 5: Not cases 1-4, namely, $i>1$, $j<n$, $t(\alpha_{i-1})=i-1$
and $t(\alpha_j)=j+1$.
In this case $M\{i,j\}$ is a source. We list all objects in
$\phi=\phi_{ij}$ as
$$M\{i_1,j_1\}, \cdots, M\{i_w,j_w\}$$
where $1\leq i_s\leq i$ and $j\leq j_s\leq n$.
By Lemma \ref{xxlem8.3}, all $i_s$ are distinct.
The same holds true for $j_s$. Therefore $|\phi|=w\leq
d:= \min\{i,n-j+1\}$. Since every
entry of $A(\phi, M\{i,j\}\otimes -)$ is
at most 1, we obtain that
$\rho(A(\phi, M\{i,j\}\otimes -))\leq |\phi|\leq d$
as desired.

Combining \eqref{E8.11.1} with \eqref{E8.11.3},
we finish the proof.
\end{proof}

Note that, for $M,N\in \Repr(Q)$,
$$\Hom_{D^b(\Repr(Q))}(M[0], N[1])\cong
\Ext^1_{\Repr(Q)}(M,N).$$
For the rest of this section we use
$\Ext^1(M,N)$ instead of $\Ext^1_{\Repr(Q)}(M,N)$.
The {\it Euler characteristic} of two representations
$M$ and $N$ of $Q$ is defined to be
$$\langle\mathbf{dim} M,\mathbf{dim} N\rangle_{Q}
=\sum_{v\in Q_0}x_v y_v-\sum_{\alpha\in Q_1}
x_{s(\alpha)} y_{t(\alpha)}$$
where $\mathbf{dim}$ denotes the dimension vector and
$x_v=\dim((M)_v)$, $y_v=\dim ((N)_v)$ for any
$v\in Q_0$. By \cite[p.65]{GR1992}, we have
\begin{equation}
\label{E8.11.4}\tag{E8.11.4}
\dim \Hom(M,N)-\dim \Ext^1(M,N)=
\langle\mathbf{dim} M,\mathbf{dim} N\rangle_{Q}.
\end{equation}

One can verify the following.

\begin{lemma}
\label{xxlem8.12}
Assume $Q$ is of type $\mathbb{A}_n$. Let
$N=M\{i_1,j_1\},N'=M\{i_2,j_2\}$ and $i_1\leq i_2$.
\begin{enumerate}
\item[(1)]
If $j_1\leq i_2-2$, then $\Ext^1(N,N')=\Ext^1(N',N)=0.$
\item[(2)]
Suppose that $j_1=i_2-1$.
\begin{enumerate}
\item
If $s(\alpha_{j_1})=j_1$, then $\Ext^1(N,N')\cong\Bbbk$,  $\Ext^1(N',N)=0$.
\item
If $s(\alpha_{j_1})=i_2$, then $\Ext^1(N,N')=0$,  $\Ext^1(N',N)\cong\Bbbk$.
\end{enumerate}
\item[(3)]
Suppose either $i_1<i_2\leq j_1<j_2$ or $i_1<i_2\leq j_2<j_1$.
\begin{enumerate}
\item
If $\Hom(N,N')\cong\Bbbk$, then $\Ext^1(N,N')=0$, $\Ext^1(N',N)\cong\Bbbk$.
\item
If $\Hom(N',N)\cong\Bbbk$, then $\Ext^1(N,N')\cong\Bbbk$, $\Ext^1(N',N)=0$.
\item
If $\{N,N'\}$ is a brick set, then $\Ext^1(N,N')=\Ext^1(N',N)=0$.
\end{enumerate}
\item[(4)]
If $i_1=i_2$ or $j_1=j_2$, then $\Ext^1(M,N)=\Ext^1(N,M)=0$.
\end{enumerate}
\end{lemma}

\begin{proof}
When $j_1\leq i_2-2$, it is easy to see
\[
\dim \Hom(N,N')=\dim \Hom(N',N)=0
\]
and
\[
\langle\mathbf{dim} N,\mathbf{dim} N'\rangle_{Q}=
\langle\mathbf{dim} N',\mathbf{dim} N\rangle_{Q}=0.
\]
Therefore, $\Ext^1(N,N')=\Ext^1(N',N)=0.$

As for (2), (3) and (4), the proofs are similar and we omit them here.
\end{proof}

A direct corollary of Lemma \ref{xxlem8.12} is

\begin{corollary}
\label{xxcor8.13}
Assume $Q$ is of type $\mathbb{A}_n$.
If $\Hom(M\{i_1,j_1\},M\{i_2,j_2\})\cong \Bbbk$,
then \[\Ext^1(M\{i_1,j_1\},M\{i_2,j_2\})=0.\]
\end{corollary}

Any brick set $\phi$ in $\Repr(Q)$ is also a brick set in $D^b(\Repr(Q))$.
In the next lemma we are working with the category $D^b(\Repr(Q))$
and $\phi$ (respectively, $\phi_i$ and $\phi_j$) still denotes a brick
set in $\Repr(Q)$.

\begin{lemma}
\label{xxlem8.14}
Retain the notation above. Then $A(\phi_i, M\{i,j\}[1]\otimes-)$
is similar to a strictly lower triangular matrix.
\end{lemma}

\begin{proof}
By Lemma \ref{xxlem8.6}, $(\{M\{i,j\}\otimes N\mid N\in \phi_i\},\succ)$
is a totally ordered set, which can be listed as
\begin{equation}
\label{E8.14.1}\tag{E8.14.1}
M\{i, j_1\}\succ M\{i,j_2\}\succ \cdots M\{i,j_{|\phi|_i}\}.
\end{equation}
When we compute the adjacency matrix $A(\phi_i, M\{i,j\}[1]\otimes-)$,
we order elements in $\phi_i$ according to \eqref{E8.14.1}.
For any two objects $M\{i_{s_1}, j_{s_1}\},M\{i_{s_2}, j_{s_2}\}$ in
$\phi_i$ with $s_1< s_2$, we have $M\{i, j_{s_1}\}\succ M\{i, j_{s_2}\}$.
An easy analysis shows that either
$\{M\{i_{s_1}, j_{s_1}\}, M\{i, j_{s_2}\}\}$ is a brick set or
$\Hom(M\{i_{s_1}, j_{s_1}\}, M\{i, j_{s_2}\})\cong \Bbbk$.
By Lemma \ref{xxlem8.12}(3),
$$\Ext^1 (M\{i_{s_1}, j_{s_1}\}, M\{i,j\}\otimes M\{i_{s_2}, j_{s_2}\})=0.$$
By Lemma \ref{xxlem8.12}(4), we have
$$\Ext^1 (M\{i_{s_1}, j_{s_1}\}, M\{i,j\}\otimes M\{i_{s_1}, j_{s_1}\})=
\Ext^1 (M\{i_{s_1}, j_{s_1}\}, M\{i, j_{s_1}\})=0.$$
As a consequence, $A(\phi_i,M\{i,j\}[1]\otimes -)$
is a strictly lower triangular matrix.
\end{proof}

\begin{lemma}
\label{xxlem8.15}
Let $N\in \hat\phi_j$, $N'\in \phi_i\cup \phi_{ij}$ and
$N''\in \tilde\phi_j$. then
$$\Ext^1 (N, M\{i,j\}\otimes N')=\Ext^1 (N', M\{i,j\}\otimes N'')=
\Ext^1 (N, M\{i,j\}\otimes N'')=0.$$
\end{lemma}

\begin{proof}
Similar to the proof of Lemma \ref{xxlem8.12}, we only prove the first
equation and leave out the proof of the last two equations.

Write $N=M\{i_1, j_1\}$ and $N'=M\{i_2, j_2\}$. By definition,
$\hat\phi_j$ is nonempty. This implies that $s(\alpha_{i_1-1})=i_1-1$.

First we suppose that $N'\in \phi_i$. If $j_2\neq i_1-1$, then,
by Lemmas \ref{xxlem8.1} and \ref{xxlem8.12}(1,3c),
$\Ext^1 (N, M\{i,j\}\otimes N')=0$. If $j_2=i_1-1$, then
$\Ext^1 (N, M\{i,j\}\otimes N')=0$ by Lemma \ref{xxlem8.12}(2).
Therefore, $\Ext^1 (N, M\{i,j\}\otimes N')=0$ always holds for
$N'\in \phi_i$.

Next we suppose that $N'\in \phi_{ij}$. Then either $\{N,M\{i,j\}\}$
is a brick set or $\Hom(N,M\{i,j\})\cong \Bbbk$. By Lemma
\ref{xxlem8.12}(3),  $\Ext^1 (N, M\{i,j\}\otimes N')=0$ since
$M\{i,j\}\otimes N'=M\{i,j\}$.

The assertion follows.
\end{proof}

Now, we prove Theorem \ref{xxthm0.8}(3).

\begin{theorem}
\label{xxthm8.16}
Let $Q$ be a quiver of type $\mathbb{A}_n$ given in
\eqref{E0.7.1} for some $n\geq 2$. Then
$$\fpd(M\{i,j\}[1])=\begin{cases}
\min\{i-1, n-j\} & {\text{if $M\{i,j\}$ is a sink}},\\
1 & {\text{if $M\{i,j\}$ is a source}},\\
1 &{\text{if $M\{i,j\}$ is a flow}}.
\end{cases}$$
\end{theorem}

\begin{proof} Since this is a statement about the derived
category $D^b(\Repr(Q))$, we need to consider all brick
objects in this derived category. However, by the
argument given in the proof of Lemma\ref{xxlem4.11} (2), we
only need to consider brick sets of the form
$$\phi=\{N_1, \cdots, N_m\mid N_s\in \Repr(Q)\}$$
which consists of objects in the abelian category $\Repr(Q)$.

The rest of the proof is somewhat similar to the proof of
Theorem \ref{xxthm8.11}.

If there exists an object $N_1=M\{i_0, j_0\}\in \phi$ satisfying
$M\{i,j\}\otimes N_1\cong N_1$, by Lemma \ref{xxlem8.2}, there
exist at most one object $N_2=M\{i_1,j_1\}\in \phi$ satisfying $j_1=i_{0} - 1$
and at most one object $N_3=M\{i_2,j_2\}\in \phi$ satisfying $i_2=j_0+1$.
Then, by Lemmas \ref{xxlem8.1} and \ref{xxlem8.12},
in the first column and the first row of $A(\phi, M\{i, j\}[1]\otimes-)$,
all entries are zero except for $a_{12}, a_{21}, a_{13}, a_{31}$, and
$a_{12} a_{21}=a_{13} a_{31}=0$.
No matter which case is, we always have
$$\rho(A(\phi,M\{i,j\}[1]\otimes -))
=\rho(A(\phi\setminus \{N_1\},M\{i,j\}[1]\otimes -)).$$
Also, if there is an object $N\in \phi$ satisfying
$M\{i,j\}\otimes N=0$, we also have
$$\rho(A(\phi,M\{i,j\}[1]\otimes -))
=\rho(A(\phi\setminus \{N\},M\{i,j\}[1]\otimes -)).$$

Similar to the proof of Theorem \ref{xxthm8.11}, it suffices
to consider the brick set $\phi$ satisfying the following conditions:
\begin{enumerate}
\item[(*)] $\phi=\phi_i\cup \phi_j\cup \phi_{ij}$,
\item[(**)] for every $N\in \phi$, $M\{i,j\}\otimes N\not\cong N.$
\end{enumerate}

By Lemma \ref{xxlem8.14}, if we re-arrange objects in
$\phi$ as $\hat\phi_j$, $\hat\phi_i$, $\phi_{ij}$, $\tilde\phi_i$
and $\tilde\phi_j$, then $A(\phi, M\{i, j\}[1]\otimes-)$ is
a block lower triangular matrix.
By Lemma \ref{xxlem8.15},
$$\rho(\phi, M\{i,j\}[1]\otimes-)=\rho(\phi_{ij}, M\{i,j\}[1]\otimes-).$$
Therefore, for the rest we consider the brick set $\phi$
satisfying $\phi_{ij}=\phi$.

We divide the rest of the proof into 3 small cases.

Case 1: $M\{i,j\}$ is a source. In this case, for any
$N\in \phi$, $\Hom(N, M\{i,j\})\cong \Bbbk$. Then by Lemma
\ref{xxlem8.12}(3), $\Ext^1 (N, M\{i,j\})=0$.
As a consequence, the adjacency matrix
$A(\phi,M\{i,j\}[1]\otimes -)$ is a zero matrix.
Therefore, in this case, $\fpd(M\{i,j\}[1])=0$.

Case 2: $M\{i,j\}$ is a flow, without loss of generality, assume that
$\alpha_{i-1}=\alpha_j=\longleftarrow$.
For any $N=M\{i_1,j_1\}\in \phi$,
if $i_1=i$, by Lemma \ref{xxlem8.12}(4),
$\Ext^1 (N,M\{i,j\})=0$.
If $i_1<i$, either $\Hom(N, M\{i,j\})\cong \Bbbk$ or
$\{N, M\{i,j\}\}$ is a brick set, then by Lemma \ref{xxlem8.12}(3),
$\Ext^1 (N,M\{i,j\})=0.$
As a consequence, the adjacency matrix
$A(\phi,M\{i,j\}[1]\otimes -)$ is a zero matrix.
Therefore, in this case, $\fpd(M\{i,j\}[1])=0$.

Case 3: $M\{i,j\}$ is a sink.
In this case, for any $N=M\{i_1,j_1\}\in \phi$,
if $i_1=i$, by Lemma \ref{xxlem8.12}(4),
$\Ext^1 (N,M\{i,j\})=0$.
If $j_1=j$, by Lemma \ref{xxlem8.12}(4),
$\Ext^1 (N,M\{i,j\})=0$.
Therefore, since $M\{i,j\}\otimes N=M\{i,j\}$,
we can assume that $i_1<i$ and $j_1>j$.
Now, it's easy to see $\Ext^1 (N,M\{i,j\})\cong \Bbbk$
by Lemma  \ref{xxlem8.12}(3).
As a consequence, all entries in the adjacency matrix
$A(\phi,M\{i,j\}[1]\otimes -)$ are 1 and
$\rho(A(\phi,M\{i,j\}[1]\otimes -))=|\phi_{ij}|\leq \min\{i-1,n-j\}$.

On the other hand, by Lemma \ref{xxlem8.6},
$(\{M\{k,i\}\mid 1\leq k\leq i-1\},\succ)$ and
$(\{M\{j,m\}\mid j+1\leq m\leq n\}, \succ)$
are two totally ordered sets. We list elements
in these two sets as
\begin{equation*}
{\text{
$M\{k_1,i\}\succ  \cdots\succ M\{k_{i-1},i\}$ and
$M\{j,m_1\}\succ \cdots\succ M\{j,m_{n-j}\}$}}
\end{equation*}
where $\{k_l\}_{l=1}^{i-1}$ and $\{m_l\}_{l=1}^{n-j}$
are distinct integers from $1$ to $i-1$ and
from $j+1$ to $n$ respectively.
Let $d=\min\{i-1,n-j\}$, then we have a set of $d$ elements
$$\phi=\{M\{k_1,m_{n-j}\}, M\{k_2,m_{n-j-1}\}, \cdots,
M\{k_d, m_{n-j+1-d}\}\}$$ which is a brick set.
Using this brick set, one see every entry in the
matrix $A(\phi, M\{i, j\}[1]\otimes-)$
is 1 by Lemma \ref{xxlem8.12}, consequently,
$\rho(A(\phi, M\{i, j\}[1]\otimes-))=d$.
Hence, in this case, $\fpd(M\{i,j\}[1])=\min\{i-1,n-j\}$.
\end{proof}

\begin{proof}[Proof of Theorem \ref{xxthm0.8}]:
(1) This follows from Lemma \ref{xxlem4.11}(1).

(2) This follows from Lemma \ref{xxlem4.11}(2)
and Theorem \ref{xxthm8.11}.

(3) This follows from Theorem \ref{xxthm8.16}.
\end{proof}

\subsection*{Acknowledgments}
The authors thank the referee for his/her very careful reading and 
valuable comments and thank Professors Jianmin Chen and 
Xiao-Wu Chen for many useful conversations on the subject. 
J.J. Zhang was partially supported by the US National Science 
Foundation (Grant Nos. DMS-1700825 and DMS-2001015). J.-H. Zhou 
was partially supported by Fudan University Exchange Program 
Scholarship for Doctoral Students (Grant No. 2018024).


\end{document}